\documentclass[11pt]{article}
\usepackage{amsthm,amsmath,amssymb}
\usepackage{natbib}
\usepackage{multirow}
\usepackage[pdftex]{graphicx}
\usepackage{subfigure}
\usepackage{makecell}
\usepackage{booktabs}
\usepackage{array}
\usepackage{fullpage}
\usepackage{url}
\usepackage{algorithm}
\usepackage{algorithmic}
\usepackage{bm}
\usepackage{smile}
\usepackage{mathtools}
\usepackage{wrapfig}
\usepackage{lipsum}
\usepackage{mathrsfs}
\usepackage{dsfont}

\usepackage{natbib}
\usepackage{multirow}

\usepackage{appendix}

\newcolumntype {Q}{>{$\displaystyle}l<{$}}
\newcolumntype {A}{>{$}c <{$}}

\newcommand{\bel}{\begin{eqnarray}\label}
\newcommand{\eel}{\end{eqnarray}}
\newcommand{\bes}{\begin{eqnarray*}}
\newcommand{\ees}{\end{eqnarray*}}

\usepackage[usenames,dvipsnames,svgnames,table]{xcolor}
\usepackage[colorlinks=true,
            linkcolor=blue,
            urlcolor=blue,
            citecolor=blue]{hyperref}
            
\numberwithin{equation}{section}
\numberwithin{theorem}{section}
\numberwithin{corollary}{section}
\numberwithin{asmp}{section}
\numberwithin{definition}{section}

\def\ind{{\mathds 1}}

\def\cov{{\rm Cov}}
\def\var{{\rm Var}}
\def\reals{{\mathbb R}}

\def\as{{\rm a.s.}}
\def\card{{\rm card}}
\def\ken{{\rm Ken}}
\def\ap{{\rm AP}}
\def\sgn{{\rm sgn}}

\def\bi{\boldsymbol{i}}
\def\bj{\boldsymbol{j}}
\def\bk{\boldsymbol{k}}
\def\br{\boldsymbol{r}}
\def\bs{\boldsymbol{s}}

\def\Xbi{X_{\boldsymbol{i}}}
\def\Xbj{X_{\boldsymbol{j}}}
\def\Xbr{X_{\boldsymbol{r}}}
\def\Xbs{X_{\boldsymbol{s}}}

\begin{document}

\setlength{\abovedisplayskip}{5pt}
\setlength{\belowdisplayskip}{5pt}
\setlength{\abovedisplayshortskip}{5pt}
\setlength{\belowdisplayshortskip}{5pt}

\title{\LARGE On inference validity of weighted U-statistics under data heterogeneity}

\author{Fang Han\thanks{Department of Statistics, University of Washington, Seattle, WA 98195; e-mail: {\tt fanghan@uw.edu}}~~~and~~Tianchen Qian\thanks{Department of Statistics, Harvard University, Cambridge, MA 02138; e-mail: {\tt qiantianchen@fas.harvard.edu}}}

\date{}

\maketitle

\begin{abstract} 
Motivated by challenges on studying a new correlation measurement being popularized in evaluating online ranking algorithms' performance, this manuscript explores the validity of uncertainty assessment for weighted U-statistics. Without any commonly adopted assumption, we verify Efron's bootstrap and a new resampling procedure's inference validity. Specifically, in its full generality, our theory allows both kernels and weights asymmetric and data points not identically distributed, which are all new issues that historically have not been addressed. For achieving strict generalization, for example, we have to carefully control the order of the ``degenerate" term in U-statistics which are no longer degenerate under the empirical measure for non-i.i.d. data. 
Our result applies to the motivating task, giving the region at which solid statistical inference can be made. 
\end{abstract}




\section{Introduction} \label{sec:introduction}

This manuscript studies the following general weighted U-statistic of degree $m$:
\begin{align}\label{eq:mainU}
U_n = \frac{(n-m)!}{n!}\sum_{\substack{1\leq i_1,i_2,\ldots,i_m\leq n:\\ i_j\ne i_k~{\rm if}~j\ne k}}a_n(i_1,\ldots,i_m)h_n(X_{i_1},\ldots,X_{i_m}).
\end{align}
Here we assume $X_1,\ldots,X_n$ are independent but not necessarily identically distributed random variables, taking values in a measurable space $(\cX,\cB_{\cX})$ \citep{korolyuk2013theory}. 
The weight function $a_n(\cdot)$ and kernel function $h_n(\cdot)$ are both possibly asymmetric, and they are both allowed to be sample size dependent.

Our study on weighted U-statistics is motivated from the following new correlation measurement popularized in the information retrieval area \citep{yilmaz2008new}. It is formulated as a weighted U-statistic  of asymmetric kernels and weights:
\begin{align}\label{eq:tauAP}
\tau^{\rm AP}:=\frac{2}{n-1}\sum_{i=2}^n \frac{\sum_{j=1}^{i-1}\ind(X_j>X_i)}{i-1}-1.
\end{align}
Here $\ind(\cdot)$ represents the indicator function and $X_1,\ldots,X_n$ are specified to be real-valued. For this specific example, $X_1,\ldots, X_n$ correspond to the scores the ranking machine gives for each online page, aligned by the rankings of human labels. The data points $X_1,\ldots,X_n$ are usually modeled by a location-scale model, and are usually non-i.i.d.. The statistic in \eqref{eq:tauAP}, named average-precision (AP) correlation, aims to evaluate the performance of any given online ranking algorithm 
by calculating a reweighted rank correlation measurement between the algorithm's rankings, while ``giving more weights to the errors at high rankings". For the AP correlation, it is desirable to derive confidence intervals for solid inference. 

Obviously, $\tau^{\rm AP}$ is an extension to the Kendall's tau statistic:
\begin{align}\label{eq:tau}
\tau^{\rm Ken}:=\frac{2}{n(n-1)}\sum_{i\ne j}\Big\{\ind(X_i>X_j)\ind(i<j)+\ind(X_i<X_j)\ind(i>j)\Big\}-1.
\end{align}
Compared to $\tau^{\rm Ken}$, the analysis of $\tau^{\rm AP}$ is much more involved, but naturally falls into the application regime of our theory.

The analysis of unweighted U-statistics (i.e., $a_n(\cdot)\equiv 1$) has a long history. There has been a vast literature on evaluating their asymptotic behaviors since the seminal paper of \cite{hoeffding1948ustat}. Specifically, regarding the simple independent and identically distributed (i.i.d.) setting, inference results have been summarized in \cite{lee1990u}, \cite{serfling2009approximation}, and \cite{korolyuk2013theory}. For extensions, \cite{lee1990u} proved the asymptotic normality under a Lyapunov-type non-i.i.d. condition. \cite{yoshihara1976limiting} and \cite{dehling2010central} derived central limit theorem and (block) bootstrap inference validity for stationary weakly dependent time series. \cite{csorgHo2013asymptotics} proved the $m$-out-of-$n$ bootstrap inference validity. 

Weighted U-statistic is comparably less touched in the literature. Here, under the i.i.d. setting, \cite{shapiro1979asymptotic} and \cite{oneil1993weighted} conducted asymptotic analysis for weighted U-statistics of degree two. \cite{major1994asymptotic} and \cite{rifi2000asymptotic} made extensions to weighted U-statistics of degree $m\geq 2$, with focus on the degenerate cases. \cite{hsing2004weighted} relaxed the independence assumption, proving the asymptotic normality for a wide range of stationary stochastic processes. Recently, \cite{zhou2014inference} generalized the results in \cite{hsing2004weighted}, proving central and noncentral limit theorems for a class of nonstationary time series. 

Despite the above substantial advances, 
i.i.d. or stationary assumption is commonly posed, especially for proving Efron's bootstrap inference validity.   A notable exception is \cite{zhou2014inference}, who established central limit theorem for nonstationary time series. However, bootstrap inference is not discussed, and the regularity conditions therein are too strong to include statistics like $\tau^{\rm AP}$. In addition, the kernels and weights are required to be symmetric. 

Motivated from our study on the AP correlation, this manuscript aims to fill the above gaps. In particular, we build unified theory for analyzing nondegenerate weighted U-statistics, namely, establishing sufficient conditions for their asymptotic normality and bootstrap inference validity. Both Efron's bootstrap and a new resampling procedure stemmed from \cite{politis1994large} and \cite{bickel1997resampling} are considered. For this, we waive the i.i.d. assumptions, allowing researchers to analyze statistics like $\tau^{\rm AP}$ in practical settings. In addition, our analysis allows both the kernels and weights to be asymmetric.

\subsection{Other related work}

Our results are very related to bootstrap inference under data heterogeneity. In \cite{liu1988bootstrap}, Regina Liu pioneered the study on Efron's bootstrap inference validity for non-i.i.d. models. Her results showed that bootstrap is robust to these specific  non-i.i.d. settings with common locations (means). However, bootstrap is very sensitive to mean differences. The inference validity is captured by a function of $\{\mu_i:=E X_i\}_{i=1}^n$, which she called ``heterogeneity factors" \citep{liu1988bootstrap,liu1995using}. For example, for the sample mean, at the worse case, the distance between the largest and smallest means needs to shrink to zero as $n\to \infty$ for bootstrap consistency. \cite{mammen2012bootstrap} summarized the existing results, providing necessary and sufficient conditions of bootstrap validity for the sample-mean-type statistics under non-i.i.d. settings. 

Politis and Romano's subsampling \citep{politis1999subsampling} and many other resampling schemes \citep{bickel1997resampling} are appealing alternatives to Efron's bootstrap. They are designed to correct the bootstrap inference inconsistency problem in many different settings, where the data could be, for example, dependent or heavy-tailed. In this manuscript, we examine a new resampling procedure's inference validity for weighted U-statistics. 



\subsection{Notation}

\label{sec:notation}

Let $\mathbb{R}$ be the set of real numbers, and $\mathbb{Z}$ be the set of integers. For a positive integer $n$, we write $[n]=\{a\in\mathbb{Z}:1\leq a\leq n\}$. For any set $\cA$, let $\card(\cA)$ represent the cardinality of $\cA$. Let $\stackrel{d}{\to}$ denote ``convergence in distribution", and $\stackrel{P}{\to}$ denote ``convergence in probability". Let ``$\as$" be the abbreviation of ``almost surely''. 
Let $\Phi(t)$ be the cumulative distribution function of the standard Gaussian. 
For two positive integers $m<n$, define 
\[
\binom{n}{m}=\frac{n!}{(n-m)!m!},
\]
where $n!$ represents the factorial of $n$. Let $C$ be a generic absolute positive constant, whose actual value may vary at different locations. For any two real sequences $\{a_n\}$ and $\{b_n\}$, we write $a_n\lesssim b_n$, or equivalently $b_n\gtrsim a_n$, if there exists an absolute constant $C$ such that $|a_n|\leq C|b_n|$ for all sufficiently large $n$. We write $a_n\asymp b_n$ if both $a_n\lesssim b_n$ and $a_n\gtrsim b_n$ hold. We write $a_n\gnsim b_n$, or equivalently $b_n\lnsim a_n$, if $a_n\gtrsim b_n$ holds, but $a_n\lesssim b_n$ does not. We write $a_n=O(b_n)$ if $a_n\lesssim b_n$, and $a_n=o(b_n)$ if $a_n=O(b_n)$ and $b_n\ne O(a_n)$. We write $a_n=O_{P}(b_n)$ if $a_n/b_n$ is stochastically bounded, that is, for any $\epsilon > 0$, there exists a finite $M>0$ and a finite $N>0$ such that $P(|a_n /b_n| > M) < \epsilon$ for all $n>N$. We write $a_n=o_{P}(b_n)$ if for any $\epsilon > 0$, $\lim_{n\to\infty} P(|a_n / b_n| \geq \epsilon) = 0$.

\subsection{Structure of the manuscript}

The rest of the manuscript is organized as follows. In Section \ref{sec:main} we provide the unified theory for asymmetric weighted U-statistic, deriving central limit theorem, bootstrap, and a new resampling procedure's inference validity under data non-i.i.d. settings. In Section \ref{sec:application}, we apply the developed theory to explore the inference validity of Kendall's tau in \eqref{eq:tau} and AP correlation in \eqref{eq:tauAP}. All proofs are relegated to Appendix.

\section{Main results}\label{sec:main}

Throughout the manuscript, we focus on the following triangular array setting: Assume we have $n$ independent random variables $\{X_{n,i}\},n\geq1,1\leq i\leq n$. Each $X_{n,i}$ follows the distribution $P_{n,i}$. The elements in $\{P_{n,i}, i\in[n]\}$ are not necessarily equal to each other. When $n$ increases, $P_{n,i}$ could possibly change. For notational simplicity, in the sequel we drop $n$ in the subscripts of $X_{n,i}$ and $P_{n,i}$ when no confusion could be made. 

We are focused on the following weighted U-statistic of degree $m$, with weight function $a(\cdot):\mathbb{Z}^m\to\reals$ and kernel $h(\cdot):\cX^m\to\reals$: 
\begin{align}\label{eq:def-Un}
U_n = U_n(X_1,\ldots,X_n) =\frac{(n-m)!}{n!}\sum_{I_n^m}a_n(i_1,\ldots,i_m)h_n(X_{i_1},\ldots,X_{i_m}).
\end{align}
Here the summation is over all possible $m$ elements in $[n]$ without overlap:
\[
I_n^m:=\Big\{ 1\leq i_1, i_2, \ldots,i_m\leq n:i_j\neq i_k\mbox{ if }j\neq k\Big\} .
\]
Such $U_n$ is usually referred to as a weighted U-statistic in the literature \citep{serfling2009approximation}. We do not assume symmetry of $a_n(\cdot)$ or $h_n(\cdot)$ in their arguments.  For notation simplicity, in the sequel we omit the subscript $n$ in $a_n(\cdot)$ and $h_n(\cdot)$. 


Let's define 
\begin{equation}\label{eq:def-theta}
\theta(i_1,\ldots,i_m):=E\{h(X_{i_1},\ldots,X_{i_m})\}=\int h(y_1,\ldots,y_m)dP_{i_1}(y_1)\ldots dP_{i_m}(y_m)
\end{equation}
to be the population mean of $h(X_{i_1},\ldots,X_{i_m})$.
For any $l\in[m]$, define $\pi_l(\cdot;\cdot)$ to be a function that takes two arguments (a scalar and a vector of length $m-1$), and returns a vector of length $m$ by inserting the first argument into the $l$-th position of the second argument. Formally, we define
\[
\pi_l(y;y_1,y_2,\ldots,y_{m-1}):=(y_1,\ldots,y_{l-1},y,y_l,\ldots,y_{m-1}).
\]
We further define
\begin{align*}
a^{(l)}(i;i_1,i_2,\ldots,i_{m-1}) & :=a\{\pi_l(i;i_1,i_2,\ldots,i_{m-1})\},\\
h^{(l)}(x;x_1,\ldots,x_{m-1}) & :=h\{\pi_l(x;x_1,\ldots,x_{m-1})\},\\
\theta^{(l)}(i;i_1,i_2,\ldots,i_{m-1}) & :=\theta\{\pi_l(i;i_1,i_2,\ldots,i_{m-1})\}.
\end{align*}
Define the first order expansion of $h(\cdot)$ for each $X_i$, regarding the specific sequence $X_{i_1},\ldots,X_{i_{m-1}}$, to be:
\begin{align*}
h_{1,i;i_1,\ldots,i_{m-1}}(x) & :=\sum_{l=1}^ma^{(l)}(i;i_1,\ldots,i_{m-1})\big\{f_{i_1,\ldots,i_{m-1}}^{(l)}(x)-\theta^{(l)}(i;i_1,\ldots,i_{m-1})\big\},
\end{align*}
where 
\begin{align}
f_{i_1,\ldots,i_{m-1}}^{(l)}(x) & :=E_{i_1,\ldots,i_{m-1}}\{h^{(l)}(x;Y_1,\ldots,Y_{m-1})\}\nonumber  \\
 & = \int h^{(l)}(x;y_1,\ldots,y_{m-1})dP_{i_1}(y_1)\ldots dP_{i_{m-1}}(y_{m-1}).\label{eq:def-fl}
\end{align}
Define the first order expansion of $h(\cdot)$ for $X_i$ to be
\begin{align}
h_{1,i}(x) & :=\frac{(n-m)!}{(n-1)!}\sum_{I_{n-1}^{m-1}(-i)}h_{1,i;i_1,\ldots,i_{m-1}}(x),\label{eq:def-h1}
\end{align}
where the summation is over 
\begin{align*}
I_{n-1}^{m-1}(-i) & :=\Big\{ 1\leq i_1,\ldots,i_{m-1}\leq n:i_j\neq i_k\mbox{ if }j\neq k,\mbox{ and }i_j\neq i\mbox{ for all }j\in[m-1]\Big\} .
\end{align*}
For $l\in[m]$, we write $(i_1,\ldots,i_m)\backslash i_l:=(i_1,\ldots,i_{l-1},i_{l+1},\ldots,i_m)$,
and define 
\begin{align}\label{eq:def-h2}
h_{2;i_1,\ldots,i_m}(x_1,\ldots,x_m) & :=h(x_1,\ldots,x_m)-\sum_{l=1}^mf_{(i_1,\ldots,i_m)\backslash i_l}^{(l)}(x_l)+(m-1)\theta(i_1,\ldots,i_m),
\end{align}
where by (\ref{eq:def-fl}) we have 
\begin{align*}
f_{(i_1,\ldots,i_m)\backslash i_l}^{(l)}(x) & =\int h(y_1,\ldots,y_{l-1},x,y_{l+1},\ldots,y_m)dP_{i_1}(y_1)\ldots dP_{i_{l-1}}(y_{l-1})dP_{i_{l+1}}(y_{l+1})\ldots dP_{i_m}(y_m).
\end{align*}

Before presenting the main theorem, we have to introduce more notation on the weight function $a(\cdot)$. For $K,q\in\mathbb{Z}$ with $K\geq2$ and $0\leq q\leq m$, let $(I_n^m)_{\geq q}^{\otimes K}$ be the collection of all $K$-dimensional index vectors from $I_n^m$ that share at least
$q$ common indices: 
\[
(I_n^m)_{\geq q}^{\otimes K}:=\Big\{(i_1^{(1)},\ldots,i_m^{(1)})\in I_n^m,\ldots,(i_1^{(K)},\ldots,i_m^{(K)})\in I_n^m:\card\Big(\bigcap_{k=1}^{K}\{i_1^{(k)},\ldots,i_m^{(k)}\}\Big)\geq q\Big\},
\]
and $(I_n^m)_{=q}^{\otimes K}$ be the collection of all $K$-dimensional index vectors from $I_n^m$ that share exactly $q$ indices in common:
\[
(I_n^m)_{=q}^{\otimes K}=\Big\{(i_1^{(1)},\ldots,i_m^{(1)})\in I_n^m,\ldots,(i_1^{(K)},\ldots,i_m^{(K)})\in I_n^m:\card\Big(\bigcap_{k=1}^{K}\{i_1^{(k)},\ldots,i_m^{(k)}\}\Big)=q\Big\}.
\]
With fixed $K,q,m$, it is easy to observe $\card\{(I_n^m)_{\geq q}^{\otimes K}\}\asymp\card\{(I_n^m)_{=q}^{\otimes K}\}$
as $n\to\infty$, and 
\[
\card\{(I_n^m)_{=q}^{\otimes K}\}\asymp\binom{n}{q}\binom{n-q}{m-q}\cdots\binom{n-(K-1)m-q}{m-q}\asymp n^{q+K(m-q)}.
\]
In particular, we have $\card\{(I_n^m)_{\geq2}^{\otimes2}\}\asymp n^{2m-2}$,
$\card\{(I_n^m)_{\geq1}^{\otimes2}\}\asymp n^{2m-1}$, and $\card\{(I_n^m)_{\geq1}^{\otimes3}\}\asymp n^{3m-2}$. Define the average weight, $A_{K,q}(n)$, as 
\begin{align}
A_{K,q}(n) & :=\frac{1}{\card\{(I_n^m)_{\geq q}^{\otimes K}\}}\sum_{(I_n^m)_{\geq q}^{\otimes K}}\left|a(i_1^{(1)},\ldots,i_m^{(1)})\cdots a(i_1^{(K)},\ldots,i_m^{(K)})\right|.\label{eq:def-A2}
\end{align}

The following theorem gives sufficient conditions on the weights and distributions of $\{X_i\}$ for guaranteeing $U_n$ to be asymptotically normal.
\begin{theorem}[Sufficient condition for asymptotic normality of $U_n$]
\label{thm:clt-degreem}
For each $n$, assume there exists a positive constant $M(n)>0$ only depending on $n$ such that 
\begin{align}
\max_{(i_1,\ldots,i_m)\in I_n^m}E\{h(X_{i_1},\ldots,X_{i_m})^{4}\} & \leq M(n).\label{eq:clt-degreem-condition-moment}
\end{align}
Define $V(n)=\var\{n^{-1}\sum_{i=1}^nh_{1,i}(X_i)\}$ with $h_{1,i}(\cdot)$ defined in (\ref{eq:def-h1}). Assume the following conditions hold:
\begin{align}
n^{-2}V(n)^{-1}A_{2,2}(n)M(n)^{1/2} & \to0,\label{eq:clt-degreem-condition-1}\\
n^{-2}V(n)^{-3/2}A_{3,1}(n)M(n)^{3/4} & \to0.\label{eq:clt-degreem-condition-2}
\end{align}
Then we have 
\begin{equation}
\var(U_n)/V(n)\to1,\label{eq:clt-degreem-variancetendtoone}
\end{equation}
and 
\begin{align}
\var(U_n)^{-1/2}\{U_n-E(U_n)\} & \stackrel{d}{\to}N(0,1).\label{eq:clt-degreem}
\end{align}
\end{theorem}

The first step of the proof, which establishes a von-Mises-expansion type result, is simple yet inspiring. 
Of note, under i.i.d. settings, an analogous theorem has been (inexplicitly) stated in \cite{shapiro1979asymptotic}.

\begin{lemma}[Hoeffding's decomposition]
\label{lem:decomposition} With $h_{1,i}(\cdot)$ and $h_{2;i_1,\ldots,i_m}(\cdot)$
defined in (\ref{eq:def-h1}) and (\ref{eq:def-h2}), we have 
\begin{align}
U_n-E(U_n) & =\frac{1}{n}\sum_{i=1}^nh_{1,i}(X_i)+U_n(a,h_2),\label{eq:hoeffding-decomposition-degreem}
\end{align}
where
\begin{equation}
U_n(a,h_2):=\frac{(n-m)!}{n!}\sum_{I_n^m}a(i_1,\ldots,i_m)h_{2;i_1,\ldots,i_m}(X_{i_1},\ldots,X_{i_m}),\label{eq:def-Uh2}
\end{equation}
and for any $i,k\in[n]$ and $(i_1,\ldots,i_m)\in I_n^m$, 
\begin{align}
E\{h_{1,i}(X_i)\} & =0\label{eq:Eh1i},\\
E\{h_{2;i_1,\ldots,i_m}(X_{i_1},\ldots,X_{i_m})\mid X_k\} & =0~~\as.\label{eq:Eh2i}
\end{align}
\end{lemma}




For putting Theorem \ref{thm:clt-degreem} appropriately in the literature, let's first give a brief review on the most relevant existing results. The first proof of asymptotic normality for (unweighted) nondegenerate U-statistics was given in \citet{hoeffding1948ustat}. \citet{grams1973} studied general unweighted U-statistics of degree $m\geq 2$ and bounded their central moments . The techniques therein also play a central role in our analysis. 
\cite{shapiro1979asymptotic} and \citet{oneil1993weighted} analyzed the asymptotic behavior of weighted U-statistics of degree 2. They assumed weight function $a(\cdot)$ symmetric. The above results all assume data i.i.d.-ness. For unweighted U-statistics, \cite{lee1990u} outlined an extension to non-i.i.d. data. 

Theorem \ref{thm:clt-degreem} is stronger than the results in the literature, allowing $a(\cdot)$ and $h(\cdot)$ asymmetric, and the $X_i$'s non-i.i.d.. 
By examining the proof, one can also easily check that, when the corresponding symmetry, boundedness, or i.i.d. assumptions are made, our results can reduce to the ones in \citet{hoeffding1948ustat}, \cite{shapiro1979asymptotic}, \citet{oneil1993weighted}, and \cite{lee1990u}. 
\begin{remark}
Condition \eqref{eq:clt-degreem-condition-1} is added to enforce domination of $n^{-1}\sum_{i=1}^nh_{1,i}(X_i)$ over $U_n(a,h_2)$ in \eqref{eq:hoeffding-decomposition-degreem}. Condition (\ref{eq:clt-degreem-condition-2}) evolves from the Lyapunov condition with $\delta=1$, which is readily weakened to the condition of a smaller $0<\delta<1$ or the Lindeberg-Feller condition. Condition \eqref{eq:clt-degreem-condition-moment} is made and could be weakened based on the same argument. For presentation clearness, we choose the current conditions. 
\end{remark}

Inferring the distribution of $U_n$ or approximating $\var(U_n)$ is usually challenging in practice. Resampling procedures are hence recommended. The rest of this section gives asymptotic results for Efron's bootstrap \citep{efron1979} and a new resampling procedure for approximating $\var(U_n)$. 

Due to the heterogeneity in $P_i$, it is well known that bootstrap could possibly no longer be consistent \citep{liu1988bootstrap}. However, it is still possible to recover bootstrap consistency by restricting the heterogeneity degree. 
But before that, let's first provide a theoretically interesting theorem. It states that, under very mild conditions, bootstrapped mean from the set $\{h_{1,i}(X_i):1\leq i\leq n\}$ approximates the distribution of $n^{-1}\sum_{i=1}^nh_{1,i}(X_i)$ consistently. This is consistent to the discovery in \cite{liu1988bootstrap} by noting that $E\{ h_{1,i}(X_i) \} = 0$ no matter how different $\{P_i\}_{i=1}^n$ are.

\begin{theorem}[Sufficient condition for bootstrapping main term to work]
 \label{thm:bootstrap-mainterm} Denote
\begin{equation}
\sigma_n^2:=\var(U_n).\label{eq:def-sigma2}
\end{equation}
Consider the term $n^{-1}\sum_{i=1}^nh_{1,i}(X_i)$ with $h_{1,i}(X_i)$ defined in (\ref{eq:def-h1}) and its bootstrapped version $n^{-1}\sum_{i=1}^n\{h_{1,i}(X_i)\}^*$, where conditional on $X_1,\ldots,X_n$ the $\{h_{1,i}(X_i)\}^*$'s are i.i.d. draws from the empirical distribution of $\{h_{1,j}(X_j):1\leq j\leq n\}$. Assume (\ref{eq:clt-degreem-condition-moment}) and (\ref{eq:clt-degreem-condition-2}) hold. In addition, assume
for every $\epsilon>0$, we have 
\begin{align}
\sup_{1\leq i\leq n}P\Big\{\Big|\frac{h_{1,i}(X_i)}{n\sigma_n}\Big|\geq\epsilon\Big\} & \to0,\label{eq:ustat-bootstrap-condition-1}\\
\sum_{i=1}^n\Big[E\Big\{\frac{h_{1,i}(X_i)}{n\sigma_n}\ind\Big(\Big|\frac{h_{1,i}(X_i)}{n\sigma_n}\Big|\leq\epsilon\Big)\Big\}\Big]^2 & \to0.\label{eq:ustat-bootstrap-condition-2}
\end{align}
Then 
\begin{align}
\sup_{t\in\mathbb{R}}\bigg|P^*\Big\{\sum_{i=1}^n\frac{\{h_{1,i}(X_i)\}^*}{n\sigma_n}-\sum_{i=1}^n\frac{h_{1,i}(X_i)}{n\sigma_n}\leq t\Big\} - P\Big\{\sum_{i=1}^n\frac{h_{1,i}(X_i)}{n\sigma_n}\leq t\Big\}\bigg| & \stackrel{P}{\to}0,\label{eq:ustat-mainterm-bootstrap-result}
\end{align}
where $P^*$ denotes the conditional probability given $X_1,\ldots,X_n$. If further (\ref{eq:clt-degreem-condition-1}) holds, then
\begin{equation}
\sup_{t\in\mathbb{R}}\bigg|P^*\Big\{\sum_{i=1}^n\frac{\{h_{1,i}(X_i)\}^*}{n\sigma_n}-\sum_{i=1}^n\frac{h_{1,i}(X_i)}{n\sigma_n}\leq t\Big\} - P\Big\{\var(U_n)^{-1/2}\{U_n-E(U_n)\}\leq t\Big\}\bigg|  \stackrel{P}{\to}0. \label{eq:ustat-mainterm-bootstrap-result-additional}
\end{equation}
\end{theorem}
\begin{remark}
Equations (\ref{eq:ustat-bootstrap-condition-1}) and (\ref{eq:ustat-bootstrap-condition-2}) are rather mild constraints. As we will show in Corollary \ref{cor:bootstrap-tauap-taukendall-mainterm}, usually they can be directly deduced from the asymptotic normality of $U_n$. However, unless we know much about $X_i$, the form of $h_{1,i}(\cdot)$ is unknown. 
\end{remark}

We now focus on bootstrapping the original U-statistic for estimating $\var(U_n)$. The following theorem shows that Efron's bootstrap still gives consistent variance estimate for $U_n$ under some additional conditions on data heterogeneity. Although the bootstrap inference validity for U-statistics under i.i.d. assumptions has been established (check, for example, \cite{korolyuk2013theory}), the corresponding one for non-i.i.d. settings, even for the simplest unweighted U-statistics,  is still absent in the literature. Our manuscript fills this gap.

\begin{theorem}[Sufficient condition for consistent bootstrap variance estimation]
 \label{thm:bootstrap-u-stat} Given $X_1,\ldots,X_n$, let $X_1^*,\ldots,X_n^*$
denote the bootstrapped sample, which are i.i.d. draws from the empirical
distribution of $X_1,\ldots,X_n$. Define the bootstrapped U-statistic
\[
U_n^*=\frac{(n-m)!}{n!}\sum_{I_n^m}a(i_1,\ldots,i_m)h(X_{i_1}^*,\ldots,X_{i_m}^*).
\]
Assume all conditions in Theorem \ref{thm:clt-degreem} are satisfied. Also assume the following conditions hold: 
\begin{enumerate}
\item[(i)] Bounded second moment of von-Mises type kernel: 
\begin{equation}
\limsup_{n\to\infty}\max_{1\leq i_1,\ldots,i_m\leq n}E\{h(X_{i_1},\ldots,X_{i_m})^2\}<\infty.\label{eq:ustat-bootstrap-condition-vonmises}
\end{equation}

\item[(ii)] Control of heterogeneity in the distributions of $X_i$: 
\begin{align}	
\frac{1}{n}\sum_{i=1}^n\sum_{j=1}^n\Big\{\frac{h_{1,i}(X_j)}{n\sigma_n}\Big\}^2 & \stackrel{P}{\to}1,\label{eq:ustat-bootstrap-condition-homo1}\\
\frac{1}{n^2}\sum_{i=1}^n\Big\{\sum_{j=1}^n\frac{h_{1,i}(X_j)}{n\sigma_n}\Big\}^2 & \stackrel{P}{\to}0,\label{eq:ustat-bootstrap-condition-homo2}
\end{align}
and 
\begin{align}
n^{-1}\sigma_n^{-2}A_{2,1}(n)\{M_1(n)^2+M_2(n)+n^{-1}\} & \to0,\label{eq:ustat-bootstrap-condition-3}
\end{align}
where 
\begin{align}
M_1(n) & =\max_{(I_n^m)_{\geq0}^{\otimes2}}|\theta(i_1,\ldots,i_m)-\theta(j_1,\ldots,j_m)|,\label{eq:def-M1}\\
M_2(n) & =\max_{1\leq p,q\leq m}\max_{\substack{\br,\bs\in(I_n^m)_{=1}^{\otimes2}\\
\br\cap\bs=r_p=s_q
}
}\max_{\substack{\bk\in I_n^m\\
\bk\cap\bs=k_p=s_q
}
}\Big|E[E\{h(X_{r_1},\ldots,X_{r_m})h(X_{s_1}\ldots X_{s_m})\mid X_{k_p}\}]\nonumber  \\
 & ~~~~~-E[E\{h(X_{k_1},\ldots,X_{k_m})h(X_{s_1},\ldots,X_{s_m})\mid X_{k_p}\}]\Big|.\label{eq:def-M2}
\end{align}
Here we define $\br:=(r_1,\ldots,r_m)$, and similarly for $\bs,\bk$. 
\end{enumerate}
Then we have
\begin{align}
\left|\var^*(\sigma_n^{-1}U_n^*)-\var(\sigma_n^{-1}U_n)\right| & \stackrel{P}{\to}0,\label{eq:ustat-bootstrap-result}
\end{align}
where the operator $\var^*(\cdot)$ denotes the conditional variance given $X_1,\ldots,X_n$.
\end{theorem}
The detailed proof of Theorem \ref{thm:bootstrap-u-stat} is very involved and highly combinatorial. 
Of note, in the theorem, \eqref{eq:ustat-bootstrap-condition-vonmises} comes from \cite{bickel1981some}, ensuring that the bootstrapped U-statistic won't explode.  Equations (\ref{eq:ustat-bootstrap-condition-homo1})
and (\ref{eq:ustat-bootstrap-condition-homo2}) ensure that the conditional
variance of $n^{-1}\sum_{i=1}^nh_{1,i}(X_i^*)$ approximates
$\var(U_n)$. Equation (\ref{eq:ustat-bootstrap-condition-3}) ensures
that $U_n^*(a,h_2)$ is negligible compared to $n^{-1}\sum_{i=1}^nh_{1,i}(X_i^*)$.
\begin{remark}
Although $U_n(a,h_2)$ in the decomposition (\ref{eq:hoeffding-decomposition-degreem}) is degenerate and hence negligible under the conditions of Theorem \ref{thm:clt-degreem}, its bootstrapped version $U_n^*(a,h_2)$ is not necessarily degenerate, because the empirical measure does not equal the true measure. This makes $U_n^*(a,h_2)$ not necessarily negligible compared to the bootstrapped version of the main term, $n^{-1}\sum_{i=1}^nh_{1,i}(X_i^*)$.
Therefore, bootstrap may fail without careful control on both the main term and the remainder $U_n^*(a,h_2)$. We developed delicate analysis to bound $U_n^*(a,h_2)$ and showed that it is negligible under the constraint (\ref{eq:ustat-bootstrap-condition-3}). 
\end{remark}

\begin{remark}
Condition (\ref{eq:ustat-bootstrap-condition-3}) puts homogeneity conditions mainly on the means. This is consistent to Theorem \ref{thm:bootstrap-mainterm} and the discoveries in \cite{liu1988bootstrap}, who showed that bootstrap is most sensitive to mean differences. 
To illustrate, assume $a(\cdot)\equiv1$ and the kernel $h(\cdot)$ to be a bounded function. Assume the assumptions in Theorem \ref{thm:clt-degreem} hold, so that we have asymptotic normality of $U_n$. Equation (\ref{eq:clt-degreem-condition-2}) requires $\sigma_n^2\gnsim n^{-4/3}$. Therefore, for (\ref{eq:ustat-bootstrap-condition-3}) to hold, it is necessary that $M_1(n)^2\lnsim n^{-1/3}$ and $M_2(n)\lnsim n^{-1/3}$. The space to improve our requirements, if existing, is relatively small. This is by noting that, even for the simplest sample-mean-type statistics, for most cases, \cite{liu1988bootstrap} required the mean differences shrink to zero as $n\to\infty$ for bootstrap consistency. 
\end{remark}

An immediate implication of Theorem \ref{thm:bootstrap-u-stat} proves the validity of bootstrapping weighted U-statistics for i.i.d. data.

\begin{corollary}
\label{cor:bootstrap-iid}Assume that $X_1,\ldots,X_n$ are i.i.d.,
and that (\ref{eq:clt-degreem-condition-moment}), (\ref{eq:clt-degreem-condition-1}),
(\ref{eq:clt-degreem-condition-2}), and (\ref{eq:ustat-bootstrap-condition-vonmises})
hold. In addition, assume $n^{-2}\sigma_n^{-2}A_{2,1}(n)\to0$.
Then (\ref{eq:ustat-bootstrap-condition-homo1}), (\ref{eq:ustat-bootstrap-condition-homo2}),
and (\ref{eq:ustat-bootstrap-condition-3}) hold, and we have
\[
\left|\var^*(\sigma_n^{-1}U_n^*)-\var(\sigma_n^{-1}U_n)\right|\stackrel{P}{\to}0.
\]
\end{corollary}

\begin{remark}
\label{rem:bootstrap-iid-mild-assumption}The assumption $n^{-2}\sigma_n^{-2}A_{2,1}(n)\to0$ is mild. Actually it follows immediately from (\ref{eq:clt-degreem-condition-1}) if we have $A_{2,1}(n)\lesssim A_{2,2}(n)$. It is reasonable to expect $A_{2,1}(n)$ and $A_{2,2}(n)$ to be of similar order because of their definitions in (\ref{eq:def-A2}). Indeed, for the two applications in Section \ref{sec:application}, we have $A_{2,1}(n)\asymp A_{2,2}(n)$ for $U_n^{\ken}$ and $A_{2,1}(n)\lesssim A_{2,2}(n)\lesssim A_{2,1}(n)\log n$ for $U_n^{\ap}$.
\end{remark}



In many cases, although the data are in general non-i.i.d., they possess some locally stationary property \citep{dahlhaus1997fitting}. For example, consider the following nonparametric regression model. Assume $X_i\sim N(\mu_i,1)$ with $\mu_i=g_n(i/n)$ for $i=1,\ldots,n$. If the function $g_n(\cdot)$ is smooth enough (e.g., $\epsilon(n)$-Lipschitz), then, although $|g_n(1)-g_n(0)|$ could increase to infinity, the subsample $\{X_i, X_{i+1},\ldots,X_{i+b-1}\}$, for each $i\in 1,\ldots,n-b+1$, can be approximately i.i.d.. 

Adopting this thinking, we consider the following revised resampling procedure whose idea comes from \cite{politis1994large} and \cite{bickel1997resampling}, but is tailored for non-i.i.d. data. 
This is also related to the local block bootstrap developed in \cite{paparoditis2002local} and \cite{kreiss2015bootstrapping}.
In detail, for $m<b\to \infty$, we consider the following statistic:
\[
V^*_n=\frac{1}{h_n(n-b+1)}\sum_{i=1}^{n-b+1}\var^*(U_{b,i}^*),~~{\rm where}~~U^*_{b,i}:=\frac{(b-m)!}{b!}\sum_{I_{b}^m}a(i_1,\ldots,i_m)h(X_{i_1,b,i}^*,\ldots,X_{i_m,b,i}^*),
\]
and for each $i\in[n-b+1]$, $X_{i_1,b,i}^*,\ldots,X_{i_m,b,i}^*$ are independently drawn from the empirical distribution of $\{X_i,\ldots,X_{i+b-1}\}$ with replacement. The tuning parameter $h_n$ regulates the scale.

The following theorem verifies the new resampling procedure's inference consistency for $V^*_n$, showing that the procedure tends to give conservative variance estimate under non-i.i.d. settings. It also shows that the inference is more tractable compared to Efron's bootstrap when we have more prior information on the heterogeneity degree, reflected in the consistency rate of $U_n$ and the choice of $h_n$. We also refer the readers to Remark \ref{remark:rate} and discussions therein for the order of $\sigma_n$ in a specific example.

\begin{theorem} \label{thm:bootstrap-resampling}
Assume that all conditions in Theorem \ref{thm:bootstrap-u-stat} hold for each ``moving block" $\{X_i,\ldots,X_{i+b-1}\}$ of $i\in[n-b+1]$ as $n, b\to\infty$. Assume $\var(U_{b,i}(X_i,\ldots,X_{i+b-1}))=\sigma_b^2(1+o(1))$ for any $i\in[n-b+1]$, and $\sigma_b^2/\sigma_n^2= \zeta_{n,b}\cdot(1+o(1))$ for some $\zeta_{n,b}>0$. We then have
\[
\sigma_n^{-2}V_n^*-\var(\sigma_n^{-1}U_n) = \frac{\zeta_{n,b}}{h_n}\cdot(1+o_P(1))-1.
\]
\end{theorem}

\section{Application\label{sec:application}}


This section explores two specific statistics, the Kendall's tau (denoted as $\tau^{\ken}$) \citep{kendall1938new} and average-precision (AP) correlation (denoted as $\tau^{\ap}$) \citep{yilmaz2008new}:
\begin{align*}
\tau^{\ken} & =\frac{2}{n(n-1)}\sum_{i\neq j}\left\{ \ind(X_i>X_j)\ind(i<j)+\ind(X_j>X_i)\ind(j<i)\right\} -1,\\
\tau^{\ap} & =\frac{2}{n-1}\sum_{i=2}^n\frac{\sum_{j=1}^{i-1}\ind(X_j>X_i)}{i-1}-1.
\end{align*}
Without loss of generality, we focus on the transformed versions of these two statistics: 
\begin{align*}
U_n^{\ken} & =\frac{\tau^{\ken}+1}{4}=\frac{1}{n(n-1)}\sum_{i\neq j}\ind(j<i)\ind(X_j>X_i),\\
U_n^{\ap} & :=\frac{\tau^{\ap}+1}{2}=\frac{1}{n(n-1)}\sum_{i\neq j}\frac{n\ind(j<i)}{i-1}\ind(X_j>X_i).
\end{align*}
We assume $\{P_i,i\in[n]\}$ to be absolutely continuous with regard to the Lebesgue measure. Obviously, both $U_n^{\ken}$ and $U_n^{\ap}$ enjoy the distribution-free property \citep{stuart1968advanced} when the data are i.i.d.. Of note, these two statistics could also be treated as (weighted) U-statistics of symmetric kernels and weights with non-i.i.d. data $(X_1,1),\ldots,(X_n,n)$. However, we found the following analysis based on $X_1,\ldots,X_n$ much neater, and as will be seen in the proof, non-i.i.d.-ness is the major obstacle in analysis. 



\subsection{Asymptotic theory} \label{subsec:application-asymptotics}

Note that the statistics $U_n^{\ken}$ and $U_n^{\ap}$ have the same kernel $h(x,y)=\ind(y>x)$. Using the definition in (\ref{eq:def-theta}), we have $\theta(i,j)=E\{h(X_i,X_j)\}=P(X_j>X_i)$. The forms of $h_{1,i}(\cdot)$ and $h_{2;i,j}(\cdot)$ for $U_n^{\ken}$ and $U_n^{\ap}$ are then summarized in the following two lemmas.
\begin{lemma}[Hoeffding's decomposition of $U_n^{\ken}$]
 \label{lem:decomposition-taukendall} We have 
\begin{align*}
U_n^{\ken}-E(U_n^{\ken}) & =\frac{1}{n}\sum_{i=1}^nh_{1,i}^{\ken}(X_i)+\frac{1}{n(n-1)}\sum_{i\neq j}\ind(j<i)h_{2;i,j}^{\ken}(X_i,X_j),
\end{align*}
where 
\begin{align}
h_{1,i}^{\ken}(x) & =\frac{1}{n-1}\sum_{j=1}^n\{\ind(j<i)-\ind(j>i)\}\{P(X_j>x)-\theta(i,j)\} \label{eq:taukendall-h1i}
\end{align}
and
\[
h_{2;i,j}^{\ken}(x,y)=\ind(y>x)-P(X_j>x)-P(y>X_i)+\theta(i,j).
\]

\end{lemma}

\begin{lemma}[Hoeffding's Decomposition of $U_n^{AP}$]
 \label{lem:decomposition-tauap} We have 
\begin{align*}
U_n^{\ap}-E(U_n^{\ap}) & =\frac{1}{n}\sum_{i=1}^nh_{1,i}^{\ap}(X_i)+\frac{1}{n(n-1)}\sum_{i\neq j}\frac{n\ind(j<i)}{i-1}h_{2;i,j}^{\ap}(X_i,X_j),
\end{align*}
where 
\begin{align}
h_{1,i}^{\ap}(x) & =\frac{1}{n-1}\sum_{j=1}^n\Big\{\frac{n\ind(j<i)}{i-1}-\frac{n\ind(j>i)}{j-1}\Big\}\{P(X_j>x)-\theta(i,j)\},\label{eq:tauap-h1i}
\end{align}
and
\[
h_{2;i,j}^{\ap}(x,y)=\ind(y>x)-P(X_j>x)-P(y>X_i)+\theta(i,j).
\]
In (\ref{eq:tauap-h1i}), by convention, we have $0/0:=0$.
\end{lemma}

The next theorem characterizes sufficient distributional conditions for $U_n^{\ken}$ and $U_n^{\ap}$ to be asymptotically normal, allowing for data non-i.i.d.-ness. 

\begin{theorem}[Sufficient conditions for asymptotic normality of $U_n^{\ken}$
and $U_n^{\ap}$]
 \label{thm:tauap-taukendall-AN} Assume a sequence $\{\delta_n\in(0,1)\}_{n=1}^{\infty}$ and a sequence $\{p_n\in(0,1)\}_{n=1}^{\infty}$ such that for any sufficiently large $n$ and for each $i\in[n]$, one of the following two conditions holds:
\begin{enumerate}
\item[(i)] $P\{P(X_j>X_i\mid X_i)-P(X_j>X_i)\in [\delta_n,1],\forall j\in[n]\backslash\{i\}\}\geq p_n$; 
\item[(ii)] $P\{P(X_j>X_i\mid X_i)-P(X_j>X_i)\in[-1,-\delta_n],\forall j\in[n]\backslash\{i\}\}\geq p_n$. 
\end{enumerate}
In addition, if
\begin{align}
\delta_n^{3}p_n\gnsim & n^{-1/3},\label{eq:taukendall-AN-rateassumption}
\end{align}
then $U_n^{\ken}$ is asymptotically normal,
\[
\var(U_n^{\ken})^{-1/2}\{U_n^{\ken}-E(U_n^{\ken})\}\stackrel{d}{\to}N(0,1).
\]
If we have
\begin{align}
\delta_n^{3}p_n & \gnsim n^{-1/3}(\log n)^2,\label{eq:tauap-AN-rateassumption}
\end{align}
then $U_n^{\ap}$ is asymptotically normal, 
\[
\var(U_n^{\ap})^{-1/2}\{U_n^{\ap}-E(U_n^{\ap})\}\stackrel{d}{\to}N(0,1).
\]
\end{theorem}

The proof of Theorem \ref{thm:tauap-taukendall-AN} exploits Theorem \ref{thm:clt-degreem}. 
A key step in the proof is to bound $V(n):=n^{-2}\sum_i\var\{h_{1,i}(X_i)\}$ from below. The magnitude of $\var\{h_{1,i}(X_i)\}$ varies greatly with different $i$, making it a challenging task to bound the entire summation. To tackle this, we break $V(n)$ into summations over multiple subsets of $[n]$. Within each of these summations, the magnitude of $\var\{h_{1,i}(X_i)\}$ is stable. Then we develop bounds on the summations for $i$ with large $\var\{h_{1,i}(X_i)\}$.

The sequences $\{\delta_n\}$ and $\{p_n\}$ in Conditions (i) and (ii) of Theorem \ref{thm:tauap-taukendall-AN} characterize the heterogeneity degree among the $P_i$'s. If all $P_i$'s are identical, it is easy to check that there exist absolute constants $\delta_n$ and $p_n$ not depending on $n$ such that Condition (i) or (ii) holds. Equations (\ref{eq:taukendall-AN-rateassumption}) and (\ref{eq:tauap-AN-rateassumption}) allow $\delta_n$ and $p_n$ to decay to zero as $n\to\infty$. The legitimate decaying rate of $\delta_n^{3}p_n$ depends on the average weight of each of the two statistics. The conditions for asymptotic normality of $U_n^{\ap}$ (\ref{eq:tauap-AN-rateassumption}) are slightly stronger than that for $U_n^{\ken}$ (\ref{eq:taukendall-AN-rateassumption}), because for $U_n^{\ap}$ the weight is much more skewed.

\begin{remark}\label{remark:rate}
In the literature about Kendall's tau, the classical result gives root-$n$ convergence rate \citep{sen1968kendall}. Theorem \ref{thm:tauap-taukendall-AN} gives a more general result regarding the convergence rate due to the non-i.i.d.-ness of $\{X_1,\ldots,X_n\}$. In the proof of Theorem \ref{thm:tauap-taukendall-AN}, we show that the $\var(U_n^{\ken}) \gnsim n^{-1}\delta_n^{3}p_n$. As we vary the distribution of $X_i$ from i.i.d. to the more heterogeneous ones, $\delta_n^{3}p_n$ changes from $O(1)$ to $O(n^{-1/3+\epsilon})$ for some small $\epsilon>0$. Therefore, the upper bound on the order of $\var(U_n^{\ken})^{-1/2}$ can vary from $n^{1/2}$ to $n^{2/3-\epsilon/2}$.
\end{remark}

Motivated by the studies in \cite{yilmaz2008new}, in the sequel we consider the following specific location-scale model. In particular, given two sets of real values $\mu_i$ with $\mu_1\geq\mu_2\geq\ldots\geq\mu_n$ and $\sigma_1^2,\ldots,\sigma_n^2>0$, let's consider absolute continuous (with respect to Lebesgue measure) probability distribution $P_i$ with mean $\mu_i$ and variance $\sigma_i^2$ for $i\in[n]$. Assume $X_1,\ldots,X_n$ are independent draws from $P_1,\ldots,P_n$. The following theorem characterizes the explicit sufficient conditions on $\{(\mu_i,\sigma_i),i\in[n]\}$ for Kendall's tau and AP correlation to be asymptotically normal.

\begin{theorem}[Sufficient condition for asymptotic normality of $U_n^{\ken}$ and $U_n^{\ap}$ under two tail conditions]  \label{thm:tauap-AN-tail-conditions}
 For each $i\in[n]$, assume $X_i$ follows distribution $P_i$ with mean $\mu_i$ and variance $\sigma_i^2$. Define 
 \[
 r_{ij}:=(\mu_i-\mu_j)/\sigma_i, ~~R_n:=\max_{1\leq i\neq j\leq n}|r_{ij}|, ~~\rho_{ij}:=\sigma_i/\sigma_j, ~{\rm and}~~ \rho_n:=\max_{1\leq i\neq j\leq n}\rho_{ij}. 
 \]
 For $n,i,j$ such that $1\leq i\neq j\leq n$, define 
\begin{align}
F_j^{c}(t)=P\Big(\frac{X_j-\mu_j}{\sigma_j}>t\Big) & ~~{\rm and}~~F_{ji}^{c}(t)=P\Big\{\frac{X_j-X_i-(\mu_j-\mu_i)}{(\sigma_i^2+\sigma_j^2)^{1/2}}>t\Big\}.\label{eq:def-FjFij}
\end{align}
Then the following results hold.
\begin{enumerate}
\item[(i)]  Assume there exist absolute constants $c_1,c_2>0$, $b_1>b_2>0$, and $t_0>0$,
such that for any $n,i,j$ with $1\leq i\neq j\leq n$ and for any
$t\geq t_0$,
\begin{align}
c_1t^{-b_1} & \leq F_j^{c}(t)\leq c_2t^{-b_2},\label{eq:tail-condition-1-1}\\
c_1t^{-b_1} & \leq F_{ji}^{c}(t)\leq c_2t^{-b_2}.\label{eq:tail-condition-1-2}
\end{align}
Then the sufficient condition for asymptotic normality of $U_n^{\ken}$
is 
\begin{align}
R_n^{(3b_1b_2+b_1^2)/b_2}\rho_n^{b_1} & \lnsim n^{1/3},\label{eq:tail-condition-1-rate-kendall}
\end{align}
and the sufficient condition for asymptotic normality of $U_n^{\ap}$
is
\begin{align}
R_n^{(3b_1b_2+b_1^2)/b_2}\rho_n^{b_1} & \lnsim n^{1/3}(\log n)^{-2}.\label{eq:tail-condition-1-rate-ap}
\end{align}

\item[(ii)]  Assume there exist absolute constants $c_1,c_2>0$, $b_1>b_2>0$, and $t_0>0$,
such that for any $n,i,j$ with $1\leq i\neq j\leq n$ and for any
$t\geq t_0$,
\begin{align}
c_1\exp(-b_1t^{\lambda}) & \leq F_j^{c}(t)\leq c_2\exp(-b_2t^{\lambda}),\label{eq:tail-condition-2-1}\\
c_1\exp(-b_1t^{\lambda}) & \leq F_{ji}^{c}(t)\leq c_2\exp(-b_2t^{\lambda}).\label{eq:tail-condition-2-2}
\end{align}
Then the sufficient condition for asymptotic normality of $U_n^{\ken}$
is
\begin{equation}
3b_1R_n^{\lambda}+b_1(R_n+K_{3}\rho_n+K_{4}\rho_nR_n)^{\lambda}\lnsim\frac{1}{3}\log n,\label{eq:tail-condition-2-rate-kendall}
\end{equation}
and the sufficient condition for asymptotic normality of $U_n^{\ap}$
is
\begin{equation}
3b_1R_n^{\lambda}+b_1(R_n+K_{3}\rho_n+K_{4}\rho_nR_n)^{\lambda}\lnsim\frac{1}{3}\log n-2\log\log n,\label{eq:tail-condition-2-rate-ap}
\end{equation}
where
\begin{align}
K_{3} & :=t_0+\Big(-\frac{1}{b_2}\log\frac{c_1}{2c_2}+\frac{b_1}{b_2}t_0^{\lambda}\Big)^{1/\lambda}+\xi(\lambda^{-1})\Big(-\frac{1}{b_2}\log\frac{c_1}{2c_2}\Big)^{1/\lambda},\nonumber \\
K_{4} & :=\xi(\lambda^{-1})\Big(\frac{b_1}{b_2}\Big)^{1/\lambda},\label{eq:tail-conditiondef-K3K4}
\end{align}
and $\xi(p):=\ind(p\leq1)+2^{p-1}\ind(p>1)$.
\end{enumerate}
\end{theorem}
\begin{remark}
It is worth noting that distributions satisfying (\ref{eq:tail-condition-1-1}) in Theorem \ref{thm:tauap-AN-tail-conditions}(i) are commonly referred to as ``heavy-tailed'' distributions, whereas distributions satisfying (\ref{eq:tail-condition-2-1}) in Theorem \ref{thm:tauap-AN-tail-conditions}(ii) are considered to be ``light-tailed'' \citep{mikosch1999regular, resnick2007heavy}.
\end{remark}

We compare Condition (\ref{eq:tail-condition-1-rate-kendall}) in (i) and Condition (\ref{eq:tail-condition-2-rate-kendall}) in (ii) for $U_n^{\ken}$. Assume $\sigma_i=1$ for all $i\in[n]$. In this case, we have $\rho_n=1$, and $R_n=\max_{1\leq i\neq j\leq n}|\mu_i-\mu_j|$ becomes the spread of the means. Equation (\ref{eq:tail-condition-1-rate-kendall}) becomes
\begin{equation}
R_n\lnsim n^{\frac{b_2}{3(3b_1b_2+b_1^2)}}.\label{eq:tail-condition-remarkuse-1}
\end{equation}
Equation (\ref{eq:tail-condition-2-rate-kendall}) becomes
\begin{equation}
3b_1R_n^{\lambda}+b_1(R_n+K_{3}+K_{4}R_n)^{\lambda}\lnsim\frac{1}{3}\log n.  \label{eq:tail-condition-remarkuse-2}
\end{equation}
Lemma \ref{lem:bound-sum-power} in Appendix yields $(R_n+K_{3}+K_{4}R_n)^{\lambda}\leq\xi(\lambda)(1+K_{4})^{\lambda}R_n^{\lambda}+\xi(\lambda)K_{3}^{\lambda}$. So for (\ref{eq:tail-condition-remarkuse-2}) to hold, it suffices to have $\xi(\lambda)(1+K_{4})^{\lambda}R_n^{\lambda}+3b_1R_n^{\lambda}\lnsim (\log n) / 3$. Rearranging terms, we obtain a sufficient condition for (\ref{eq:tail-condition-remarkuse-2}) to hold:
\begin{equation}
R_n\lnsim\Big[\frac{\log n}{3b_1\{3+\xi(\lambda)(1+K_{4})^{\lambda}\}}\Big]^{1/\lambda}.\label{eq:tail-condition-remarkuse-3}
\end{equation}
For heavy-tailed distributions in (i), (\ref{eq:tail-condition-remarkuse-1}) implies that the spread of means should not grow faster than a polynomial of $n$. For light-tailed distributions in (ii), (\ref{eq:tail-condition-remarkuse-3}) implies that the spread of means should not grow faster than the logarithm of $n$ (up to some constant scaling factor). 
Of note, under both tail conditions, $R_n$ is allowed to increase to infinity at proper rates.

\begin{example}
A special distribution satisfying the conditions in Theorem \ref{thm:tauap-AN-tail-conditions}(ii) is  the Gaussian. Again, consider $U_n^{\ken}$ and assume $\sigma_i=1$ for all $i\in[n]$. Note in this case $F_j^{c}(\cdot)$ is the survival function for Gaussian with variance $1$, whereas $F_{ji}^{c}(\cdot)$ is for Gaussian with variance 2. Let $\lambda=2$, $b_1=1/2+\epsilon$, $b_2=1/4-\epsilon$ for arbitrarily small $\epsilon>0$, and $c_1,c_2,t_0$ be properly chosen constants (whose value does not affect the rate in (\ref{eq:tail-condition-remarkuse-3})). Equations (\ref{eq:tail-condition-1-1}) and (\ref{eq:tail-condition-1-2}) are satisfied due to Lemma \ref{lem:bound-normal-cdf}. It then follows from (\ref{eq:tail-condition-remarkuse-3}) that
\[
R_n\lnsim\Big(\frac{2 \log n}{27 + 12\sqrt{2}}\Big)^{1/2}
\]
is sufficient for $U_n^{\ken}$ to be asymptotically normal.

\end{example}

\begin{remark}
We comment on a modified version of Theorem \ref{thm:tauap-AN-tail-conditions}(i), with a condition alternative to (\ref{eq:tail-condition-1-1}) (A similar modification applies to Theorem \ref{thm:tauap-AN-tail-conditions}(ii)). In detail, define $F_j(t)=P\{(X_j-\mu_j)/\sigma_j\leq t\}$ to be the standardized cumulative distribution function that is complement to the survival function $F_j^c(t)$. The conclusion in  Theorem \ref{thm:tauap-AN-tail-conditions}(i) still holds if we replace the condition (\ref{eq:tail-condition-1-1}) by
\begin{equation}
c_1t^{-b_1}  \leq F_j(-t) \leq c_2t^{-b_2}. \label{eq:tail-condition-1-1-alter}
\end{equation}
For comparison, (\ref{eq:tail-condition-1-1}) regulates the upper-tail behavior of $X_j$, whereas (\ref{eq:tail-condition-1-1-alter}) regulates the lower-tail of $X_j$. Technically speaking, the proof of Theorem \ref{thm:tauap-AN-tail-conditions}(i) examines Condition (ii) in Theorem \ref{thm:tauap-AN-tail-conditions}, whereas the alternative version examines Condition (i) in Theorem \ref{thm:tauap-AN-tail-conditions}. Note that (\ref{eq:tail-condition-1-2}) is required in both versions, and regulates both the upper- and lower-tail behaviors of $X_j - X_i$.
\end{remark}

The following three corollaries give asymptotic results for bootstrapping $U_n^{\ken}$ and $U_n^{\ap}$. The first of them states that bootstrapping the main term is very insensitive to data non-i.i.d.-ness. This is as expected by the results in \cite{liu1988bootstrap}. 

\begin{corollary}[Bootstrap of main term works for $U_n^{\ken}$ and $U_n^{\ap}$]
 \label{cor:bootstrap-tauap-taukendall-mainterm} If (\ref{eq:taukendall-AN-rateassumption})
holds, we have that (\ref{eq:ustat-bootstrap-condition-1}) and (\ref{eq:ustat-bootstrap-condition-2})
hold for $h_{1,i}^{\ken}$. If (\ref{eq:tauap-AN-rateassumption})
holds, we have that (\ref{eq:ustat-bootstrap-condition-1}) and (\ref{eq:ustat-bootstrap-condition-2}) 
hold for $h_{1,i}^{\ap}$.
\end{corollary}

As has been shown in Section \ref{sec:main}, bootstrapping the whole U-statistic requires much stronger assumptions for guaranteeing its consistency. The following two corollaries provide sufficient conditions for bootstrap inference validity of the two considered U-statistics.

\begin{corollary}[Sufficient condition for consistent bootstrap variance estimation
of $U_n^{\ken}$]
\label{cor:bootstrap-taukendall} Assume (\ref{eq:taukendall-AN-rateassumption}) holds. Assume there exist $\theta>0$ and an absolute constant $C>0$ such that for all $(i,j)\in I_n^2$,
\begin{equation}
|P(X_i>X_j)-\theta|\leq Cn^{-1/6}.  \label{eq:taukendall-bootstrap-condition-1}
\end{equation}
In addition, assume there exist $\eta^2>0$ and an absolute constant $C>0$ such
that for all $i\in[n]$ and all $1\leq j,k\leq n$ such that $j\neq i$
and $k\neq i$,
\begin{equation}
|E\{P(X_j>X_i\mid X_i)P(X_k>X_i\mid X_i)\}-\eta^2|\leq Cn^{-1/3}.\label{eq:taukendall-bootstrap-condition-2}
\end{equation}
Assume $\eta^2\neq\theta^2$. Then we have
\begin{align*}
|\var^*(\sigma_n^{-1}U_n^{\ken*})-\var(\sigma_n^{-1}U_n^{\ken})| & \stackrel{P}{\to}0.
\end{align*}
\end{corollary}

\begin{corollary}[Sufficient condition for consistent bootstrap variance estimation
of $U_n^{\ap}$]
 \label{cor:bootstrap-tauap} Assume (\ref{eq:tauap-AN-rateassumption}) holds. Assume there exist $\theta>0$ and an absolute constant $C>0$ such that for all $(i,j)\in I_n^2$,
\begin{equation}
|P(X_i>X_j)-\theta|\leq Cn^{-1/6}\log n. \label{eq:tauap-bootstrap-condition-1}
\end{equation}
In addition, assume there exist $\eta^2>0$ and an absolute constant $C>0$ such
that for all $1\le i\leq n$ and all $1\leq j,k\leq n$ such that
$j\neq i$ and $k\neq i$,
\begin{equation}
|E\{P(X_j>X_i\mid X_i)P(X_k>X_i\mid X_i)\}-\eta^2|\leq Cn^{-1/3}(\log n)^2.\label{eq:tauap-bootstrap-condition-2}
\end{equation}
Assume $\eta^2\neq\theta^2$. Then we have
\begin{align*}
|\var^*(\sigma_n^{-1}U_n^{\ap*})-\var(\sigma_n^{-1}U_n^{\ap})| & \stackrel{P}{\to}0.
\end{align*}
\end{corollary}



In the proof of Corollaries \ref{cor:bootstrap-taukendall} and \ref{cor:bootstrap-tauap}, for verifying (\ref{eq:ustat-bootstrap-condition-homo1}), we exploit the weak law of large numbers for independent but not identically distributed variables. For verifying (\ref{eq:ustat-bootstrap-condition-homo2}), we break the left-hand side into the sum of an unweighted U-statistic and a negligible term, and apply the law of large numbers for unweighted U-statistics.

\begin{remark}
The condition $\eta^2\neq\theta^2$ in Corollaries \ref{cor:bootstrap-taukendall} and \ref{cor:bootstrap-tauap} is mild. Under the i.i.d. case, it essentially requires that the $X_i$'s are not degenerate random variables. To see this, let $\theta:=P(X_1>X_2)$ and $\eta^2:=E\{P(X_1>X_2\mid X_1)^2\}$. Since the $X_i$'s are i.i.d., it follows that 
\[
|P(X_i>X_j)-\theta|=0~~{\rm and}~~|E\{P(X_j>X_i\mid X_i)P(X_k>X_i\mid X_i)\}-\eta^2|=0.
\]
Jensen's inequality implies that $\eta^2\geq\theta^2$, with equality only if $X_i$ is a degenerate random variable. 
\end{remark}

\subsection{Numerical experiments}



In this section, we evaluate the developed theory and examine the finite sample behavior of Kendall's tau and AP correlation via synthetic data analysis. Both central limit theorem and bootstrap inference validity are examined under different data heterogeneity degree. The numerical results show that central limit theorem holds under relatively weaker data homogeneity requirement, whereas bootstrap variance estimation is much more sensitive to data heterogeneity. 
These observations agree with the theory developed in this manuscript.


In the first simulation study, we examine the validity of central limit theorem for Kendall's tau and AP correlation. We consider generating the data from Gaussian distribution and $t$-distribution. For Gaussian distribution, each time, we generate the data sequence $X_1, \ldots, X_n$ with $X_i \sim N(\theta_i,1)$ for $i\in[n]$. The means $\{\theta_i, i\in[n]\}$ are assigned equally spaced between $R_n$ and 0, with $R_n=\max|\theta_i-\theta_j|$ representing the heterogeneity degree, taking values 0, 10, 30, and 50. For $t$-distribution, we generate $X_1, \ldots, X_n$ with $X_i$ follows noncentral $t$-distribution with noncentrality parameter $\theta_i$ and 5 degrees of freedom. The noncentrality parameters $\{\theta_i, i\in[n]\}$ are assigned equally spaced between $R_n$ and 0, and $R_n$ takes values 0, 8, 25, and 42. We choose these $R_n$, so that the difference between the means of $X_1$ and $X_n$ are similar under Gaussian distribution and under $t$-distribution. We consider the sample size $n$ being 50, 100, 200, and 500. 

Under each setting, we repeat the simulation for 50,000 times. We use two goodness-of-fit tests to examine the normality of the considered statistics: Cramer-von Mises test (CvM) and Lilliefors test (L). Both tests are implemented in the R package ``{\sf Rnortest}'', and we refer the readers to \citet{thode2002testing} for detailed descriptions on these tests. We also calculate the coverage probability for confidence intervals of nominal level 80\% and 95\% based on Gaussian approximation.

Table \ref{tab:clt-normal} presents the $p$-values of two tests for normality and the coverage probabilities, when the data are generated from Gaussian distribution.  For both statistics, with large sample size ($n=500$) normality is plausible for $R_n$ up to 50, as both tests fail to reject at significance level $0.05$. Test rejection occurs as sample size decreases. In terms of confidence interval, for sample size as small as $n=50$, the coverage probabilities are all close to the nominal level even for large $R_n$. Note that with $R_n=50$ the 95\% confidence interval for $U_n^{\ap}$ becomes slightly conservative, especially with smaller sample size.

Table \ref{tab:clt-t} presents the $p$-values and the coverage probabilities when the data follow noncentral $t$-distribution. The trend is similar to that of Table \ref{tab:clt-normal}, while by comparison, we observed that the statistics are more robust to location shifts for the heavy-tailed $t$-distribution as compared to Gaussian distribution, supporting our theoretical discoveries.

\bigskip

In the second simulation study, we examine the bootstrap variance estimation consistency and present the results in Tables \ref{tab:boot-kendall-normal} - \ref{tab:boot-ap-t}. We consider the following three approaches:
(i) bootstrapping the original U-statistic as in Theorem \ref{thm:bootstrap-u-stat}, termed as ``Efron'' in the tables;
(ii) bootstrapping the main term of the U-statistic as in Theorem \ref{thm:bootstrap-mainterm}, termed as ``Efron-main term'' in the tables;
(iii) the new resampling strategy as in Theorem \ref{thm:bootstrap-resampling}, termed as ``moving-block'' in the tables.
Among them, the ``Efron-main term'' bootstrap is not of practical use because it requires knowledge of $h_{1i}(X_i)$, which depends on the probability distribution of $X_i$. We include it for theoretical purpose in order to validate Theorem \ref{thm:bootstrap-mainterm}.
Similar to the first simulation study, we generate the data from Gaussian distribution and $t$-distribution. For Gaussian distribution, we simulate $X_i\sim N(\theta_i, 1)$. For $t$-distribution, we simulate $X_i$ following noncentral $t$-distribution with noncentrality parameter $\theta_i$ and 5 degrees of freedom. For both distributions, parameters $\{\theta_i, i\in[n]\}$ are assigned equally spaced between $R_n$ and 0, and the degree of heterogeneity $R_n$ is set to be 0, 1, 2, and 3. We consider the sample size $n$ being 50, 100, 200, and 500. We set the number of bootstrap replicates within each simulation to be 2,000 in bootstrap approaches (i) and (ii), 200 for each block in bootstrap approach (iii), and the block-size in (iii) to be $n/5$.

Under each setting, we repeat the simulation for 50,000 times. In the ``bias'' column of each table, we present the relative bias of the bootstrap variance estimators, where the relative bias is defined as $\{\var(U_n) - \widehat\var(U_n) \}/ \var(U_n)$. Relative bias being positive/negative means that the bootstrap method tends to underestimate/overestimate the variance. We also compute the coverage probability for confidence intervals of nominal level 80\% and 95\% based on Gaussian approximation and the estimated variance.

Table \ref{tab:boot-kendall-normal} shows the performance of three bootstrap variance estimators for $U_n^{\ken}$ when $X_i$ follows Gaussian distribution. When there is no heterogeneity in the data ($R_n = 0$), all three bootstrap methods consistently estimates the variance, with close to zero bias and close to nominal level confidence interval coverage. As the distribution of $X_i$ becomes more heterogeneous (larger $R_n$), bootstrapping the main-term still estimates the variance consistently, whereas Efron's bootstrap and the moving-block bootstrap tend to overestimates the variance, resulting in negative relative bias and larger than nominal confidence interval coverage. This is as expected due to Corollary \ref{cor:bootstrap-tauap-taukendall-mainterm} and Corollary \ref{cor:bootstrap-taukendall}. Table \ref{tab:boot-ap-normal} gives the bootstrap performance for $U_n^{\ap}$ when $X_i$ follows Gaussian distribution, and the trend is similar to $U_n^{\ken}$. 
A comparison between Table \ref{tab:boot-kendall-normal} and Table \ref{tab:boot-ap-normal} shows that the finite sample performance of all three bootstrap methods is better for $U_n^{\ap}$ than for $U_n^{\ken}$. This is consistent with our theoretical findings in Corollary \ref{cor:bootstrap-taukendall} and Corollary \ref{cor:bootstrap-tauap}.
Tables \ref{tab:boot-kendall-t} and \ref{tab:boot-ap-t} present the results for both statistics when $X_i$ follows $t$-distribution. The trends there are similar to the Gaussian case, and by comparison, the statistics are more robust to location shifts for $t$-distribution, supporting our theorems. 

Comparing the first and the second simulation studies, we see that the central limit theorem for our considered statistics holds under much weaker homogeneity conditions than the resampling procedures. This is as expected due to the theory developed in Section \ref{subsec:application-asymptotics}. We also see that central limit theorem holds approximately with sample size as small as $n=50$, whereas bootstrap variance estimation requires much larger sample size to have decent performance.


\begin{table}[htbp]
\caption{\small Asymptotic normality of $U_n^{\ken}$ and $U_n^{\ap}$ when $X_i$ follows Gaussian distribution.}
\label{tab:clt-normal}
\centering
\resizebox{\linewidth}{!}{\begin{tabular}[t]{>{\bfseries}cccccccccccccccccc}
\toprule
\multicolumn{1}{c}{ } & \multicolumn{1}{c}{ } & \multicolumn{4}{c}{$R_n=0$} & \multicolumn{4}{c}{$R_n=10$} & \multicolumn{4}{c}{$R_n=30$} & \multicolumn{4}{c}{$R_n=50$} \\
\cmidrule(l{2pt}r{2pt}){3-6} \cmidrule(l{2pt}r{2pt}){7-10} \cmidrule(l{2pt}r{2pt}){11-14} \cmidrule(l{2pt}r{2pt}){15-18}
\multicolumn{1}{c}{ } & \multicolumn{1}{c}{ } & \multicolumn{2}{c}{p-value} & \multicolumn{2}{c}{cov.prob.(\%)} & \multicolumn{2}{c}{p-value} & \multicolumn{2}{c}{cov.prob.(\%)} & \multicolumn{2}{c}{p-value} & \multicolumn{2}{c}{cov.prob.(\%)} & \multicolumn{2}{c}{p-value} & \multicolumn{2}{c}{cov.prob.(\%)} \\
\cmidrule(l{2pt}r{2pt}){3-4} \cmidrule(l{2pt}r{2pt}){5-6} \cmidrule(l{2pt}r{2pt}){7-8} \cmidrule(l{2pt}r{2pt}){9-10} \cmidrule(l{2pt}r{2pt}){11-12} \cmidrule(l{2pt}r{2pt}){13-14} \cmidrule(l{2pt}r{2pt}){15-16} \cmidrule(l{2pt}r{2pt}){17-18}
statistic & $n$ & CvM & L & 80\% & 95\% & CvM & L & 80\% & 95\% & CvM & L & 80\% & 95\% & CvM & L & 80\% & 95\%\\
\midrule
 & 50 & 0.11 & 0.05 & 79.5 & 95.1 & 0.00 & 0.00 & 79.7 & 95.1 & 0.00 & 0.00 & 82.1 & 94.1 & 0.00 & 0.00 & 79.4 & 95.0\\

 & 100 & 0.10 & 0.09 & 79.9 & 95.0 & 0.00 & 0.00 & 80.1 & 95.0 & 0.00 & 0.00 & 80.0 & 94.9 & 0.00 & 0.00 & 81.3 & 95.0\\

 & 200 & 0.78 & 0.67 & 80.1 & 95.1 & 0.61 & 0.43 & 80.1 & 95.1 & 0.02 & 0.01 & 79.9 & 94.9 & 0.02 & 0.00 & 80.1 & 95.0\\

\multirow{-4}{*}{\centering\arraybackslash $U_n^{\ken}$} & 500 & 0.92 & 0.68 & 80.2 & 94.9 & 0.09 & 0.10 & 80.3 & 95.0 & 0.32 & 0.21 & 80.1 & 95.0 & 0.33 & 0.26 & 79.9 & 95.0\\
\cmidrule{1-18}
 & 50 & 0.37 & 0.63 & 80.0 & 95.0 & 0.00 & 0.01 & 79.8 & 95.2 & 0.00 & 0.00 & 79.6 & 95.8 & 0.00 & 0.00 & 81.1 & 96.0\\

 & 100 & 0.23 & 0.12 & 79.9 & 95.0 & 0.02 & 0.04 & 79.9 & 95.1 & 0.00 & 0.00 & 79.9 & 95.2 & 0.00 & 0.00 & 79.7 & 95.5\\

 & 200 & 0.02 & 0.07 & 79.9 & 95.0 & 0.71 & 0.54 & 80.0 & 95.0 & 0.01 & 0.08 & 79.8 & 95.1 & 0.01 & 0.09 & 79.7 & 95.3\\

\multirow{-4}{*}{\centering\arraybackslash $U_n^{\ap}$} & 500 & 0.88 & 0.87 & 79.9 & 94.9 & 0.69 & 0.67 & 79.8 & 95.1 & 0.19 & 0.23 & 80.1 & 95.1 & 0.06 & 0.09 & 79.9 & 95.0\\
\bottomrule
\end{tabular}}

\end{table}

\begin{table}[htbp]
\caption{\small Asymptotic normality of $U_n^{\ken}$ and $U_n^{\ap}$ when $X_i$ follows $t$-distribution.}
\label{tab:clt-t}
\centering
\resizebox{\linewidth}{!}{\begin{tabular}[t]{>{\bfseries}cccccccccccccccccc}
\toprule
\multicolumn{1}{c}{ } & \multicolumn{1}{c}{ } & \multicolumn{4}{c}{$R_n=0$} & \multicolumn{4}{c}{$R_n=8$} & \multicolumn{4}{c}{$R_n=25$} & \multicolumn{4}{c}{$R_n=42$} \\
\cmidrule(l{2pt}r{2pt}){3-6} \cmidrule(l{2pt}r{2pt}){7-10} \cmidrule(l{2pt}r{2pt}){11-14} \cmidrule(l{2pt}r{2pt}){15-18}
\multicolumn{1}{c}{ } & \multicolumn{1}{c}{ } & \multicolumn{2}{c}{p-value} & \multicolumn{2}{c}{cov.prob.(\%)} & \multicolumn{2}{c}{p-value} & \multicolumn{2}{c}{cov.prob.(\%)} & \multicolumn{2}{c}{p-value} & \multicolumn{2}{c}{cov.prob.(\%)} & \multicolumn{2}{c}{p-value} & \multicolumn{2}{c}{cov.prob.(\%)} \\
\cmidrule(l{2pt}r{2pt}){3-4} \cmidrule(l{2pt}r{2pt}){5-6} \cmidrule(l{2pt}r{2pt}){7-8} \cmidrule(l{2pt}r{2pt}){9-10} \cmidrule(l{2pt}r{2pt}){11-12} \cmidrule(l{2pt}r{2pt}){13-14} \cmidrule(l{2pt}r{2pt}){15-16} \cmidrule(l{2pt}r{2pt}){17-18}
statistic & $n$ & CvM & L & 80\% & 95\% & CvM & L & 80\% & 95\% & CvM & L & 80\% & 95\% & CvM & L & 80\% & 95\%\\
\midrule
 & 50 & 0.29 & 0.06 & 79.6 & 95.1 & 0.00 & 0.00 & 79.9 & 95.2 & 0.00 & 0.00 & 80.6 & 95.2 & 0.00 & 0.00 & 80.7 & 95.1\\

 & 100 & 0.06 & 0.09 & 80.0 & 95.0 & 0.00 & 0.00 & 80.0 & 95.1 & 0.00 & 0.00 & 80.1 & 95.1 & 0.00 & 0.00 & 79.9 & 95.1\\

 & 200 & 0.64 & 0.36 & 80.0 & 94.9 & 0.01 & 0.06 & 80.0 & 95.1 & 0.00 & 0.00 & 80.1 & 95.1 & 0.00 & 0.00 & 80.2 & 95.0\\

\multirow{-4}{*}{\centering\arraybackslash $U_n^{\ken}$} & 500 & 0.62 & 0.59 & 80.0 & 95.0 & 0.20 & 0.33 & 80.0 & 95.0 & 0.75 & 0.88 & 80.0 & 95.0 & 0.51 & 0.44 & 80.0 & 95.0\\
\cmidrule{1-18}
 & 50 & 0.72 & 0.39 & 79.9 & 95.0 & 0.33 & 0.32 & 80.1 & 95.0 & 0.00 & 0.00 & 80.1 & 95.1 & 0.00 & 0.00 & 80.1 & 95.1\\

 & 100 & 0.08 & 0.11 & 80.0 & 95.1 & 0.02 & 0.04 & 80.0 & 95.0 & 0.03 & 0.08 & 80.1 & 95.0 & 0.03 & 0.02 & 80.1 & 95.0\\

 & 200 & 0.74 & 0.62 & 79.8 & 95.1 & 0.01 & 0.11 & 80.0 & 94.9 & 0.00 & 0.03 & 80.2 & 95.0 & 0.00 & 0.01 & 80.2 & 95.0\\

\multirow{-4}{*}{\centering\arraybackslash $U_n^{\ap}$} & 500 & 0.53 & 0.75 & 80.1 & 94.8 & 0.54 & 0.62 & 79.9 & 95.1 & 0.29 & 0.35 & 80.1 & 95.1 & 0.16 & 0.36 & 79.9 & 95.1\\
\bottomrule
\end{tabular}}

\end{table}


\begin{table}[htbp]
\caption{\small Bootstrap variance estimation validity for $U_n^{\ken}$ when $X_i$ follows Gaussian distribution.}
\label{tab:boot-kendall-normal}
\centering
\resizebox{\linewidth}{!}{\begin{tabular}[t]{>{\bfseries}cccccccccccccc}
\toprule
\multicolumn{1}{c}{ } & \multicolumn{1}{c}{ } & \multicolumn{3}{c}{$R_n=0$} & \multicolumn{3}{c}{$R_n=1$} & \multicolumn{3}{c}{$R_n=2$} & \multicolumn{3}{c}{$R_n=3$} \\
\cmidrule(l{2pt}r{2pt}){3-5} \cmidrule(l{2pt}r{2pt}){6-8} \cmidrule(l{2pt}r{2pt}){9-11} \cmidrule(l{2pt}r{2pt}){12-14}
\multicolumn{1}{c}{ } & \multicolumn{1}{c}{ } & \multicolumn{1}{c}{ } & \multicolumn{2}{c}{cov.prob.(\%)} & \multicolumn{1}{c}{ } & \multicolumn{2}{c}{cov.prob.(\%)} & \multicolumn{1}{c}{ } & \multicolumn{2}{c}{cov.prob.(\%)} & \multicolumn{1}{c}{ } & \multicolumn{2}{c}{cov.prob.(\%)} \\
\cmidrule(l{2pt}r{2pt}){4-5} \cmidrule(l{2pt}r{2pt}){7-8} \cmidrule(l{2pt}r{2pt}){10-11} \cmidrule(l{2pt}r{2pt}){13-14}
bootstrap method & $n$ & bias & 80\% & 95\% & bias & 80\% & 95\% & bias & 80\% & 95\% & bias & 80\% & 95\%\\
\midrule
 & 50 & 0.02 & 79.5 & 95.0 & -0.10 & 82.1 & 95.9 & -0.49 & 88.4 & 98.1 & -1.25 & 94.6 & 99.7\\

 & 100 & -0.02 & 80.8 & 95.2 & -0.13 & 83.0 & 96.1 & -0.51 & 88.8 & 98.3 & -1.23 & 94.7 & 99.5\\

 & 200 & -0.02 & 80.6 & 95.1 & -0.14 & 83.2 & 96.4 & -0.55 & 89.0 & 98.8 & -1.31 & 95.0 & 99.6\\

\multirow{-4}{*}{\centering\arraybackslash Efron} & 500 & -0.01 & 79.9 & 95.4 & -0.13 & 82.7 & 96.3 & -0.52 & 88.9 & 98.5 & -1.28 & 94.5 & 99.7\\
\cmidrule{1-14}
 & 50 & 0.06 & 78.1 & 93.6 & 0.07 & 77.9 & 93.4 & 0.07 & 77.5 & 93.1 & 0.07 & 77.0 & 92.3\\

 & 100 & 0.00 & 80.1 & 94.6 & 0.02 & 79.0 & 94.1 & 0.03 & 78.8 & 93.7 & 0.05 & 78.6 & 93.4\\

 & 200 & 0.00 & 79.9 & 94.8 & -0.01 & 80.3 & 94.8 & -0.01 & 80.3 & 94.8 & -0.01 & 80.0 & 94.8\\

\multirow{-4}{*}{\centering\arraybackslash Efron-main term} & 500 & -0.01 & 79.7 & 95.4 & -0.01 & 80.0 & 95.0 & 0.00 & 79.9 & 94.9 & 0.00 & 79.8 & 94.6\\
\cmidrule{1-14}
 & 50 & -0.28 & 85.4 & 97.3 & -0.42 & 87.9 & 98.1 & -0.93 & 92.8 & 99.2 & -1.92 & 97.1 & 99.9\\

 & 100 & -0.17 & 83.5 & 96.6 & -0.29 & 85.7 & 97.3 & -0.73 & 90.7 & 99.0 & -1.55 & 96.2 & 99.7\\

 & 200 & -0.09 & 82.1 & 95.8 & -0.22 & 84.5 & 97.0 & -0.66 & 90.1 & 99.0 & -1.48 & 95.8 & 99.7\\

\multirow{-4}{*}{\centering\arraybackslash moving-block} & 500 & -0.04 & 80.4 & 95.6 & -0.16 & 83.3 & 96.6 & -0.57 & 89.2 & 98.6 & -1.34 & 94.9 & 99.7\\
\bottomrule
\end{tabular}}
\end{table}

\begin{table}[htbp]
\caption{\small Bootstrap variance estimation validity for $U_n^{\ap}$ when $X_i$ follows Gaussian distribution.}
\label{tab:boot-ap-normal}
\centering
\resizebox{\linewidth}{!}{\begin{tabular}[t]{>{\bfseries}cccccccccccccc}
\toprule
\multicolumn{1}{c}{ } & \multicolumn{1}{c}{ } & \multicolumn{3}{c}{$R_n=0$} & \multicolumn{3}{c}{$R_n=1$} & \multicolumn{3}{c}{$R_n=2$} & \multicolumn{3}{c}{$R_n=3$} \\
\cmidrule(l{2pt}r{2pt}){3-5} \cmidrule(l{2pt}r{2pt}){6-8} \cmidrule(l{2pt}r{2pt}){9-11} \cmidrule(l{2pt}r{2pt}){12-14}
\multicolumn{1}{c}{ } & \multicolumn{1}{c}{ } & \multicolumn{1}{c}{ } & \multicolumn{2}{c}{cov.prob.(\%)} & \multicolumn{1}{c}{ } & \multicolumn{2}{c}{cov.prob.(\%)} & \multicolumn{1}{c}{ } & \multicolumn{2}{c}{cov.prob.(\%)} & \multicolumn{1}{c}{ } & \multicolumn{2}{c}{cov.prob.(\%)} \\
\cmidrule(l{2pt}r{2pt}){4-5} \cmidrule(l{2pt}r{2pt}){7-8} \cmidrule(l{2pt}r{2pt}){10-11} \cmidrule(l{2pt}r{2pt}){13-14}
bootstrap method & $n$ & bias & 80\% & 95\% & bias & 80\% & 95\% & bias & 80\% & 95\% & bias & 80\% & 95\%\\
\midrule
 & 50 & 0.01 & 79.6 & 94.9 & -0.05 & 80.5 & 95.5 & -0.25 & 84.9 & 97.3 & -0.53 & 89.0 & 98.5\\

 & 100 & -0.03 & 80.6 & 95.5 & -0.09 & 81.9 & 95.9 & -0.29 & 85.6 & 97.3 & -0.59 & 89.7 & 98.5\\

 & 200 & -0.02 & 81.1 & 95.0 & -0.09 & 82.6 & 95.5 & -0.30 & 85.9 & 97.2 & -0.62 & 89.7 & 98.8\\

\multirow{-4}{*}{\centering\arraybackslash Efron} & 500 & -0.02 & 80.5 & 96.0 & -0.09 & 81.9 & 96.0 & -0.30 & 85.6 & 97.3 & -0.61 & 89.9 & 98.6\\
\cmidrule{1-14}
 & 50 & 0.10 & 76.2 & 93.3 & 0.12 & 75.8 & 92.9 & 0.13 & 75.4 & 92.7 & 0.14 & 74.7 & 92.2\\

 & 100 & 0.04 & 79.0 & 94.2 & 0.04 & 78.3 & 94.5 & 0.05 & 77.7 & 93.9 & 0.07 & 76.6 & 93.6\\

 & 200 & 0.02 & 79.7 & 94.9 & 0.01 & 80.1 & 94.7 & 0.02 & 79.7 & 94.7 & 0.02 & 78.8 & 94.8\\

\multirow{-4}{*}{\centering\arraybackslash Efron-main term} & 500 & 0.00 & 79.8 & 95.4 & 0.00 & 79.7 & 95.1 & 0.01 & 79.8 & 95.0 & 0.01 & 79.4 & 94.7\\
\cmidrule{1-14}
 & 50 & -0.49 & 88.5 & 98.4 & -0.58 & 89.5 & 98.8 & -0.87 & 91.8 & 99.3 & -1.30 & 94.6 & 99.7\\

 & 100 & -0.33 & 86.0 & 97.5 & -0.42 & 87.2 & 97.9 & -0.68 & 90.6 & 98.8 & -1.06 & 93.5 & 99.4\\

 & 200 & -0.19 & 84.1 & 96.8 & -0.28 & 85.9 & 97.2 & -0.53 & 88.6 & 98.3 & -0.90 & 92.3 & 99.4\\

\multirow{-4}{*}{\centering\arraybackslash moving-block} & 500 & -0.11 & 82.3 & 96.6 & -0.18 & 83.6 & 97.0 & -0.41 & 87.2 & 97.8 & -0.75 & 91.6 & 99.0\\
\bottomrule
\end{tabular}}
\end{table}

\begin{table}[htbp]
\caption{\small Bootstrap variance estimation validity for $U_n^{\ken}$ when $X_i$ follows $t$-distribution. }
\label{tab:boot-kendall-t}
\centering
\resizebox{\linewidth}{!}{\begin{tabular}[t]{>{\bfseries}cccccccccccccc}
\toprule
\multicolumn{1}{c}{ } & \multicolumn{1}{c}{ } & \multicolumn{3}{c}{$R_n=0$} & \multicolumn{3}{c}{$R_n=1$} & \multicolumn{3}{c}{$R_n=2$} & \multicolumn{3}{c}{$R_n=3$} \\
\cmidrule(l{2pt}r{2pt}){3-5} \cmidrule(l{2pt}r{2pt}){6-8} \cmidrule(l{2pt}r{2pt}){9-11} \cmidrule(l{2pt}r{2pt}){12-14}
\multicolumn{1}{c}{ } & \multicolumn{1}{c}{ } & \multicolumn{1}{c}{ } & \multicolumn{2}{c}{cov.prob.(\%)} & \multicolumn{1}{c}{ } & \multicolumn{2}{c}{cov.prob.(\%)} & \multicolumn{1}{c}{ } & \multicolumn{2}{c}{cov.prob.(\%)} & \multicolumn{1}{c}{ } & \multicolumn{2}{c}{cov.prob.(\%)} \\
\cmidrule(l{2pt}r{2pt}){4-5} \cmidrule(l{2pt}r{2pt}){7-8} \cmidrule(l{2pt}r{2pt}){10-11} \cmidrule(l{2pt}r{2pt}){13-14}
bootstrap method & $n$ & bias & 80\% & 95\% & bias & 80\% & 95\% & bias & 80\% & 95\% & bias & 80\% & 95\%\\
\midrule
 & 50 & 0.03 & 79.5 & 94.5 & -0.08 & 81.3 & 96.2 & -0.40 & 87.0 & 98.0 & -0.88 & 91.9 & 99.2\\

 & 100 & -0.02 & 80.1 & 95.1 & -0.13 & 82.9 & 96.3 & -0.46 & 88.1 & 98.3 & -0.97 & 93.0 & 99.4\\

 & 200 & 0.03 & 79.2 & 94.6 & -0.09 & 82.2 & 96.0 & -0.42 & 87.1 & 98.2 & -0.91 & 92.3 & 99.3\\

\multirow{-4}{*}{\centering\arraybackslash Efron} & 500 & 0.01 & 79.8 & 94.7 & -0.09 & 81.2 & 95.9 & -0.40 & 86.9 & 97.8 & -0.89 & 92.0 & 99.3\\
\cmidrule{1-14}
 & 50 & 0.08 & 77.5 & 93.1 & 0.08 & 76.4 & 93.5 & 0.09 & 76.8 & 92.4 & 0.09 & 77.2 & 91.7\\

 & 100 & 0.01 & 79.1 & 94.6 & 0.01 & 79.5 & 94.7 & 0.01 & 78.5 & 94.6 & 0.01 & 79.2 & 94.2\\

 & 200 & 0.04 & 78.8 & 94.3 & 0.03 & 79.4 & 94.3 & 0.03 & 79.3 & 94.2 & 0.02 & 79.3 & 94.2\\

\multirow{-4}{*}{\centering\arraybackslash Efron-main term} & 500 & 0.02 & 79.2 & 94.7 & 0.03 & 78.7 & 94.4 & 0.03 & 79.0 & 94.1 & 0.02 & 79.5 & 94.6\\
\cmidrule{1-14}
 & 50 & -0.26 & 85.1 & 97.1 & -0.40 & 87.0 & 98.2 & -0.81 & 91.5 & 99.0 & -1.44 & 95.7 & 99.6\\

 & 100 & -0.17 & 83.5 & 96.6 & -0.30 & 85.7 & 97.4 & -0.67 & 90.5 & 98.9 & -1.25 & 94.8 & 99.6\\

 & 200 & -0.04 & 80.8 & 95.4 & -0.17 & 83.7 & 96.8 & -0.52 & 88.5 & 98.6 & -1.05 & 93.3 & 99.4\\

\multirow{-4}{*}{\centering\arraybackslash moving-block} & 500 & -0.02 & 80.2 & 95.0 & -0.12 & 81.9 & 96.1 & -0.44 & 87.4 & 98.1 & -0.94 & 92.6 & 99.3\\
\bottomrule
\end{tabular}}
\end{table}

\begin{table}[htbp]
\caption{\small Bootstrap variance estimation validity for $U_n^{\ap}$ when $X_i$ follows $t$-distribution.}
\label{tab:boot-ap-t}
\centering
\resizebox{\linewidth}{!}{\begin{tabular}[t]{>{\bfseries}cccccccccccccc}
\toprule
\multicolumn{1}{c}{ } & \multicolumn{1}{c}{ } & \multicolumn{3}{c}{$R_n=0$} & \multicolumn{3}{c}{$R_n=1$} & \multicolumn{3}{c}{$R_n=2$} & \multicolumn{3}{c}{$R_n=3$} \\
\cmidrule(l{2pt}r{2pt}){3-5} \cmidrule(l{2pt}r{2pt}){6-8} \cmidrule(l{2pt}r{2pt}){9-11} \cmidrule(l{2pt}r{2pt}){12-14}
\multicolumn{1}{c}{ } & \multicolumn{1}{c}{ } & \multicolumn{1}{c}{ } & \multicolumn{2}{c}{cov.prob.(\%)} & \multicolumn{1}{c}{ } & \multicolumn{2}{c}{cov.prob.(\%)} & \multicolumn{1}{c}{ } & \multicolumn{2}{c}{cov.prob.(\%)} & \multicolumn{1}{c}{ } & \multicolumn{2}{c}{cov.prob.(\%)} \\
\cmidrule(l{2pt}r{2pt}){4-5} \cmidrule(l{2pt}r{2pt}){7-8} \cmidrule(l{2pt}r{2pt}){10-11} \cmidrule(l{2pt}r{2pt}){13-14}
bootstrap method & $n$ & bias & 80\% & 95\% & bias & 80\% & 95\% & bias & 80\% & 95\% & bias & 80\% & 95\%\\
\midrule
 & 50 & 0.02 & 79.2 & 95.1 & -0.05 & 80.6 & 95.8 & -0.21 & 84.0 & 96.8 & -0.43 & 87.3 & 98.2\\

 & 100 & 0.01 & 79.8 & 94.7 & -0.04 & 81.1 & 95.2 & -0.20 & 84.2 & 96.7 & -0.41 & 86.6 & 98.1\\

 & 200 & 0.02 & 79.7 & 94.7 & -0.05 & 81.0 & 95.4 & -0.22 & 84.1 & 97.0 & -0.44 & 87.7 & 98.1\\

\multirow{-4}{*}{\centering\arraybackslash Efron} & 500 & -0.01 & 80.0 & 94.6 & -0.07 & 81.0 & 95.7 & -0.24 & 84.3 & 97.0 & -0.47 & 87.7 & 98.2\\
\cmidrule{1-14}
 & 50 & 0.11 & 75.7 & 93.1 & 0.11 & 75.5 & 93.4 & 0.13 & 75.0 & 92.9 & 0.13 & 75.4 & 92.8\\

 & 100 & 0.07 & 77.9 & 94.2 & 0.07 & 77.7 & 94.3 & 0.08 & 77.0 & 93.9 & 0.08 & 77.0 & 94.1\\

 & 200 & 0.05 & 78.6 & 94.0 & 0.05 & 78.8 & 94.1 & 0.05 & 78.2 & 94.2 & 0.05 & 78.3 & 94.5\\

\multirow{-4}{*}{\centering\arraybackslash Efron-main term} & 500 & 0.01 & 79.4 & 94.6 & 0.01 & 79.3 & 94.7 & 0.02 & 79.1 & 94.6 & 0.00 & 79.3 & 94.7\\
\cmidrule{1-14}
 & 50 & -0.48 & 88.3 & 98.3 & -0.58 & 89.4 & 98.7 & -0.81 & 91.6 & 99.3 & -1.14 & 93.8 & 99.7\\

 & 100 & -0.28 & 85.4 & 97.5 & -0.35 & 86.8 & 97.8 & -0.56 & 89.1 & 98.6 & -0.83 & 91.7 & 99.3\\

 & 200 & -0.15 & 83.1 & 96.3 & -0.23 & 84.3 & 97.0 & -0.43 & 87.0 & 98.2 & -0.69 & 90.6 & 98.8\\

\multirow{-4}{*}{\centering\arraybackslash moving-block} & 500 & -0.09 & 81.5 & 95.8 & -0.16 & 83.3 & 96.3 & -0.34 & 85.7 & 97.6 & -0.59 & 89.3 & 98.6\\
\bottomrule
\end{tabular}}
\end{table}


\section{Discussion}

One of the main focus of this manuscript is the consistency of bootstrap variance estimator for U-statistics under data heterogeneity. The proof is based on brutal combinatorial calculation. This cannot be readily extended to analyzing bootstrap distributional consistency. We believe using techniques developed by \cite{mammen2012bootstrap} and \cite{hall2013bootstrap}, it is promising to devise the corresponding bootstrap distributional consistency theory. However, there are still some challenges and open problems to be resolved before rigorous distributional consistency theory can be established. Details will be worked out in a future work.

We have considered U-statistics with data that are independent but not identically distributed. In the literature, there have been many developments of bootstrap methods for stationary time series since the seminal work of block bootstrap methods by \cite{kunsch1989jackknife}. See, for example, \cite{politis1992circular}, \cite{lahiri1993moving}, and \cite{politis1994large}. Among the few developments for nonstationary time series, \cite{fitzenberger1998moving} showed that block bootstrap is robust for linear regression estimation, and \cite{gonccalves2002bootstrap} established the consistency for block bootstrap variance estimator of sample means. To the best of knowledge, there is no work on bootstrapping U-statistics in the nonstationary time-series setting. It would be interesting to extend our current techniques in this manuscript to allow for dependent data. We believe our technique and the techniques used in the bootstrapping time-series literature (e.g., \cite{carlstein1998matched}, \cite{paparoditis2001tapered}, and \cite{shao2010dependent}) can be potentially combined for analyzing bootstrapping U-statistics for nonstationary time series data. However, the analysis will become even more challenging technically, and will be left for future research.

\section*{Acknowledgement}

We thank the editor, the associate editor, and two anonymous referees for their careful reviews and constructive comments. Dr. Fang Han's research is supported by NSF grant DMS-1712536 and a UW faculty start-up grant. Dr. Tianchen Qian's research is supported by Patient-Centered Outcomes Research Institute (ME-1306-03198).


\appendix

\section{Proofs of main results}
\label{sec:proofs}

In this section, we prove theoretical results presented in the manuscript. The results are proved in the order they appear in the manuscript. For succinctness, the supporting lemmas that appear in the proof are proven in Section \ref{sec:proof-supportinglemmas}. In those proof, sometimes we also have to refer to certain auxiliary results. Those are numbered by \ref{sec:aux-lemmas}.


\subsection{Proof of Theorem \ref{thm:clt-degreem}}
\begin{proof}
By Lemma \ref{lem:decomposition}, we have 
\begin{align}
\left\{ \frac{\var(U_n)}{V(n)}\right\} ^{1/2}\frac{U_n-E(U_n)}{\var(U_n)^{1/2}} & =\frac{n^{-1}\sum_{i=1}^nh_{1,i}(X_i)}{V(n)^{1/2}}+\frac{U_n(a,h_2)}{V(n)^{1/2}}.\label{eq:clt-proofuse-0.8}
\end{align}
For proving Theorem \ref{thm:clt-degreem}, by Slutsky's theorem it suffices to establish the following results: 
\begin{align}
V(n)^{-1/2}n^{-1}\sum_{i=1}^nh_{1,i}(X_i) & \stackrel{d}{\to}N(0,1),\label{eq:clt-proofuse-toshow1}\\
V(n)^{-1/2}U_n(a,h_2) & \stackrel{P}{\to}0,\label{eq:clt-proofuse-toshow2}
\end{align}
and
\begin{equation}
\var(U_n)/V(n)\to1.\label{eq:clt-proofuse-toshow3}
\end{equation}

First we show (\ref{eq:clt-proofuse-toshow1}) using Lyapunov's Central Limit Theorem (Lemma \ref{lem:lyapunov-clt}).
The following lemma gives bound on $\sum_{i=1}^nE|h_{1,i}(X_i)|^{3}$.
\begin{lemma}
\label{lem:clt-proofuse}For $A_{3,1}(n)$ defined in (\ref{eq:def-A2})
and $M(n)$ defined in (\ref{eq:clt-degreem-condition-moment}), we have
\[
\sum_{i=1}^nE|h_{1,i}(X_i)|^{3}\leq CnA_{3,1}(n)M(n)^{3/4},
\]
where $C$ is some absolute constant.
\end{lemma}
\bigskip{}
By Lemma \ref{lem:clt-proofuse} and the fact that $E\{h_{1,i}(X_i)\}=0$,
we deduce
\begin{equation}
\sum_{i=1}^nE|h_{1,i}(X_i)-E\{h_{1,i}(X_i)\}|^{3}\leq CnA_{3,1}(n)M(n)^{3/4}.\label{eq:clt-proofuse-0.9}
\end{equation}
Since $V(n):=n^{-2}\sum_{i=1}^n\var\{h_{1,i}(X_i)\}$, it
follows from (\ref{eq:clt-proofuse-0.9}) and (\ref{eq:clt-degreem-condition-2}) that
\begin{equation}
\frac{\sum_{i=1}^nE|h_{1,i}(X_i)-E\{h_{1,i}(X_i)\}|^{3}}{\Big[\sum_{i=1}^nE|h_{1,i}(X_i)-E\{h_{1,i}(X_i)\}|^2\Big]^{3/2}}\leq\frac{CnA_{3,1}(n)M(n)^{3/4}}{n^{3}V(n)^{3/2}}\to0.\label{eq:clt-proofuse-0.95}
\end{equation}
Equation (\ref{eq:clt-proofuse-0.95}) and Lemma \ref{lem:lyapunov-clt} with $\delta=1$ yield (\ref{eq:clt-proofuse-toshow1}).

Next we show (\ref{eq:clt-proofuse-toshow2}). To simplify notation,
let $\bi$ denote the index vector $(i_1,\ldots,i_m)$ and $\Xbi$
denote $(X_{i_1},\ldots,X_{i_m})$. Consider two index vectors
$\bi,\bj$ from $I_n^m$. If $\bi\cap\bj=\emptyset$, by independence
of the $X_i$'s we have $\cov\{h_{2;\bi}(\Xbi),h_{2;\bj}(\Xbj)\}=0$.
If $\bi\cap\bj=i_p=j_q$ for some $p,q\in[n]$ (i.e., the two
vectors only share one common index), Lemma \ref{lem:cov-conditional-expectation}
and (\ref{eq:Eh2i}) imply that 
\[
\cov\{h_{2;\bi}(\Xbi),h_{2;\bj}(\Xbj)\}=\cov[E\{h_{2;\bi}(\Xbi)\mid X_{i_p}\},E\{h_{2;\bj}(\Xbj)\mid X_{j_q}\}]=0.
\]
Therefore, we have
\begin{align}
\var\{U_n(a,h_2)\} & =\Big\{\frac{(n-m)!}{n!}\Big\}^2\sum_{\bi,\bj\in(I_n^m)_{\geq2}^{\otimes2}}a(\bi)a(\bj)\cov\{h_{2;\bi}(\Xbi),h_{2;\bj}(\Xbj)\}.\label{eq:clt-proofuse-2.9}
\end{align}
By Lemma \ref{lem:moment-bound}(i) and Cauchy-Schwarz inequality,
the right-hand side of (\ref{eq:clt-proofuse-2.9}) is bounded by
$Cn^{-2}A_{2,2}(n)M(n)^{1/2}$ for some absolute constant $C$, where
$A_{2,2}(n)$ is defined in (\ref{eq:def-A2}). This combined with
(\ref{eq:clt-degreem-condition-1}) yields that
\begin{equation}
V(n)^{-1}\var\{U_n(a,h_2)\}\leq CV(n)^{-1}n^{-2}A_{2,2}(n)M(n)\to0.\label{eq:clt-proofuse-3.0}
\end{equation}
Equation (\ref{eq:clt-proofuse-toshow2}) follows from (\ref{eq:clt-proofuse-3.0})
and Lemma \ref{lem:var-Pconverge}.

Lastly, we establish (\ref{eq:clt-proofuse-toshow3}). Taking variance
on both sides of (\ref{eq:clt-proofuse-0.8}) gives
\begin{align}
\frac{\var(U_n)}{V(n)} & =1+\frac{\var\{U_n(a,h_2)\}}{V(n)}+\cov\Big\{\frac{\sum_{i=1}^nh_{1,i}(X_i)}{nV(n)^{1/2}},\frac{U_n(a,h_2)}{V(n)^{1/2}}\Big\} \nonumber \\
& =1+\frac{\var\{U_n(a,h_2)\}}{V(n)}
\label{eq:clt-proofuse-1.2}
\end{align}
Equations (\ref{eq:clt-proofuse-3.0}) and (\ref{eq:clt-proofuse-1.2}) imply that 
\[
\var(U_n)/V(n)\to1.
\]
This completes the proof.
\end{proof}


\subsection{Proof of Lemma \ref{lem:decomposition}}
\begin{proof}
We have 
\[
U_n-E(U_n)=\sum_{i=1}^n\{E(U_n\mid X_i)-E(U_n)\}+\Big[U_n-E(U_n)-\sum_{i=1}^n\{E(U_n\mid X_i)-E(U_n)\}\Big].
\]
For proving Lemma \ref{lem:decomposition}, it suffices to show 
\begin{equation}
\sum_{i=1}^n\{E(U_n\mid X_i)-E(U_n)\}=\frac{1}{n}\sum_{i=1}^nh_{1,i}(X_i)\label{eq:decomposition-pf-1}
\end{equation}
and 
\begin{equation}
U_n-E(U_n)-\sum_{i=1}^n\{E(U_n\mid X_i)-E(U_n)\}=\frac{(n-m)!}{n!}\sum_{I_n^m}a(i_1,\ldots,i_m)h_{2;i_1,\ldots,i_m}(X_{i_1},\ldots,X_{i_m}),\label{eq:decomposition-pf-2}
\end{equation}
where $h_{1,i}(\cdot)$ and $h_{2;i_1,\ldots,i_m}(\cdot)$ are
defined in \eqref{eq:Eh1i} and \eqref{eq:Eh2i}, respectively.

First we establish (\ref{eq:decomposition-pf-1}). We have 
\begin{align}
E(U_n\mid X_i)-E(U_n)=\frac{(n-m)!}{n!}\sum_{I_n^m}a(i_1,\ldots,i_m)\Big[E\{h(X_{i_1},\ldots,X_{i_m})\mid X_i\}-\theta(i_1,\ldots,i_m)\Big].\label{eq:decomposition-pf-2.5}
\end{align}
Consider a fixed $i\in[n]$ and fixed $(i_1,\ldots,i_m)\in I_n^m$.
If $i\notin\{i_1,\ldots,i_m\}$, 
\[
E\{h(X_{i_1},\ldots,X_{i_m})\mid X_i\}-\theta(i_1,\ldots,i_m)=0~~\as.
\]
It follows that 
\begin{align}
 & \sum_{I_n^m}a(i_1,\ldots,i_m)\Big[E\{h(X_{i_1},\ldots,X_{i_m})\mid X_i\}-\theta(i_1,\ldots,i_m)\Big]\nonumber \\
  =&\sum_{I_{n-1}^{m-1}(-i)}a(i,i_1,\ldots,i_{m-1})\Big[E\{h(X_i,X_{i_1,}\ldots,X_{i_{m-1}})\mid X_i\}-\theta(i,i_1,\ldots,i_{m-1})\Big]\nonumber \\
  ~~&+\sum_{I_{n-1}^{m-1}(-i)}a(i_1,i,i_2,\ldots,i_{m-1})\Big[E\{h(X_{i_1},X_i,X_{i_2},\ldots,X_{i_{m-1}})\mid X_i\}-\theta(i_1,i,i_2,\ldots,i_{m-1})\Big]+\cdots\nonumber  \\
  ~~&+\sum_{I_{n-1}^{m-1}(-i)}a(i_1,\ldots,i_{m-1},i)\Big[E\{h(X_{i_1,}\ldots,X_{i_{m-1}},X_i)\mid X_i\}-\theta(i_1,\ldots,i_{m-1},i)\Big]\nonumber  \\
  =&\sum_{I_{n-1}^{m-1}(-i)}\sum_{l=1}^ma^{(l)}(i;i_1,\ldots,i_{m-1})\Big[E\{h^{(l)}(X_i;X_{i_1},\ldots,X_{i_{m-1}})\mid X_i\}-\theta^{(l)}(i;i_1,\ldots,i_{m-1})\Big].\label{eq:decomposition-pf-3}
\end{align}
By the definition of $h_{1,i}(\cdot)$, (\ref{eq:decomposition-pf-3})
equals $\{(n-1)!/(n-m)!\}h_{1,i}(X_i)$. Combining this with \eqref{eq:decomposition-pf-2.5}
yields (\ref{eq:decomposition-pf-1}).

Next we establish (\ref{eq:decomposition-pf-2}). The following lemma
shows that $\sum_{i=1}^n\{E(U_n\mid X_i)-E(U_n)\}$ is a U-statistic.
\begin{lemma}
\label{lem:decomposition-pf-1} We have 
\begin{align}
 & \sum_{l=1}^m\sum_{i=1}^n\sum_{I_{n-1}^{m-1}(-i)}a^{(l)}(i;i_1,\ldots,i_{m-1})\left[E\{h^{(l)}(X_i;X_{i_1},\ldots,X_{i_{m-1}})\mid X_i\}-\theta^{(l)}(i;i_1,\ldots,i_{m-1})\right]\nonumber \\
  =& \sum_{I_n^m}a(i_1,\ldots,i_m)\Big[\sum_{j=1}^mE\{h(X_{i_1},\ldots,X_{i_m})\mid X_{i_j}\}-m\theta(i_1,\ldots,i_m)\Big].\label{eq:decomposition-proofuse-lemma}
\end{align}
 
\end{lemma}
\bigskip{}

Using Lemma \ref{lem:decomposition-pf-1}, it follows from (\ref{eq:decomposition-pf-3}) that 
\[
\sum_{i=1}^n\{E(U_n\mid X_i)-E(U_n)\}=\sum_{I_n^m}a(i_1,\ldots,i_m)\Big[\sum_{j=1}^mE\{h(X_{i_1},\ldots,X_{i_m})\mid X_{i_j}\}-m\theta(i_1,\ldots,i_m)\Big].
\]
By the definition of $h_{2;i_1,\ldots,i_m}(\cdot)$, we deduce that (\ref{eq:decomposition-pf-2}) holds.

Equations (\ref{eq:Eh1i}) and (\ref{eq:Eh2i}) follow immediately from the definitions in (\ref{eq:def-h1}) and (\ref{eq:def-h2}).
This completes the proof. 
\end{proof}


\subsection{Proof of Theorem \ref{thm:bootstrap-mainterm}}
\begin{proof}
In Lemma \ref{lem:bootstrap-noniid}, let $Y_{n,i}=\sigma_n^{-1}h_{1,i}(X_i)$,
$g_n$ be the identity function, $t_n=0$ and $\sigma_n^2=1$.
By the definition of $\hat{T}_n$ we have $\hat{T}_n=n^{-1}\sum_{i=1}^n\sigma_n^{-1}h_{1,i}(X_i)$.
Equation (\ref{eq:ustat-bootstrap-condition-1}) implies (\ref{eq:bootstrap-noniid-condition1}).
(\ref{eq:ustat-bootstrap-condition-2}) implies (\ref{eq:bootstrap-noniid-condition2}).
Equations (\ref{eq:clt-proofuse-toshow1}), (\ref{eq:clt-degreem-variancetendtoone})
and Slutsky's theorem imply that for any $t\in\mathbb{R}$, 
\[
P\left\{ \hat{T}_n-t_n\leq t\right\} -\Phi(t)\to0.
\]
By Lemma \ref{lem:cdf-uniform-convergence} the above convergence
is uniform in $t\in\mathbb{R}$. This yields (\ref{eq:bootstrap-noniid-condition3}).
Therefore, all conditions in Lemma \ref{lem:bootstrap-noniid} hold,
which implies
\[
\sup_{t\in\mathbb{R}}\bigg|P^*\bigg\{\sum_{i=1}^n\frac{\{h_{1,i}(X_i)\}^*}{n\sigma_n}-\sum_{i=1}^n\frac{h_{1,i}(X_i)}{n\sigma_n}\leq t\bigg\}-P\bigg\{\sum_{i=1}^n\frac{h_{1,i}(X_i)}{n\sigma_n}\leq t\bigg\}\bigg|\stackrel{P}{\to}0.
\]
This proves (\ref{eq:ustat-mainterm-bootstrap-result}). Equation (\ref{eq:ustat-mainterm-bootstrap-result-additional}) follows immediately from Theorem \ref{thm:clt-degreem}. 
\end{proof}


\subsection{Proof of Theorem \ref{thm:bootstrap-u-stat}}
\begin{proof}
By the definition of $\sigma_n^2$, we have $\var(\sigma_n^{-1}U_n)=1$.
For proving Theorem \ref{thm:bootstrap-u-stat} it suffices to show
that 
\begin{equation}
\var^*(\sigma_n^{-1}U_n^*)\stackrel{P}{\to}1.\label{eq:bootstrap-proofuse-0.3}
\end{equation}
In Lemma \ref{lem:decomposition}, replacing $X_i$ by $X_i^*$ yields
\begin{equation}
U_n^*-E(U_n)=\frac{1}{n}\sum_{i=1}^nh_{1,i}(X_i^*)+U_n^*(a,h_2),\label{eq:bootstrap-proofuse-0.5}
\end{equation}
where
\begin{equation}
U_n^*(a,h_2):=\frac{(n-m)!}{n!}\sum_{I_n^m}a(i_1,\ldots,i_m)h_{2;i_1,\ldots,i_m}(X_{i_1}^*,\ldots,X_{i_m}^*).\label{eq:def-Unstar}
\end{equation}
Multiplying $\sigma_n^{-1}$ and then taking $\var^*$ on both
sides of (\ref{eq:bootstrap-proofuse-0.5}) yields
\begin{equation}
\var^*(\sigma_n^{-1}U_n^*)=\var^*\Big\{\sum_{i=1}^n\frac{h_{1,i}(X_i^*)}{n\sigma_n}\Big\}+\var^*\Big\{\frac{U_n^*(a,h_2)}{\sigma_n}\Big\}+\cov^*\Big\{\sum_{i=1}^n\frac{h_{1,i}(X_i^*)}{n\sigma_n},\frac{U_n^*(a,h_2)}{\sigma_n}\Big\},\label{eq:bootstrap-proofuse-0.54}
\end{equation}
where $\cov^*(\cdot)$ denotes the covariance operator on the empirical measure. By (\ref{eq:bootstrap-proofuse-0.54}) and Slutsky's theorem, for
proving (\ref{eq:bootstrap-proofuse-0.3}) it suffices to show the
following: 
\begin{align}
\var^*\Big\{\sum_{i=1}^n\frac{h_{1,i}(X_i^*)}{n\sigma_n}\Big\} & \stackrel{P}{\to}1,\label{eq:bootstrap-proofuse-0.55}\\
\var^*\Big\{\frac{U_n^*(a,h_2)}{\sigma_n}\Big\} & \stackrel{P}{\to}0,\label{eq:bootstrap-proofuse-0.6}
\end{align}
and
\begin{equation}
\cov^*\Big\{\sum_{i=1}^n\frac{h_{1,i}(X_i^*)}{n\sigma_n},\frac{U_n^*(a,h_2)}{\sigma_n}\Big\}\stackrel{P}{\to}0,\label{eq:bootstrap-proofuse-0.65}
\end{equation}

First we prove (\ref{eq:bootstrap-proofuse-0.55}).
Since conditional on $X_1,\ldots,X_n$ the $X_i^*$'s are i.i.d. draws from the empirical distribution of $X_1,\ldots,X_n$, we have 
\begin{align}
E^*\Big[\Big\{\sum_{i=1}^n\frac{h_{1,i}(X_i^*)}{n\sigma_n}\Big\}^2\Big] & =\frac{1}{n}\sum_{i=1}^n\sum_{j=1}^n\Big\{\frac{h_{1,i}(X_j)}{n\sigma_n}\Big\}^2+\frac{1}{n^2}\sum_{i_1\neq i_2}\sum_{j_1=1}^n\sum_{j_2=1}^n\frac{h_{1,i_1}(X_{j_1})}{n\sigma_n}\frac{h_{1,i_2}(X_{j_2})}{n\sigma_n},\label{eq:bootstrap-proofuse-0.7}
\end{align}
and
\begin{align}
\Big[E^*\Big\{\sum_{i=1}^n\frac{h_{1,i}(X_i^*)}{n\sigma_n}\Big\}\Big]^2 & =\frac{1}{n^2}\Big\{\sum_{i_1=1}^n\sum_{j_1=1}^n\frac{h_{1,i_1}(X_{j_1})}{n\sigma_n}\Big\}\Big\{\sum_{i_2=1}^n\sum_{j_2=1}^n\frac{h_{1,i_2}(X_{j_2})}{n\sigma_n}\Big\}\nonumber \\
=\frac{1}{n^2}\sum_{i=1}^n\Big\{\sum_{j_1=1}^n\frac{h_{1,i}(X_{j_1})}{n\sigma_n}\Big\} & \Big\{\sum_{j_2=1}^n\frac{h_{1,i}(X_{j_2})}{n\sigma_n}\Big\}+\frac{1}{n^2}\sum_{i_1\neq i_2}\sum_{j_1=1}^n\sum_{j_2=1}^n\frac{h_{1,i_1}(X_{j_1})}{n\sigma_n}\frac{h_{1,i_2}(X_{j_2})}{n\sigma_n}.\label{eq:bootstrap-proofuse-0.8}
\end{align}
Equations (\ref{eq:bootstrap-proofuse-0.7}) and (\ref{eq:bootstrap-proofuse-0.8})
yield 
\begin{align}
\var^*\Big\{\sum_{i=1}^n\frac{h_{1,i}(X_i^*)}{n\sigma_n}\Big\} & =E^*\Big[\Big\{\sum_{i=1}^n\frac{h_{1,i}(X_i^*)}{n\sigma_n}\Big\}^2\Big]-\Big[E^*\Big\{\sum_{i=1}^n\frac{h_{1,i}(X_i^*)}{n\sigma_n}\Big\}\Big]^2\nonumber  \\
 & =\frac{1}{n}\sum_{i=1}^n\sum_{j=1}^n\Big\{\frac{h_{1,i}(X_j)}{n\sigma_n}\Big\}^2 - \frac{1}{n^2}\sum_{i=1}^n\Big\{\sum_{j=1}^n\frac{h_{1,i}(X_j)}{n\sigma_n}\Big\}^2.\label{eq:bootstrap-proofuse-1}
\end{align}
Equation (\ref{eq:bootstrap-proofuse-0.55}) follows from (\ref{eq:bootstrap-proofuse-1}), (\ref{eq:ustat-bootstrap-condition-homo1}), (\ref{eq:ustat-bootstrap-condition-homo2}), and Slutsky's theorem.


The following lemma establishes (\ref{eq:bootstrap-proofuse-0.6}).
\begin{lemma} \label{lem:u-stat-bootstrap-remainder-small}
Under conditions of Theorem \ref{thm:bootstrap-u-stat}, we have $\var^*\{U_n^*(a,h_2) / \sigma_n\}\stackrel{P}{\to}0$, where $U_n^*(a,h_2)$ is defined in (\ref{eq:def-Unstar}).
\end{lemma}
\bigskip{}

Equation (\ref{eq:bootstrap-proofuse-0.65}) follows from (\ref{eq:bootstrap-proofuse-0.55}), (\ref{eq:bootstrap-proofuse-0.6}), and Cauchy-Schwarz inequality. This completes the proof.
\end{proof}


\subsection{Proof of Corollary \ref{cor:bootstrap-iid}}
\begin{proof}
For proving Corollary \ref{cor:bootstrap-iid}, by Theorem \ref{thm:bootstrap-u-stat}, it suffices to show (\ref{eq:ustat-bootstrap-condition-homo1}), (\ref{eq:ustat-bootstrap-condition-homo2}), and (\ref{eq:ustat-bootstrap-condition-3}) when the $X_i$'s are i.i.d..

First we show (\ref{eq:ustat-bootstrap-condition-homo1}). Equations
(\ref{eq:clt-degreem-condition-moment}), (\ref{eq:clt-degreem-condition-1}),
and (\ref{eq:clt-degreem-condition-2}) imply (\ref{eq:clt-degreem-variancetendtoone})
according to Theorem \ref{thm:clt-degreem}. By the i.i.d.-ness of the
$X_i$'s we have $E\{h_{1,i}(X_j)\}=E\{h_{1,i}(X_i)\}=0$ and
$E\{h_{1,i}(X_j)^2\}=E\{h_{1,i}(X_i)^2\}$. It follows from
(\ref{eq:clt-degreem-variancetendtoone}) that for any $j\in[n]$,
\begin{equation}
E\Big[\sum_{i=1}^n\Big\{\frac{h_{1,i}(X_j)}{n\sigma_n}\Big\}^2\Big]=\sum_{i=1}^n\frac{E\{h_{1,i}(X_j)^2\}}{n^2\sigma_n^2}=\frac{\sum_{i=1}^n\var\{h_{1,i}(X_i)\}}{n^2\sigma_n^2}\to1.\label{eq:bootstrap-iid-proofuse-3}
\end{equation}
By the weak law of large numbers for i.i.d. random variables, we have
\begin{equation}
\frac{1}{n}\sum_{j=1}^n\sum_{i=1}^n\Big\{\frac{h_{1,i}(X_j)}{n\sigma_n}\Big\}^2-E\Big[\sum_{i=1}^n\Big\{\frac{h_{1,i}(X_j)}{n\sigma_n}\Big\}^2\Big]\stackrel{P}{\to}0.\label{eq:bootstrap-iid-proofuse-4}
\end{equation}
Equations (\ref{eq:bootstrap-iid-proofuse-3}), (\ref{eq:bootstrap-iid-proofuse-4}),
and Slutsky's theorem yield (\ref{eq:ustat-bootstrap-condition-homo1}).

Next we prove (\ref{eq:ustat-bootstrap-condition-homo2}). By algebra
we have 
\begin{equation}
\frac{1}{n^2}\sum_{i=1}^n\Big\{\sum_{j=1}^n\frac{h_{1,i}(X_j)}{n\sigma_n}\Big\}^2=\frac{1}{n^2}\sum_{j=1}^n\sum_{i=1}^n\Big\{\frac{h_{1,i}(X_j)}{n\sigma_n}\Big\}^2+\frac{1}{n^2}\sum_{j_1\neq j_2}\sum_{i=1}^n\frac{h_{1,i}(X_{j_1})h_{1,i}(X_{j_2})}{n^2\sigma_n^2}.\label{eq:bootstrap-iid-proofuse-5}
\end{equation}
Equation (\ref{eq:ustat-bootstrap-condition-homo1}) implies that
the first term on the right-hand side of (\ref{eq:bootstrap-iid-proofuse-5})
converges to 0 in probability. The second term on the right-hand side
of (\ref{eq:bootstrap-iid-proofuse-5}) equals $(n-1)/n$ times a
U-statistic with symmetric kernel $g(x,y)=n^{-2}\sigma_n^{-2}\sum_{i=1}^nh_{1,i}(x)h_{1,i}(y)$.
By the triangle inequality, Jensen's inequality, and the i.i.d.-ness
of the $X_i$'s, we deduce 
\begin{equation}
E\left|g(X_1,X_2)\right|\leq\sum_{i=1}^nE\Big|\frac{h_{1,i}(X_1)}{n\sigma_n}\Big|\ E\Big|\frac{h_{1,i}(X_2)}{n\sigma_n}\Big|\leq\sum_{i=1}^nE\Big\{\Big(\frac{h_{1,i}(X_i)}{n\sigma_n}\Big)^2\Big\} \leq 1.\label{eq:bootstrap-iid-proofuse-6}
\end{equation}
The i.i.d.-ness of the $X_i$'s and the fact that $E\{h_{1,i}(X_i)\}=0$
yield 
\begin{equation}
E\{g(X_1,X_2)\}=n^{-2}\sigma_n^{-2}\sum_{i=1}^nE\{h_{1,i}(X_1)\}E\{h_{1,i}(X_2)\}=0.\label{eq:bootstrap-iid-proofuse-7}
\end{equation}
By (\ref{eq:bootstrap-iid-proofuse-6}) and (\ref{eq:bootstrap-iid-proofuse-7}),
it follows from the weak law of large numbers for U-statistics
of i.i.d. variables \citep[Theorem 5.4 A]{serfling2009approximation}
that the second term on the right-hand side of (\ref{eq:bootstrap-iid-proofuse-5})
converges to 0 in probability. Therefore, by Slutsky's theorem, the
left-hand side of (\ref{eq:bootstrap-iid-proofuse-5}) converges to
0 in probability, which establishes (\ref{eq:ustat-bootstrap-condition-homo2}).

Lastly, we establish (\ref{eq:ustat-bootstrap-condition-3}). By the
definition of $\theta(\cdot)$ in (\ref{eq:def-theta}), we have $\theta(i_1,\ldots,i_m)-\theta(j_1,\ldots,j_m)=0$
for any $(i_1,\ldots,i_m)$ and $(j_1,\ldots,j_m)$ in $I_n^m$.
This implies that $M_1(n)=0$. For any $p,q\in[m]$ and $\br,\bs,\bk\in I_n^m$
such that $\br\cap\bs=\bk\cap\bs=r_p=s_q=k_p$, by the i.i.d.-ness
of the $X_i$'s, we have
\begin{align}
E\Big[E\Big\{ & h(X_{r_1},\ldots,X_{r_m})h(X_{s_1}\ldots X_{s_m})\mid X_{k_p}\Big\}\Big]\nonumber \\
 & =E\Big[E\Big\{ h(X_1,\ldots,X_m)h(X_{m+1},\ldots,X_{m+q-1},X_p,X_{m+q},\ldots,X_{2m-1})\mid X_p\Big\}\Big],\label{eq:bootstrap-iid-proofuse-1}
\end{align}
and
\begin{align}
E\Big[E\Big\{ & h(X_{k_1},\ldots,X_{k_m})h(X_{s_1},\ldots,X_{s_m})\mid X_{k_p}\Big\}\Big]\nonumber  \\
 & =E\Big[E\Big\{ h(X_1,\ldots,X_m)h(X_{m+1},\ldots,X_{m+q-1},X_p,X_{m+q},\ldots,X_{2m-1})\mid X_p\Big\}\Big].\label{eq:bootstrap-iid-proofuse-2}
\end{align}
Equations (\ref{eq:bootstrap-iid-proofuse-1}) and (\ref{eq:bootstrap-iid-proofuse-2})  imply that $M_2(n)=0$. Therefore, (\ref{eq:ustat-bootstrap-condition-3}) follows from the fact that $M_1(n)=M_2(n)=0$ and the assumption that $n^{-2}\sigma_n^{-2}A_{2,1}(n)\to0$.
\end{proof}


\subsection{Proof of Theorem \ref{thm:bootstrap-resampling}}

\begin{proof}
By the definition of $V_n^*$, we have
\begin{align*}
V^*_n & =\frac{1}{h_n(n-b+1)}\sum_{i=1}^{n-b+1}\var^*(U_{b,i}^*) \\
& = \frac{1}{h_n(n-b+1)}\sum_{i=1}^{n-b+1} \sigma^2_b (1+o_P(1)) \\
& = \frac{1}{h_n}\sigma^2_b (1+o_P(1)),
\end{align*}
where the second equality follows from the assumption $\var(U_{b,i}(X_i,\ldots,X_{i+b-1}))=\sigma_b^2(1+o(1))$ and Theorem \ref{thm:bootstrap-u-stat}.
This combines with the assumption $\sigma_b^2/\sigma_n^2= \zeta_{n,b}\cdot(1+o(1))$ gives the desired result.
\end{proof}


\subsection{Proof of Lemma \ref{lem:decomposition-taukendall}}
\begin{proof}
For $U_n^{\ken}$ we have $a(i,j)=\ind(j<i)$ and $h(X_i,X_j)=\ind(X_j>X_i)$.
Using definitions in (\ref{eq:def-theta}) and (\ref{eq:def-fl}),
we have $f_i^{(1)}(x)=E\{h(x,X_i)\}=P(X_i>x)$, $f_i^{(2)}(x)=E\{h(X_i,x)\}=1-P(X_i>x)$,
and $\theta(i,j)=1-\theta(j,i)$. By Lemma \ref{lem:decomposition}
we obtain 
\begin{align*}
h_{1,i}^{\ken}(x) & =\frac{1}{n-1}\sum_{\substack{j=1\\
j\neq i
}
}^na(i,j)\{f_j^{(1)}(x)-\theta(i,j)\}+a(j,i)\{f_j^{(2)}(x)-\theta(j,i)\}\\
 & =\frac{1}{n-1}\sum_{j=1}^n\{\ind(j<i)-\ind(j>i)\}\{P(X_j>x)-\theta(i,j)\},
\end{align*}
and
\[
h_{2;i,j}^{\ken}(x,y)=h(x,y)-f_j^{(1)}(x)-f_i^{(2)}(y)+\theta(i,j)=\ind(y>x)-P(X_j>x)-P(y>X_i)+\theta(i,j).
\]
This completes the proof.
\end{proof}


\subsection{Proof of Lemma \ref{eq:tauap-h1i}}
\begin{proof}
For $U_n^{\ap}$, we have $a(i,j)=n(i-1)^{-1}\ind(j<i)$ and $h(X_i,X_j)=\ind(X_j>X_i)$.
The form of $f_i^{(1)}(x)$ and $f_i^{(2)}(x)$ is the same as in
the proof of Lemma \ref{lem:decomposition-taukendall}. By Lemma \ref{lem:decomposition}
we obtain
\begin{align*}
h_{1,i}^{\ap}(x) & =\frac{1}{n-1}\sum_{\substack{j=1\\
j\neq i
}
}^na(i,j)\{f_j^{(1)}(x)-\theta(i,j)\}+a(j,i)\{f_j^{(2)}(x)-\theta(j,i)\}\\
 & =\frac{1}{n-1}\sum_{j=1}^n\Big\{\frac{n\ind(j<i)}{i-1}-\frac{n\ind(j>i)}{j-1}\Big\}\{P(X_j>x)-\theta(i,j)\},
\end{align*}
and
\[
h_{2;i,j}^{AP}(x,y)=h(x,y)-f_j^{(1)}(x)-f_i^{(2)}(y)+\theta(i,j)=\ind(y>x)-P(X_j>x)-P(y>X_i)+\theta(i,j).
\]
This completes the proof.
\end{proof}


\subsection{Proof of Theorem \ref{thm:tauap-taukendall-AN}}
\begin{proof}
We divide the proof into two parts. In Part I, we prove the theorem for $U_n^{\ap}$. In Part II, we prove the theorem for $U_n^{\ken}$.

\medskip{}

\textbf{Part I (for $U_n^{\ap}$).}
By Theorem \ref{thm:clt-degreem}, for proving asymptotic normality of $U_n^{\ap}$, it suffices to show that (\ref{eq:clt-degreem-condition-moment}), (\ref{eq:clt-degreem-condition-1}), and (\ref{eq:clt-degreem-condition-2}) hold under the assumption of Theorem \ref{lem:decomposition-taukendall}. Equation (\ref{eq:clt-degreem-condition-moment}) holds trivially with $M(n)=1$ due to boundedness of the kernel function $h(\cdot)$. In the following, we establish (\ref{eq:clt-degreem-condition-1}) and (\ref{eq:clt-degreem-condition-2}) by calculating the orders of $A_{2,2}(n)$, $A_{3,1}(n)$, and $V(n)$.

First we derive upper bound on $A_{2,2}(n)$ and $A_{3,1}(n)$. We will
repeatedly use Lemma \ref{lem:bound-harmonic-sum} to bound the partial
sum of harmonic series. By the definition of $A_{2,2}(n)$ in (\ref{eq:def-A2}),
we have 
\begin{equation}
A_{2,2}(n):= \frac{1}{n^2} \sum_{(I_n^2)_{\geq2}^{\otimes2}}|a(i_1,j_1)a(i_2,j_2)| = \frac{1}{n^2}\sum_{(i,j)\in I_n^2}\left|a(i,j)^2+a(i,j)a(j,i)\right|.\label{eq:A2-generalform}
\end{equation}
Since $a(i,j)=n(i-1)^{-1}\ind(j<i)$, we have $a(i,j)a(j,i)=0$
and $a(i,i)=0$. It then follows from (\ref{eq:A2-generalform}) that
\begin{align}
A_{2,2}(n) & = \frac{1}{n^2} \sum_{i=2}^n\sum_{j=1}^{i-1}(\frac{n}{i-1})^2=\sum_{i=2}^n\frac{1}{i-1}\leq1+\log(n-1).\label{eq:tauap-A2}
\end{align}
By the definition of $A_{3,1}(n)$ in (\ref{eq:def-A2}), we have
\begin{align}
A_{3,1}(n) & = \frac{1}{n^4} \sum_{i=1}^n\sum_{j_1,j_2,j_{3}=1}^n\Big\{|a(i,j_1)a(i,j_2)a(i,j_{3})|+3|a(i,j_1)a(i,j_2)a(j_{3},i)|\nonumber  \\
 & \ \ \ \ \ \ \ \ \ \ +3|a(i,j_1)a(j_2,i)a(j_{3},i)|+|a(j_1,i)a(j_2,i)a(j_{3},i)|\Big\}.\label{eq:A3-generalform}
\end{align}
The term $|a(i,j_1)a(i,j_2)a(i,j_{3})|$ is nonzero only if $j_1,j_2,j_{3}<i$,
so the corresponding summation in (\ref{eq:A3-generalform}) equals
\begin{align}
\frac{1}{n^4} \sum_{i=2}^n\sum_{j_1,j_2,j_{3}=1}^{i-1}\frac{n}{i-1}\cdot\frac{n}{i-1}\cdot\frac{n}{i-1} & \leq \frac{1}{n} \sum_{i=1}^{n-1}(i-1)^{3}(\frac{1}{i-1})^{3}=\frac{n-1}{n}.\label{eq:calculate-A3-tauap-1}
\end{align}
The term $|a(i,j_1)a(i,j_2)a(j_{3},i)|$ is nonzero only if $j_1,j_2<i<j_{3}$,
so the corresponding summation in (\ref{eq:A3-generalform}) equals
\begin{align}
\frac{3}{n^4} \sum_{i=2}^{n-1}\sum_{j_1,j_2=1}^{i-1}\sum_{j_{3}=i+1}^n(\frac{n}{i-1})^2(\frac{n}{j_{3}-1}) & = \frac{3}{n} \sum_{i=2}^{n-1}\sum_{j_{3}=i+1}^n\frac{1}{j_{3}-1}\leq \frac{3}{n} \sum_{i=2}^{n-1}\log\frac{n-1}{i-1}\leq 3\log n.\label{eq:calculate-A3-tauap-2}
\end{align}
The term $|a(i,j_1)a(j_2,i)a(j_{3},i)|$ is nonzero only if $j_1<i<j_2,j_{3}$,
so the corresponding summation in (\ref{eq:A3-generalform}) equals
\begin{align}
\frac{3}{n^4} \sum_{i=2}^{n-1}\sum_{j_1=1}^{i-1}\sum_{j_2,j_{3}=i+1}^n(\frac{n}{i-1})(\frac{n}{j_2-1})(\frac{n}{j_{3}-1}) & \leq \frac{3}{n}\sum_{i=2}^{n-1}\left(\log\frac{n-1}{i-1}\right)^2\leq 3 (\log n)^2.\label{eq:calculate-A3-tauap-3}
\end{align}
The term $|a(j_1,i)a(j_2,i)a(j_{3},i)|$ is nonzero only if $j_1,j_2,j_{3}>i$,
so the corresponding summation in (\ref{eq:A3-generalform}) equals
\begin{align}
\frac{1}{n^4} \sum_{i=1}^{n-1}\sum_{j_1,j_2,j_{3}=i+1}^n\frac{n}{j_1-1}\cdot\frac{n}{j_2-1}\cdot\frac{n}{j_{3}-1} & \leq \frac{1}{n}\sum_{i=1}^{n-1}\left(\log\frac{n-1}{i-1}\right)^{3}\leq(\log n)^{3}.\label{eq:calculate-A3-tauap-4}
\end{align}
By (\ref{eq:calculate-A3-tauap-1})-(\ref{eq:calculate-A3-tauap-4}),
it follows from (\ref{eq:A3-generalform}) that
\begin{equation}
A_{3,1}(n)\leq C(\log n)^{3}.\label{eq:tauap-A3}
\end{equation}

Next we establish lower bound on $V(n):=n^{-2}\sum_{i=1}^n\var\{h_{1,i}^{\ap}(X_i)\}$.
The following lemma gives lower bound on $|h_{1,i}^{\ap}(X_i)|$.
\begin{lemma} \label{lem:tauap-AN-proofuse} 
Consider a fixed $i$ with $2\leq i\leq n$.
If $\delta_n/2\geq\log\{(n-1)/(i-1)\}$, either Condition (i) or
Condition (ii) in Theorem \ref{thm:tauap-taukendall-AN} implies
\begin{equation}
P\{|h_{1,i}^{\ap}(X_i)|\geq\delta_n/2\}\geq p_n.\label{eq:lemma-bound-h1i-tauap-1}
\end{equation}
If $\delta_n\log(n/i)\geq2$, either Condition (i) or Condition
(ii) in Theorem \ref{thm:tauap-taukendall-AN} implies
\begin{equation}
P\{|h_{1,i}^{\ap}(X_i)|\geq1\}\geq p_n.\label{eq:lemma-bound-h1i-tauap-2}
\end{equation}

\end{lemma}
\bigskip{}

If $i\geq1+(n-1)\exp(-\delta_n/2)$, we have $\delta_n/2\geq\log\{(n-1)/(i-1)\}$.
Lemma \ref{lem:tauap-AN-proofuse} implies that (\ref{eq:lemma-bound-h1i-tauap-1})
holds. By Chebyshev's inequality we deduce 
\begin{align}
\var\{h_{1,i}^{\ap}(X_i)\} & \geq\frac{1}{4}\delta_n^2p_n.\label{eq:bound-tauap-var-1}
\end{align}
If $2\leq i\leq n\exp(-2/\delta_n)$, we have $\delta_n\log(n/i)\geq2.$
Lemma \ref{lem:tauap-AN-proofuse} implies that (\ref{eq:lemma-bound-h1i-tauap-2})
holds. By Chebyshev's inequality we deduce 
\begin{align}
\var\{h_{1,i}^{\ap}(X_i)\} & \geq p_n.\label{eq:bound-tauap-var-2}
\end{align}
By (\ref{eq:bound-tauap-var-1}) and (\ref{eq:bound-tauap-var-2}),
we have 
\begin{align}
\sum_{i=1}^n \var\{ & h_{1,i}^{\ap}(X_i)\}  \geq \sum_{i=2}^{\lfloor n\exp(-\frac{2}{\delta_n})\rfloor}p_n + \sum_{i=\lfloor1+(n-1)\exp(-\frac{\delta_n}{2})\rfloor+1}^n\frac{1}{4}\delta_n^2p_n \nonumber  \\
& \geq \Big\{n\exp  \Big(-\frac{2}{\delta_n}\Big)-2 \Big\} p_n+\frac{1}{4} \Big\{n-(n-1)\exp\Big(-\frac{\delta_n}{2}\Big)-1 \Big\}\delta_n^2p_n\nonumber  \\
& =n\exp\Big(  -\frac{2}{\delta_n}\Big)p_n+ \frac{n \delta_n^2 p_n}{4} \Big\{1-\exp\Big(-\frac{\delta_n}{2}\Big) \Big\} + \frac{\delta_n^2 p_n}{4} \Big\{ \exp\Big(-\frac{\delta_n}{2}\Big) - 1 \Big\} - 2p_n.
\label{eq:bound-tauap-var-3}
\end{align}
By (\ref{eq:tauap-AN-rateassumption}) we have 
\begin{equation}
n\delta_n^2p_n \Big\{1-\exp\Big(-\frac{\delta_n}{2}\Big) \Big\} \asymp n\delta_n^{3}p_n\gnsim n^{2/3}(\log n)^2.\label{eq:bound-tauap-var-3.4}
\end{equation}
Note that
\begin{equation}
n\exp\Big(-\frac{2}{\delta_n}\Big)p_n\geq0~~{\rm and}~~ \frac{\delta_n^2 p_n}{4} \Big\{ \exp\Big(-\frac{\delta_n}{2}\Big) - 1 \Big\} - 2p_n = O(1).\label{eq:bound-tauap-var-3.6}
\end{equation}
Combining (\ref{eq:bound-tauap-var-3}) with (\ref{eq:bound-tauap-var-3.4}) and (\ref{eq:bound-tauap-var-3.6}) gives
\[
\sum_{i=1}^n\var\{h_{1,i}^{\ap}(X_i)\}\gnsim n^{2/3}(\log n)^2.
\]
This implies
\begin{equation}
V(n):= \frac{1}{n^2} \sum_{i=1}^n\var\{h_{1,i}^{\ap}(X_i)\}\gnsim n^{-4/3}(\log n)^2.\label{eq:tauap-Vn}
\end{equation}

Equations (\ref{eq:tauap-A2}), (\ref{eq:tauap-A3}), and (\ref{eq:tauap-Vn}) yield (\ref{eq:clt-degreem-condition-1}) and (\ref{eq:clt-degreem-condition-2}).
The asymptotic normality of $U_n^{\ap}$ then follows from Theorem \ref{thm:clt-degreem}. This completes the proof for Part I.



\textbf{Part II (for $U_n^{\ken}$).}
Proof for $U_n^{\ken}$ follows the same logic as the proof for $U_n^{\ap}$. In the following we calculate the
orders of $A_{2,2}(n)$, $A_{3,1}(n)$, and $V(n)$ for $U_n^{\ken}$.

Since $a(i,j)=\ind(j<i)$ for $U_n^{\ken}$, we have $a(i,j)a(j,i)=0$
and $a(i,i)=0$. It then follows from (\ref{eq:A2-generalform}) that
\begin{align}
A_{2,2}(n) & = \frac{1}{n^2} \sum_{i=2}^n\sum_{j=1}^{i-1}1=O(1).\label{eq:taukendall-A2}
\end{align}
By (\ref{eq:A3-generalform}), following the same argument as in (\ref{eq:calculate-A3-tauap-1})-(\ref{eq:calculate-A3-tauap-4})
we deduce
\begin{equation}
A_{3,1}(n)=O(1).\label{eq:taukendall-A3}
\end{equation}

Next we establish lower bound on $V(n):=n^{-2}\sum_{i=1}^n\var\{h_{1,i}^{\ken}(X_i)\}$.
The following lemma gives lower bound on $|h_{1,i}^{\ken}(X_i)|$.
\begin{lemma}
\label{lem:taukendall-AN-proofuse} Consider a fixed $i\in[n]$. If
$n-i\leq(i-1)\delta_n/2$, either Condition (i) or Condition (ii)
in Theorem \ref{thm:tauap-taukendall-AN} implies
\begin{equation}
P\Big\{|h_{1,i}^{\ken}(X_i)|\geq\frac{i-1}{n-1}\frac{\delta_n}{2}\Big\}\geq p_n.\label{eq:lemma-bound-h1i-taukendall-1}
\end{equation}
If $i-1\leq(n-i)\delta_n/2$, either Condition (i) or Condition
(ii) in Theorem \ref{thm:tauap-taukendall-AN} implies
\begin{equation}
P\Big\{|h_{1,i}^{\ken}(X_i)|\geq\frac{n-i}{n-1}\frac{\delta_n}{2}\Big\}\geq p_n.\label{eq:lemma-bound-h1i-taukendall-2}
\end{equation}

\end{lemma}
\bigskip{}

If $i\geq(2n-\delta_n)/(2+\delta_n)$, we have $n-i\leq(i-1)\delta_n/2$
and $(i-1)/(n-1)\geq2/(\delta_n+2)$. Lemma \ref{lem:taukendall-AN-proofuse}
implies that (\ref{eq:lemma-bound-h1i-taukendall-1}) holds. By Chebyshev's
inequality we deduce 
\begin{align}
\var\{h_{1,i}^{\ken}(X_i)\} & \geq\{\frac{i-1}{n-1}\frac{\delta_n}{2}\}^2p_n\geq\frac{4}{(2+\delta_n)^2}\delta_n^2p_n.\label{eq:bound-taukendall-var-1}
\end{align}
If $i\leq(n\delta_n+2)/(2+\delta_n)$, we have $i-1\leq(n-i)\delta_n/2$
and $(n-i)/(n-1)\geq2/(2+\delta_n)$. Lemma \ref{lem:taukendall-AN-proofuse}
implies that (\ref{eq:lemma-bound-h1i-taukendall-2}) holds. By Chebyshev's
inequality we deduce 
\begin{align}
\var\{h_{1,i}^{\ken}(X_i)\} & \geq\{\frac{n-i}{n-1}\frac{\delta_n}{2}\}^2p_n\geq\frac{1}{(2+\delta_n)^2}\delta_n^2p_n.\label{eq:bound-taukendall-var-2}
\end{align}
By (\ref{eq:bound-taukendall-var-1}) and (\ref{eq:bound-taukendall-var-2}),
we have
\begin{align}
\sum_{i=1}^n\var\{h_{1,i}^{\ken}(X_i)\} & \geq\sum_{i=1}^{\lfloor(n\delta_n+2)/(2+\delta_n)\rfloor}\frac{1}{(2+\delta_n)^2}\delta_n^2p_n + \sum_{i=\lfloor(2n-\delta_n)/(2+\delta_n)\rfloor+1}^n\frac{4}{(2+\delta_n)^2}\delta_n^2p_n.\label{eq:bound-taukendall-var-2.5}
\end{align}
Note that
\begin{equation}
\sum_{i=1}^{\lfloor(n\delta_n+2)/(2+\delta_n)\rfloor}\frac{1}{(2+\delta_n)^2}\delta_n^2p_n=\frac{(n\delta_n+2)/(2+\delta_n)-1}{(2+\delta_n)^2}\delta_n^2p_n\asymp n\delta_n^{3}p_n.\label{eq:bound-taukendall-var-2.6}
\end{equation}
Combining (\ref{eq:bound-taukendall-var-2.5}) and (\ref{eq:bound-taukendall-var-2.6})
yields
\begin{equation}
\sum_{i=1}^n\var\{h_{1,i}^{\ken}(X_i)\}\gtrsim n\delta_n^{3}p_n.\label{eq:bound-taukendall-var-3}
\end{equation}
It follows from (\ref{eq:taukendall-AN-rateassumption}) and (\ref{eq:bound-taukendall-var-3})
that
\begin{equation}
V(n):= \frac{1}{n^2} \sum_{i=1}^n\var\{h_{1,i}^{\ken}(X_i)\}\gnsim n^{-4/3}.\label{eq:taukendall-Vn}
\end{equation}

Equations (\ref{eq:taukendall-A2}), (\ref{eq:taukendall-A3}), and
(\ref{eq:taukendall-Vn}) yield (\ref{eq:clt-degreem-condition-1})
and (\ref{eq:clt-degreem-condition-2}). The asymptotic normality
of $U_n^{\ken}$ then follows from Theorem \ref{thm:clt-degreem}.
This completes the proof for Part II.
\end{proof}


\subsection{Proof of Theorem \ref{thm:tauap-AN-tail-conditions}}
\begin{proof}
Define
\begin{equation}
f_{ij}(x):=P(X_j>x)-P(X_j>X_i),\label{eq:tail-condition-proofuse-def-fij}
\end{equation}
and 
\begin{equation}
z_i=z_i(x) := (x-\mu_i)/\sigma_i.\label{eq:tail-condition-proofuse-def-zi}
\end{equation}
Using the definitions of $F_j^{c}$ and $F_{ji}^{c}$ in (\ref{eq:def-FjFij}),
we have 
\begin{align}
f_{ij}(x) & =P\Big\{\frac{X_j-\mu_j}{\sigma_j}>\frac{(x-\mu_i)+(\mu_i-\mu_j)}{\sigma_i}\cdot\frac{\sigma_i}{\sigma_j}\Big\}-P\Big\{\frac{X_j-X_i-(\mu_j-\mu_i)}{(\sigma_i^2+\sigma_j^2)^{1/2}}>\frac{\mu_i-\mu_j}{\sigma_i}\cdot\frac{\sigma_i}{(\sigma_i^2+\sigma_j^2)^{1/2}}\Big\}\nonumber  \\
 & =F_j^{c}\{\rho_{ij}(z_i+r_{ij})\}-F_{ji}^{c}\{r_{ij}(1+\rho_{ij}^{-2})^{-1/2}\}.\label{eq:tail-condition-proofuse-fijFjFji}
\end{align}
For proving Theorem \ref{thm:tauap-AN-tail-conditions}, it suffices
to show the existence of $\delta_n$ and $p_n$ satisfying the
conditions in Theorem \ref{thm:tauap-taukendall-AN}. Because the
proofs for $U_n^{\ken}$ and $U_n^{\ap}$ are almost identical,
we give detailed proof for $U_n^{\ken}$ and comment on the proof for $U_n^{\ap}$ at
the end. We divide the proof for $U_n^{\ken}$ into two parts. In Part I we construct
such $\delta_n$ and $p_n$ under conditions (\ref{eq:tail-condition-1-1}),
(\ref{eq:tail-condition-1-2}), and (\ref{eq:tail-condition-1-rate-kendall}).
In Part II we construct such $\delta_n$ and $p_n$ under conditions
(\ref{eq:tail-condition-2-1}), (\ref{eq:tail-condition-2-2}), and
(\ref{eq:tail-condition-2-rate-kendall}).

\medskip{}

\textbf{Part I:} Assume (\ref{eq:tail-condition-1-1}), (\ref{eq:tail-condition-1-2}), and (\ref{eq:tail-condition-1-rate-kendall}) hold. The following lemma gives bound on $f_{ij}(x)$.
\begin{lemma} \label{lem:tail-condition-lemma-1}
Define
\[
K_1=t_0+ \Big(t_0^{-b_1}\frac{c_1}{2c_2} \Big)^{-1/b_2}~~{\rm and}~~K_2=\Big( \frac{c_1}{2c_2} \Big)^{-1/b_2}.
\]
Consider a fixed $i\in[n]$. If $x$ satisfies
\begin{equation}
z_i(x) \geq R_n+K_1\rho_n+K_2\rho_nR_n^{b_1/b_2}, \label{eq:tail-condition-lemma-1-zi}
\end{equation}
then for all $j\in[n]\backslash\{i\}$ we have 
\begin{equation}
f_{ij}(x)\leq-\min\Big\{\frac{c_1}{2}R_n^{-b_1}, \frac{c_1}{2}t_0^{-b_1}, \frac{1}{2}\Big\}.\label{eq:tail-condition-lemma-1-fij}
\end{equation}
 
\end{lemma}
\bigskip{}

Define $\delta_n:=\min \{ \frac{c_1}{2}R_n^{-b_1},  \frac{c_1}{2}t_0^{-b_1}, \frac{1}{2} \}.$,
$Z_i:=(X_i-\mu_i)/\sigma_i$, and
\begin{equation}
p_n:=P\{Z_i\geq R_n+K_1\rho_n+K_2\rho_nR_n^{b_1/b_2}\}.\label{eq:tail-condition-proofuse-2}
\end{equation}
Lemma \ref{lem:tail-condition-lemma-1} yields that
\[
P\{f_{ij}(x)\leq-\delta_n,\forall j\in[n]\backslash\{i\}\}\geq p_n.
\]
Since $\rho_n\geq1$, by the definition of $K_1$ we have
\begin{equation}
R_n+K_1\rho_n+K_2\rho_nR_n^{b_1/b_2}\geq K_1\rho_n\geq t_0.\label{eq:tail-condition-proofuse-4}
\end{equation}
Combining (\ref{eq:tail-condition-proofuse-2}), (\ref{eq:tail-condition-proofuse-4})
and (\ref{eq:tail-condition-1-1}) yields
\[
p_n\geq c_1(R_n+K_1\rho_n+K_2\rho_nR_n^{b_1/b_2})^{-b_1}.
\]
Thus by dropping constants we obtain
\begin{equation}
\delta_n^{3}p_n\gtrsim(R_n+\rho_n+\rho_nR_n^{b_1/b_2})^{-b_1}\min(R_n^{-3b_1},1).\label{eq:tail-condition-proofuse-6}
\end{equation}

In the following we show that (\ref{eq:tail-condition-proofuse-6})
and (\ref{eq:tail-condition-1-rate-kendall}) imply
\begin{equation}
\delta_n^{3}p_n\gnsim n^{-1/3}.\label{eq:tail-condition-proofuse-7}
\end{equation}
 If $\lim\sup_{n\to\infty}R_n=\infty$, the fact that $\rho_n\geq1$
and $b_1>b_2>0$ yields
\begin{equation}
(R_n+\rho_n+\rho_nR_n^{b_1/b_2})^{-b_1}\asymp\rho_n^{-b_1}R_n^{-b_1^2/b_2} \label{eq:tail-condition-proofuse-8}
\end{equation}
and
\begin{equation}
\min(R_n^{-3b_1},1)\asymp R_n^{-3b_1}.\label{eq:tail-condition-proofuse-9}
\end{equation}
Equation (\ref{eq:tail-condition-proofuse-6}) together with (\ref{eq:tail-condition-proofuse-8}) and (\ref{eq:tail-condition-proofuse-9}) gives 
\begin{equation}
\delta_n^{3}p_n\gtrsim\rho_n^{-b_1}R_n^{-b_1^2/b_2}R_n^{-3b_1}.\label{eq:tail-condition-proofuse-10}
\end{equation}
By (\ref{eq:tail-condition-proofuse-10}) and (\ref{eq:tail-condition-1-rate-kendall}), we deduce (\ref{eq:tail-condition-proofuse-7}). If $\lim\sup_{n\to\infty}R_n<\infty$, by (\ref{eq:tail-condition-proofuse-6}) we have
\begin{equation}
\delta_n^{3}p_n\gtrsim\rho_n^{-b_1}.\label{eq:tail-condition-proofuse-11}
\end{equation}
Equation (\ref{eq:tail-condition-1-rate-kendall}) implies
\begin{equation}
\rho_n^{-b_1}\gnsim n^{-1/3}.\label{eq:tail-condition-proofuse-12}
\end{equation}
Combining (\ref{eq:tail-condition-proofuse-11}) and (\ref{eq:tail-condition-proofuse-12})
yields (\ref{eq:tail-condition-proofuse-7}). Therefore, the asymptotic
normality of $U_n^{\ken}$ follows from Theorem \ref{thm:tauap-taukendall-AN}.

This completes the proof of Part I for $U_n^{\ken}$.
For $U_n^{\ap}$ the proof is almost the same, except that (\ref{eq:tail-condition-1-rate-kendall}) is replaced by (\ref{eq:tail-condition-1-rate-ap}), and the right-hand side of (\ref{eq:tail-condition-proofuse-7}) and (\ref{eq:tail-condition-proofuse-12}) is replaced by $n^{-1/3}(\log n)^2$.

\medskip{}

\textbf{Part II:} Assume (\ref{eq:tail-condition-1-rate-ap}), (\ref{eq:tail-condition-2-1}), and (\ref{eq:tail-condition-2-2}) hold. The following lemma
gives bound on $f_{ij}(x)$.
\begin{lemma}
\label{lem:tail-condition-lemma-2} Recall that $z_i=z_i(x) := (x-\mu_i)/\sigma_i$. For a fixed $i\in[n]$, assume
that 
\begin{equation}
z_i\geq R_n+K_{3}\rho_n+K_{4}\rho_n R_n,\label{eq:tail-condition-lemma-2-zi}
\end{equation}
where $K_{3},K_{4}$ are defind in (\ref{eq:tail-conditiondef-K3K4}).
Then for all $j\in[n]\backslash\{i\}$ we have 
\begin{equation}
f_{ij}(x)\leq-\min\Big\{ \frac{c_1}{2}\exp(-b_1 R_n^{\lambda}), \frac{c_1}{2}\exp(-b_1 t_0^{\lambda}), \frac{1}{2} \Big\}.\label{eq:tail-condition-lemma-2-fij}
\end{equation}
\end{lemma}
\bigskip{}

Define $\delta_n=\min\Big\{ \frac{c_1}{2}\exp(-b_1 R_n^{\lambda}), \frac{c_1}{2}\exp(-b_1 t_0^{\lambda}), \frac{1}{2} \Big\}$,
$Z_i=(X_i-\mu_i)/\sigma_i$, and
\begin{equation}
p_n=P\{Z_i\geq R_n+K_{3}\rho_n+K_{4}\rho_nR_n\}.\label{eq:tail-condition-proofuse-14}
\end{equation}
Lemma \ref{lem:tail-condition-lemma-2} yields that
\[
P\{f_{ij}(x)\leq-\delta_n,\forall j\in[n]\backslash\{i\}\}\geq p_n.
\]
Since $\rho_n\geq1$, by the definition of $K_{3}$, we have
\begin{equation}
R_n+K_{3}\rho_n+K_{4}\rho_nR_n\geq K_{3}\rho_n\geq t_0.\label{eq:tail-condition-proofuse-16}
\end{equation}
Combining (\ref{eq:tail-condition-proofuse-14}), (\ref{eq:tail-condition-proofuse-16}), and (\ref{eq:tail-condition-2-1}) yields
\[p_n\geq c_1\exp\{-b_1(R_n+K_{3}\rho_n+K_{4}\rho_nR_n)^{\lambda}\}.
\]
Thus by dropping constants we obtain
\begin{align}
\delta_n^{3}p_n & \gtrsim\exp\{-b_1(R_n+K_{3}\rho_n+K_{4}\rho_nR_n)^{\lambda}\}\min\{\exp(-3b_1 R_n^{\lambda}),1\}.\nonumber  \\
 & \asymp\min\Big[\exp\{-3b_1 R_n^{\lambda}-b_1(R_n+K_{3}\rho_n+K_{4}\rho_n R_n)^{\lambda}\},\exp\{-b_1(R_n+K_{3}\rho_n+K_{4}\rho_n R_n)^{\lambda}\}\Big]\label{eq:tail-condition-proofuse-18}
\end{align}
With an argument similar to (\ref{eq:tail-condition-proofuse-8})-(\ref{eq:tail-condition-proofuse-12}), it follows from (\ref{eq:tail-condition-proofuse-18}) and (\ref{eq:tail-condition-2-rate-kendall}) that
\begin{equation}
\delta_n^{3}p_n\gnsim n^{-1/3}.\label{eq:tail-condition-proofuse-19}
\end{equation}

This completes the proof of Part II for $U_n^{\ken}$. For $U_n^{\ap}$
the proof is almost the same, except that (\ref{eq:tail-condition-2-rate-kendall})
is replaced by (\ref{eq:tail-condition-2-rate-ap}), and the right-hand
side of (\ref{eq:tail-condition-proofuse-19}) is replaced by $n^{-1/3}(\log n)^2$.
\end{proof}


\subsection{Proof of Corollary \ref{cor:bootstrap-tauap-taukendall-mainterm}}
\begin{proof}
We divide the proof into two parts. In Part I, we show that (\ref{eq:ustat-bootstrap-condition-1})
and (\ref{eq:ustat-bootstrap-condition-2}) hold for $h_{1,i}^{\ken}$. In Part II, we show that (\ref{eq:ustat-bootstrap-condition-1}) and (\ref{eq:ustat-bootstrap-condition-2}) hold for $h_{1,i}^{\ap}$.

\textbf{Part I.} By (\ref{eq:taukendall-AN-rateassumption}) and Theorem
\ref{thm:tauap-taukendall-AN}, we have that (\ref{eq:clt-degreem-variancetendtoone})
holds. This combined with (\ref{eq:taukendall-Vn}) gives
\begin{equation}
n\sigma_n\gnsim n^{1/3},\label{eq:taukendall-mainterm-bootstrap-proofuse-1}
\end{equation}
where $\sigma_n^2:=\var(U_n^{\ken})$. By (\ref{eq:taukendall-h1i}),
we have $|h_{1,i}^{\ken}(x)|\leq1$ for any $x$. It then follows
from Markov's inequality that for any $\epsilon>0$,
\begin{equation}
P\Big\{\Big|\frac{h_{1,i}^{\ken}(X_i)}{n\sigma_n}\Big|\geq\epsilon\Big\}\leq\frac{E|h_{1,i}^{\ken}(X_i)|}{\epsilon n\sigma_n}\leq\frac{1}{\epsilon n\sigma_n}.\label{eq:taukendall-mainterm-bootstrap-proofuse-2}
\end{equation}
Taking $\sup_{1\leq i\leq n}$ on both sides of (\ref{eq:taukendall-mainterm-bootstrap-proofuse-2}),
we deduce (\ref{eq:ustat-bootstrap-condition-1}) from (\ref{eq:taukendall-mainterm-bootstrap-proofuse-1}).

By (\ref{eq:Eh1i}) we have
\begin{equation}
E\Big\{\frac{h_{1,i}^{\ken}(X_i)}{n\sigma_n}1\Big(\Big|\frac{h_{1,i}^{\ken}(X_i)}{n\sigma_n}\Big|\leq\epsilon\Big)\Big\}=-E\Big\{\frac{h_{1,i}^{\ken}(X_i)}{n\sigma_n}1\Big(\Big|\frac{h_{1,i}^{\ken}(X_i)}{n\sigma_n}\Big|>\epsilon\Big)\Big\}.\label{eq:taukendall-mainterm-bootstrap-proofuse-3}
\end{equation}
Cauchy-Schwarz inequality gives
\begin{equation}
\Big|E\Big\{\frac{h_{1,i}^{\ken}(X_i)}{n\sigma_n}1\Big(\Big|\frac{h_{1,i}^{\ken}(X_i)}{n\sigma_n}\Big|>\epsilon\Big)\Big\}\Big|\leq\Big[E\Big\{\frac{h_{1,i}^{\ken}(X_i)}{n\sigma_n}\Big\}^2\Big]^{1/2}P\Big(\Big|\frac{h_{1,i}^{\ken}(X_i)}{n\sigma_n}\Big|>\epsilon\Big)^{1/2}.\label{eq:taukendall-mainterm-bootstrap-proofuse-4}
\end{equation}
Combining (\ref{eq:taukendall-mainterm-bootstrap-proofuse-3}) and
(\ref{eq:taukendall-mainterm-bootstrap-proofuse-4}) yields
\begin{equation}
\Big[E\Big\{\frac{h_{1,i}^{\ken}(X_i)}{n\sigma_n}1\Big(\Big|\frac{h_{1,i}^{\ken}(X_i)}{n\sigma_n}\Big|\leq\epsilon\Big)\Big\}\Big]^2\leq E\Big\{\frac{h_{1,i}^{\ken}(X_i)}{n\sigma_n}\Big\}^2P\Big(\Big|\frac{h_{1,i}^{\ken}(X_i)}{n\sigma_n}\Big|>\epsilon\Big).\label{eq:taukendall-mainterm-bootstrap-proofuse-5}
\end{equation}
Taking summation over $1\leq i\leq n$ on both sides of (\ref{eq:taukendall-mainterm-bootstrap-proofuse-5}),
it follows from (\ref{eq:taukendall-mainterm-bootstrap-proofuse-2})
that
\begin{equation}
\sum_{i=1}^n\Big[E\Big\{\frac{h_{1,i}^{\ken}(X_i)}{n\sigma_n}1\Big(\Big|\frac{h_{1,i}^{\ken}(X_i)}{n\sigma_n}\Big|\leq\epsilon\Big)\Big\}\Big]^2\leq\frac{1}{\epsilon n\sigma_n}\sum_{i=1}^nE\Big\{\frac{h_{1,i}^{\ken}(X_i)}{n\sigma_n}\Big\}^2.\label{eq:taukendall-mainterm-bootstrap-proofuse-6}
\end{equation}
By (\ref{eq:Eh1i}) and (\ref{eq:clt-degreem-variancetendtoone})
we obtain
\begin{equation}
\sum_{i=1}^nE\Big\{\frac{h_{1,i}^{\ken}(X_i)}{n\sigma_n}\Big\}^2=\sigma_n^{-2}n^{-2}\sum_{i=1}^n\var\{h_{1,i}^{\ken}(X_i)\}\to1.\label{eq:taukendall-mainterm-bootstrap-proofuse-7}
\end{equation}
Equation (\ref{eq:ustat-bootstrap-condition-2}) then follows from
(\ref{eq:taukendall-mainterm-bootstrap-proofuse-1}), (\ref{eq:taukendall-mainterm-bootstrap-proofuse-6}),
and (\ref{eq:taukendall-mainterm-bootstrap-proofuse-7}).

\medskip{}

\textbf{Part II.} By (\ref{eq:tauap-AN-rateassumption}) and Theorem
\ref{thm:tauap-taukendall-AN}, we have that (\ref{eq:clt-degreem-variancetendtoone})
hold. This combined with (\ref{eq:tauap-Vn}) gives
\begin{equation}
n\sigma_n\gnsim n^{1/3}\log n,\label{eq:tauap-mainterm-bootstrap-proofuse-1}
\end{equation}
where $\sigma_n^2:=\var(U_n^{\ap})$. By (\ref{eq:tauap-h1i})
and the fact that $|\{\ind(j<i)-\ind(j>i)\}\{P(X_j>x)-\theta(i,j)\}|\leq1$,
we obtain
\begin{equation}
|h_{1,i}^{\ap}(x)|\leq\frac{n}{n-1}\Big(\sum_{j=1}^{i-1}\frac{1}{i-1}+\sum_{j=i+1}^n\frac{1}{j-1}\Big).\label{eq:tauap-mainterm-bootstrap-proofuse-2}
\end{equation}
It follows from (\ref{eq:tauap-mainterm-bootstrap-proofuse-2}) and
Lemma \ref{lem:bound-harmonic-sum} that
\begin{equation}
|h_{1,i}^{\ap}(x)|\leq\frac{n}{n-1}\{1+1+\log(n-1)\}\leq4+2\log n.\label{eq:tauap-mainterm-bootstrap-proofuse-3}
\end{equation}
By Markov's inequality and (\ref{eq:tauap-mainterm-bootstrap-proofuse-3})
we have for any $\epsilon>0$,
\begin{equation}
P\Big\{\Big|\frac{h_{1,i}^{\ap}(X_i)}{n\sigma_n}\Big|\geq\epsilon\Big\}\leq\frac{E|h_{1,i}^{\ap}(X_i)|}{\epsilon n\sigma_n}\leq\frac{4+2\log n}{\epsilon n\sigma_n}.\label{eq:tauap-mainterm-bootstrap-proofuse-4}
\end{equation}
Taking $\sup_{1\leq i\leq n}$ on both sides of (\ref{eq:tauap-mainterm-bootstrap-proofuse-4}),
we deduce (\ref{eq:ustat-bootstrap-condition-1}) from (\ref{eq:tauap-mainterm-bootstrap-proofuse-1}).

Equations (\ref{eq:taukendall-mainterm-bootstrap-proofuse-3}), (\ref{eq:taukendall-mainterm-bootstrap-proofuse-4})
and (\ref{eq:taukendall-mainterm-bootstrap-proofuse-5}) hold for
$h_{1,i}^{\ap}$ as well. Taking summation over $1\leq i\leq n$ on
both sides of (\ref{eq:taukendall-mainterm-bootstrap-proofuse-5}),
it follows from (\ref{eq:tauap-mainterm-bootstrap-proofuse-4}) that
\begin{equation}
\sum_{i=1}^n\Big[E\Big\{\frac{h_{1,i}^{\ap}(X_i)}{n\sigma_n}1\Big(\Big|\frac{h_{1,i}^{\ap}(X_i)}{n\sigma_n}\Big|\leq\epsilon\Big)\Big\}\Big]^2\leq\frac{4+2\log n}{\epsilon n\sigma_n}\sum_{i=1}^nE\Big\{\frac{h_{1,i}^{\ap}(X_i)}{n\sigma_n}\Big\}^2.\label{eq:tauap-mainterm-bootstrap-proofuse-5}
\end{equation}
By (\ref{eq:Eh1i}) and (\ref{eq:clt-degreem-variancetendtoone})
we obtain
\begin{equation}
\sum_{i=1}^nE\Big\{\frac{h_{1,i}^{\ap}(X_i)}{n\sigma_n}\Big\}^2=\sigma_n^{-2}n^{-2}\sum_{i=1}^n\var\{h_{1,i}^{\ap}(X_i)\}\to1.\label{eq:tauap-mainterm-bootstrap-proofuse-6}
\end{equation}
Equation (\ref{eq:ustat-bootstrap-condition-2}) then follows from
(\ref{eq:tauap-mainterm-bootstrap-proofuse-1}), (\ref{eq:tauap-mainterm-bootstrap-proofuse-5}),
and (\ref{eq:tauap-mainterm-bootstrap-proofuse-6}).

This completes the proof.
\end{proof}


\subsection{Proof of Corollary \ref{cor:bootstrap-taukendall}}
\begin{proof}
By Theorem \ref{thm:bootstrap-u-stat}, for proving Corollary \ref{cor:bootstrap-taukendall}, it suffices to show that (\ref{eq:ustat-bootstrap-condition-vonmises}), (\ref{eq:ustat-bootstrap-condition-homo1}), (\ref{eq:ustat-bootstrap-condition-homo2}),
and (\ref{eq:ustat-bootstrap-condition-3}) hold. For $U_n^{\ken}$
we have $|h(x,y)|\leq1$ for any $x,y$. This implies (\ref{eq:ustat-bootstrap-condition-vonmises}).

Now we establish (\ref{eq:ustat-bootstrap-condition-3}). For $U_n^{\ken}$,
we have $|a(i,j)|=|\ind(j<i)|\leq1$. By the definition in (\ref{eq:def-A2}),
we have
\begin{equation}
A_{2,1}(n)=n^{-3}\sum_{(I_n^2)_{\geq1}^{\otimes2}}|a(i_1,j_1)a(i_2,j_2)|=O(1).\label{eq:taukendall-bootstrap-proofuse1}
\end{equation}
By (\ref{eq:def-M1}) and (\ref{eq:taukendall-bootstrap-condition-1})
we have
\begin{equation}
M_1(n)\lesssim n^{-1/6}.\label{eq:taukendall-bootstrap-proofuse2}
\end{equation}
By (\ref{eq:def-M2}) and (\ref{eq:taukendall-bootstrap-condition-2})
we have
\begin{equation}
M_2(n)\lesssim n^{-1/3}.\label{eq:taukendall-bootstrap-proofuse3}
\end{equation}
Equation (\ref{eq:taukendall-AN-rateassumption}) implies (\ref{eq:taukendall-Vn})
by Theorem \ref{thm:tauap-taukendall-AN}. Combining (\ref{eq:taukendall-Vn})
and (\ref{eq:clt-degreem-variancetendtoone}) yields
\begin{equation}
\sigma_n^2\gnsim n^{-4/3}.\label{eq:taukendall-bootstrap-proofuse4}
\end{equation}
Equation (\ref{eq:ustat-bootstrap-condition-3}) follows from (\ref{eq:taukendall-bootstrap-proofuse1}),
(\ref{eq:taukendall-bootstrap-proofuse2}), (\ref{eq:taukendall-bootstrap-proofuse3}),
and (\ref{eq:taukendall-bootstrap-proofuse4}).

\medskip{}

Next we establish (\ref{eq:ustat-bootstrap-condition-homo1}). The
following lemma gives bounds on $\sum_{i=1}^nE\{h_{1,i}^{\ken}(X_j)^2\}$
and $\sum_{i=1}^nE\{h_{1,i}^{\ken}(X_i)^2\}$.
\begin{lemma}
\label{lem:taukendall-bootstrap-lemma}Under the assumptions of Corollary
\ref{cor:bootstrap-taukendall}, we have
\begin{equation}
\sum_{i=1}^nE\{h_{1,i}^{\ken}(X_j)^2\}=\frac{n(n+1)}{3(n-1)}(\eta^2-\theta^2)+O(n^{5/6}),\label{eq:taukendall-bootstrap-lemma-1}
\end{equation}
and
\begin{equation}
\sum_{i=1}^nE\{h_{1,i}^{\ken}(X_i)^2\}=\frac{n(n+1)}{3(n-1)}(\eta^2-\theta^2)+O(n^{5/6}).\label{eq:taukendall-bootstrap-lemma-2}
\end{equation}

\end{lemma}
\bigskip{}
By (\ref{eq:clt-degreem-variancetendtoone}) we have
\begin{equation}
\sigma_n^2=n^{-2}\sum_{i=1}^nE\{h_{1,i}^{\ken}(X_i)^2\}\{1+o(1)\}.\label{eq:taukendall-bootstrap-proofuse5}
\end{equation}
Using (\ref{eq:taukendall-bootstrap-lemma-2}) and (\ref{eq:taukendall-bootstrap-proofuse5})
we obtain
\begin{equation}
n^2\sigma_n^2=\{1+o(1)\}\Big\{\frac{n(n+1)}{3(n-1)}(\eta^2-\theta^2)+O(n^{5/6})\Big\}.\label{eq:taukendall-bootstrap-proofuse6}
\end{equation}
Note that
\begin{equation}
\frac{1}{n}\sum_{i=1}^n\sum_{j=1}^nE\Big[\Big\{\frac{h_{1,i}^{\ken}(X_j)}{n\sigma_n}\Big\}^2\Big]=\frac{n^{-1}\sum_{i=1}^n\sum_{j=1}^nE\{h_{1,i}^{\ken}(X_j)^2\}}{n^2\sigma_n^2}.\label{eq:taukendall-bootstrap-proofuse7}
\end{equation}
Combining (\ref{eq:taukendall-bootstrap-proofuse7}) with (\ref{eq:taukendall-bootstrap-lemma-1}),
(\ref{eq:taukendall-bootstrap-proofuse6}), and the fact that $\eta^2\neq\theta^2$
yields
\begin{equation}
\frac{1}{n}\sum_{i=1}^n\sum_{j=1}^nE\Big[\Big\{\frac{h_{1,i}^{\ken}(X_j)}{n\sigma_n}\Big\}^2\Big]=\frac{3^{-1}(n-1)^{-1}n(n+1)(\eta^2-\theta^2)+O(n^{5/6})}{\{1+o(1)\}\Big\{3^{-1}(n-1)^{-1}n(n+1)(\eta^2-\theta^2)+O(n^{5/6})\Big\}}\to1.\label{eq:taukendall-bootstrap-proofuse8}
\end{equation}
By (\ref{eq:taukendall-h1i}) we have $|h_{1,i}(x)|\leq1$. Therefore,
for any $x$
\begin{equation}
\Big|\sum_{i=1}^n\Big\{\frac{h_{1,i}(x)}{n\sigma_n}\Big\}^2\Big|=\Big|\frac{\sum_{i=1}^nh_{1,i}(x)^2}{n^2\sigma_n^2}\Big|\leq\frac{1}{n\sigma_n^2}.\label{eq:taukendall-bootstrap-proofuse9}
\end{equation}
Equations (\ref{eq:taukendall-bootstrap-proofuse6}) and (\ref{eq:taukendall-bootstrap-proofuse9})
imply that
\[
\sum_{j=1}^n\var\Big[\sum_{i=1}^n\Big\{\frac{h_{1,i}(X_j)}{n\sigma_n}\Big\}^2\Big]\leq\sum_{j=1}^n\frac{1}{n^2\sigma_n^{4}}=O(n)=o(n^2).
\]
It then follows from Lemma \ref{lem:WLLN} that
\begin{equation}
\frac{1}{n}\sum_{j=1}^n\Big[\sum_{i=1}^n\Big\{\frac{h_{1,i}(X_j)}{n\sigma_n}\Big\}^2-E\Big\{\sum_{i=1}^n\Big(\frac{h_{1,i}(X_j)}{n\sigma_n}\Big)^2\Big\}\Big]\stackrel{P}{\to}0.\label{eq:taukendall-bootstrap-proofuse11}
\end{equation}
Equation (\ref{eq:ustat-bootstrap-condition-homo1}) follows from
(\ref{eq:taukendall-bootstrap-proofuse8}) and (\ref{eq:taukendall-bootstrap-proofuse11}).

\medskip{}

Lastly, we prove (\ref{eq:ustat-bootstrap-condition-homo2}). By algebra
we have
\begin{equation}
\frac{1}{n^2}\sum_{i=1}^n\Big\{\sum_{j=1}^n\frac{h_{1,i}^{\ken}(X_j)}{n\sigma_n}\Big\}^2=\frac{1}{n^2}\sum_{j=1}^n\sum_{i=1}^n\Big\{\frac{h_{1,i}^{\ken}(X_j)}{n\sigma_n}\Big\}^2+\frac{1}{n^2}\sum_{j_1\neq j_2}\sum_{i=1}^n\frac{h_{1,i}^{\ken}(X_{j_1})h_{1,i}^{\ken}(X_{j_2})}{n^2\sigma_n^2}.\label{eq:taukendall-bootstrap-proofuse12}
\end{equation}
By (\ref{eq:ustat-bootstrap-condition-homo1}) we have
\begin{equation}
\frac{1}{n^2}\sum_{j=1}^n\sum_{i=1}^n\Big\{\frac{h_{1,i}^{\ken}(X_j)}{n\sigma_n}\Big\}^2\stackrel{P}{\to}0.\label{eq:taukendall-bootstrap-proofuse13}
\end{equation}
The second term on the right-hand side of (\ref{eq:taukendall-bootstrap-proofuse12})
is $(n-1)/n$ times a U-statistic with symmetric kernel $g(x,y)=n^{-2}\sigma_n^{-2}\sum_{i=1}^nh_{1,i}^{\ken}(x)h_{1,i}^{\ken}(y)$.
By (\ref{eq:taukendall-h1i}) and (\ref{eq:taukendall-bootstrap-condition-1})
we have
\begin{equation}
E\{h_{1,i}^{\ken}(X_j)\}=\frac{1}{n-1}\sum_{k=1}^n\mbox{sgn}(i-k)\{P(X_k>X_j)-P(X_k>X_i)\}=O(n^{-1/6}).\label{eq:taukendall-bootstrap-proofuse14}
\end{equation}
It follows from (\ref{eq:taukendall-bootstrap-proofuse14}) and (\ref{eq:taukendall-bootstrap-proofuse6})
that
\begin{equation}
E\{g(X_{j_1},X_{j_2})\}=n^{-2}\sigma_n^{-2}\sum_{i=1}^nE\{h_{1,i}^{\ken}(X_{j_1})\}E\{h_{1,i}^{\ken}(X_{j_2})\}\to0.\label{eq:taukendall-bootstrap-proofuse15}
\end{equation}
By (\ref{eq:taukendall-bootstrap-proofuse15}) and the weak law of
large numbers for U-statistics with independent but not identically
distributed variables \citep[Theorem 1, Section 3.7.2]{lee1990u},
we deduce
\begin{equation}
\frac{1}{n^2}\sum_{j_1\neq j_2}\sum_{i=1}^n\frac{h_{1,i}^{\ken}(X_{j_1})h_{1,i}^{\ken}(X_{j_2})}{n^2\sigma_n^2}\stackrel{P}{\to}0.\label{eq:taukendall-bootstrap-proofuse16}
\end{equation}
Equation (\ref{eq:ustat-bootstrap-condition-homo2}) follows from
(\ref{eq:taukendall-bootstrap-proofuse12}), (\ref{eq:taukendall-bootstrap-proofuse13}),
and (\ref{eq:taukendall-bootstrap-proofuse16}).

This completes the proof.
\end{proof}


\subsection{Proof of Corollary \ref{cor:bootstrap-tauap}}
\begin{proof}
By Theorem \ref{thm:bootstrap-u-stat}, for proving Corollary \ref{cor:bootstrap-tauap}, it suffices to show that (\ref{eq:ustat-bootstrap-condition-vonmises}), (\ref{eq:ustat-bootstrap-condition-homo1}), (\ref{eq:ustat-bootstrap-condition-homo2}),
and (\ref{eq:ustat-bootstrap-condition-3}) hold. For $U_n^{\ap}$
we have $|h(x,y)|\leq1$ for any $x,y$. This implies (\ref{eq:ustat-bootstrap-condition-vonmises}).

Now we establish (\ref{eq:ustat-bootstrap-condition-3}). For $U_n^{\ap}$,
we have $a(i,j)=\ind(j<i)n/(i-1)$. It follows that $a(i,j)a(j,i)=0$
and $a(i,i)=0$. By the definition in (\ref{eq:def-A2}), we have
\begin{align}
A_{2,1}(n) & =n^{-3}\sum_{(i,j)\in I_n^2}\sum_{k=1,k\neq i}^n\{|a(i,j)a(i,k)|+|a(i,j)a(k,i)|\}.\nonumber  \\
 & =n^{-3}\Big\{\sum_{i=2}^n\sum_{j=1}^{i-1}\sum_{k=1}^{i-1}\frac{n}{i-1}\cdot\frac{n}{i-1}+\sum_{i=2}^{n-1}\sum_{j=1}^{i-1}\sum_{k=i+1}^n\frac{n}{i-1}\cdot\frac{n}{k-1}\Big\}.\label{eq:tauap-bootstrap-proofuse-1}
\end{align}
By algebra, we have
\begin{equation}
\sum_{i=2}^n\sum_{j=1}^{i-1}\sum_{k=1}^{i-1}\frac{n}{i-1}\cdot\frac{n}{i-1}=n^2(n-1),\label{eq:tauap-bootstrap-proofuse-2}
\end{equation}
and
\begin{equation}
\sum_{i=2}^{n-1}\sum_{j=1}^{i-1}\sum_{k=i+1}^n\frac{n}{i-1}\cdot\frac{n}{k-1}=n^2\sum_{k=3}^n\sum_{i=2}^{k-1}\frac{1}{k-1}=n^2\sum_{k=3}^n\frac{k-2}{k-1}=O(n^{3}).\label{eq:tauap-bootstrap-proofuse-3}
\end{equation}
Combining (\ref{eq:tauap-bootstrap-proofuse-1}) with (\ref{eq:tauap-bootstrap-proofuse-2})
and (\ref{eq:tauap-bootstrap-proofuse-3}) yields
\begin{equation}
A_{2,1}(n)=O(1).\label{eq:tauap-bootstrap-proofuse-4}
\end{equation}
By (\ref{eq:def-M1}) and (\ref{eq:tauap-bootstrap-condition-1})
we have 
\begin{equation}
M_1(n)\lesssim n^{-1/6}\log n.\label{eq:tauap-bootstrap-proofuse-5}
\end{equation}
By (\ref{eq:def-M2}) and (\ref{eq:tauap-bootstrap-condition-2})
we have 
\begin{equation}
M_2(n)\lesssim n^{-1/3}(\log n)^2.\label{eq:tauap-bootstrap-proofuse-6}
\end{equation}
Equation (\ref{eq:tauap-AN-rateassumption}) implies (\ref{eq:tauap-Vn})
by Theorem \ref{thm:tauap-taukendall-AN}. Combining (\ref{eq:tauap-Vn})
and (\ref{eq:clt-degreem-variancetendtoone}) yields
\begin{equation}
\sigma_n^2\gnsim n^{-4/3}(\log n)^2.\label{eq:tauap-bootstrap-proofuse-7}
\end{equation}
Equation (\ref{eq:ustat-bootstrap-condition-3}) follows from (\ref{eq:tauap-bootstrap-proofuse-4}),
(\ref{eq:tauap-bootstrap-proofuse-5}), (\ref{eq:tauap-bootstrap-proofuse-6}),
and (\ref{eq:tauap-bootstrap-proofuse-7}).

\medskip{}

Next we establish (\ref{eq:ustat-bootstrap-condition-homo1}). The
following lemma gives useful bounds.
\begin{lemma}
\label{lem:tauap-bootstrap-lemma}Under the assumptions of Corollary
\ref{cor:bootstrap-tauap}, we have
\begin{equation}
\sum_{i=1}^nE\{h_{1,i}^{\ap}(X_j)^2\}=\frac{n^2}{n-1}(\eta^2-\theta^2)+O(n^{5/6}\log n),\label{eq:tauap-bootstrap-lemma-1}
\end{equation}
and
\begin{equation}
\sum_{i=1}^nE\{h_{1,i}^{\ap}(X_i)^2\}=\frac{n^2}{n-1}(\eta^2-\theta^2)+O(n^{5/6}\log n).\label{eq:tauap-bootstrap-lemma-2}
\end{equation}

\end{lemma}
\bigskip{}
By (\ref{eq:clt-degreem-variancetendtoone}) we have
\begin{equation}
\sigma_n^2=n^{-2}\sum_{i=1}^nE\{h_{1,i}^{\ap}(X_i)^2\}\{1+o(1)\}.\label{eq:tauap-bootstrap-proofuse-8}
\end{equation}
Using (\ref{eq:tauap-bootstrap-proofuse-8}) and (\ref{eq:tauap-bootstrap-lemma-2})
we obtain
\begin{equation}
n^2\sigma_n^2=\{1+o(1)\}\Big\{\frac{n^2}{n-1}(\eta^2-\theta^2)+O(n^{5/6}\log n)\Big\}.\label{eq:tauap-bootstrap-proofuse-9}
\end{equation}
Note that
\begin{equation}
\frac{1}{n}\sum_{i=1}^n\sum_{j=1}^nE\Big[\Big\{\frac{h_{1,i}^{\ap}(X_j)}{n\sigma_n}\Big\}^2\Big]=\frac{n^{-1}\sum_{i=1}^n\sum_{j=1}^nE\{h_{1,i}^{\ap}(X_j)^2\}}{n^2\sigma_n^2}.\label{eq:tauap-bootstrap-proofuse-10}
\end{equation}
Combining (\ref{eq:tauap-bootstrap-proofuse-10}) with (\ref{eq:tauap-bootstrap-lemma-1}),
(\ref{eq:tauap-bootstrap-proofuse-9}), and the fact that $\eta^2\neq\theta^2$
yields
\begin{equation}
\frac{1}{n}\sum_{i=1}^n\sum_{j=1}^nE\Big[\Big\{\frac{h_{1,i}^{\ap}(X_j)}{n\sigma_n}\Big\}^2\Big]=\frac{n^2(n-1)^{-1}(\eta^2-\theta^2)+O(n^{5/6}\log n)}{\{1+o(1)\}\Big\{ n^2(n-1)^{-1}(\eta^2-\theta^2)+O(n^{5/6}\log n)\Big\}}\to1.\label{eq:tauap-bootstrap-proofuse-11}
\end{equation}
By (\ref{eq:tauap-h1i}) we have $|h_{1,i}^{\ap}(x)|\leq1+\varphi(n-1)-\varphi(i-1)$
for all $x$. This combined with Lemma \ref{lem:bound-harmonic-sum} yields
\begin{equation}
|h_{1,i}^{\ap}(x)|\leq1+\log\frac{n}{i}\leq1+\log n.\label{eq:tauap-bootstrap-proofuse-12}
\end{equation}
It then follows from (\ref{eq:tauap-bootstrap-proofuse-12}) that
\begin{equation}
\Big|\sum_{i=1}^n\Big\{\frac{h_{1,i}^{\ap}(x)}{n\sigma_n}\Big\}^2\Big|=\Big|\frac{\sum_{i=1}^nh_{1,i}^{\ap}(x)^2}{n^2\sigma_n^2}\Big|\leq\frac{(1+\log n)^2}{n\sigma_n^2}.\label{eq:tauap-bootstrap-proofuse-13}
\end{equation}
Equations (\ref{eq:tauap-bootstrap-proofuse-9}) and (\ref{eq:tauap-bootstrap-proofuse-13})
imply that
\[
\sum_{j=1}^n\var\Big[\sum_{i=1}^n\Big\{\frac{h_{1,i}^{\ap}(X_j)}{n\sigma_n}\Big\}^2\Big]\leq\sum_{j=1}^n\frac{(1+\log n)^{4}}{n^2\sigma_n^{4}}=O\{n(\log n)^{4}\}=o(n^2).
\]
It then follows from Lemma \ref{lem:WLLN} that
\begin{equation}
\frac{1}{n}\sum_{j=1}^n\Big[\sum_{i=1}^n\Big\{\frac{h_{1,i}^{\ap}(X_j)}{n\sigma_n}\Big\}^2-E\Big\{\sum_{i=1}^n\Big(\frac{h_{1,i}^{\ap}(X_j)}{n\sigma_n}\Big)^2\Big\}\Big]\stackrel{P}{\to}0.\label{eq:tauap-bootstrap-proofuse-15}
\end{equation}
Equation (\ref{eq:ustat-bootstrap-condition-homo1}) follows from
(\ref{eq:tauap-bootstrap-proofuse-11}) and (\ref{eq:tauap-bootstrap-proofuse-15}).

\medskip{}

Lastly, we prove (\ref{eq:ustat-bootstrap-condition-homo2}). By algebra
we have
\begin{equation}
\frac{1}{n^2}\sum_{i=1}^n\Big\{\sum_{j=1}^n\frac{h_{1,i}^{\ap}(X_j)}{n\sigma_n}\Big\}^2=\frac{1}{n^2}\sum_{j=1}^n\sum_{i=1}^n\Big\{\frac{h_{1,i}^{\ap}(X_j)}{n\sigma_n}\Big\}^2+\frac{1}{n^2}\sum_{j_1\neq j_2}\sum_{i=1}^n\frac{h_{1,i}^{\ap}(X_{j_1})h_{1,i}^{\ap}(X_{j_2})}{n^2\sigma_n^2}.\label{eq:tauap-bootstrap-proofuse-16}
\end{equation}
By (\ref{eq:ustat-bootstrap-condition-homo1}) we have
\begin{equation}
\frac{1}{n^2}\sum_{j=1}^n\sum_{i=1}^n\Big\{\frac{h_{1,i}^{\ap}(X_j)}{n\sigma_n}\Big\}^2\stackrel{P}{\to}0.\label{eq:tauap-bootstrap-proofuse-17}
\end{equation}
The second term on the right-hand side of (\ref{eq:tauap-bootstrap-proofuse-16})
is $(n-1)/n$ times a U-statistic with symmetric kernel $g(x,y)=n^{-2}\sigma_n^{-2}\sum_{i=1}^nh_{1,i}^{\ap}(x)h_{1,i}^{\ap}(y)$.
By (\ref{eq:tauap-bootstrap-proofuse-12}) and (\ref{eq:tauap-bootstrap-condition-1})
we have
\[
E\{h_{1,i}^{\ap}(X_j)\}=\frac{1}{n-1}\sum_{k=1}^n\Big\{\frac{n\ind(j<i)}{i-1}-\frac{n\ind(j>i)}{j-1}\Big\} O(n^{-1/6}\log n).
\]
This combined with Lemma \ref{lem:bound-harmonic-sum} yields
\begin{equation}
E\{h_{1,i}^{\ap}(X_j)\}=O\{n^{-1/6}(\log n)^2\}.\label{eq:eq:tauap-bootstrap-proofuse-18}
\end{equation}
It follows from (\ref{eq:eq:tauap-bootstrap-proofuse-18}) and (\ref{eq:tauap-bootstrap-proofuse-9})
that
\begin{equation}
E\{g(X_{j_1},X_{j_2})\}=n^{-2}\sigma_n^{-2}\sum_{i=1}^nE\{h_{1,i}^{\ap}(X_{j_1})\}E\{h_{1,i}^{\ap}(X_{j_2})\}\to0.\label{eq:eq:tauap-bootstrap-proofuse-19}
\end{equation}
By (\ref{eq:eq:tauap-bootstrap-proofuse-19}) and Theorem 1 in \citet[Section 3.7.2]{lee1990u},
we deduce
\begin{equation}
\frac{1}{n^2}\sum_{j_1\neq j_2}\sum_{i=1}^n\frac{h_{1,i}^{\ap}(X_{j_1})h_{1,i}^{\ap}(X_{j_2})}{n^2\sigma_n^2}\stackrel{P}{\to}0.\label{eq:eq:tauap-bootstrap-proofuse-20}
\end{equation}
Equation (\ref{eq:ustat-bootstrap-condition-homo2}) follows from
(\ref{eq:tauap-bootstrap-proofuse-16}), (\ref{eq:tauap-bootstrap-proofuse-17}),
and (\ref{eq:eq:tauap-bootstrap-proofuse-20}).

This completes the proof.
\end{proof}

\section{Proofs of the supporting lemmas}
\label{sec:proof-supportinglemmas}

In this section, we prove the supporting lemmas that appear in Section \ref{sec:proofs}.


\subsection{Proof of Lemma \ref{lem:clt-proofuse}}
\begin{proof}
To simplify notation, define $\bi=(i_1,\ldots,i_m)$, $X_{\bi}=(X_{i_1},\ldots,X_{i_m})$,
and $\bi_{-m}=(i_1,\ldots,i_{m-1})$. By definition of $h_{1,i}(X_i)$
in (\ref{eq:def-h1}) we have 
\begin{equation}
\sum_{i=1}^nE|h_{1,i}(X_i)|^{3}\!=\!\sum_{i=1}^n\Big\{\frac{(n-m)!}{(n-1)!}\Big\}^{3}E\Big|\sum_{I_{n-1}^{m-1}(-i)}\sum_{l=1}^ma^{(l)}(i;\bi_{-m})\Big\{ f_{\bi_{-m}}^{(l)}(X_i)-\theta^{(l)}(i;\bi_{-m})\Big\}\Big|^{3}.\label{eq:clt-proofuse-1.01}
\end{equation}
Define
\[
T_{\bi_{-m}}^{(l_1)}(X_i)=f_{\bi_{-m}}^{(l_1)}(X_i)-\theta^{(l_1)}(i;\bi_{-m}),
\]
and define $T_{\bj_{-m}}^{(l_2)}(X_i)$ and $T_{\bk_{-m}}^{(l_3)}(X_i)$ similarly.
The right-hand side of (\ref{eq:clt-proofuse-1.01}) equals
\begin{equation}
\Big\{\frac{(n-m)!}{(n-1)!}\Big\}^{3}\sum \Big|a^{(l_1)}\!(i;\bi_{-m})a^{(l_2)}\!(i;\bj_{-m})a^{(l_{3})}\!(i;\bk_{-m}) \Big| E \Big|T_{\bi_{-m}}^{(l_1)}(X_i)T_{\bj_{-m}}^{(l_2)}(X_i)T_{\bk_{-m}}^{(l_3)}(X_i) \Big|,\label{eq:clt-proofuse-2.22}
\end{equation}
where the summation is over $i\in[n]$, $l_1,l_2,l_{3}\in[m]$,
and $\bi_{-m},\bj_{-m},\bk_{-m}\in I_{n-1}^{m-1}(-i)$. By Cauchy-Schwarz
inequality and Lemma \ref{lem:moment-bound}(ii), we have
\begin{align}
 & E \Big| T_{\bi_{-m}}^{(l_1)}(X_i)T_{\bj_{-m}}^{(l_2)}(X_i)T_{\bk_{-m}}^{(l_3)}(X_i) \Big| \leq\Big[E\{T_{\bi_{-m}}^{(l_1)}(X_i)^2T_{\bj_{-m}}^{(l_2)}(X_i)^2\}\Big]^{1/2}\Big[E\{T_{\bk_{-m}}^{(l_3)}(X_i)^2\}\Big]^{1/2}\nonumber  \\
 & \leq\Big[E\{T_{\bi_{-m}}^{(l_1)}(X_i)^{4}\}\Big]^{1/4}\Big[E\{T_{\bj_{-m}}^{(l_2)}(X_i)^{4}\}\Big]^{1/4}\Big[E\{T_{\bk_{-m}}^{(l_3)}(X_i)^{4}\}\Big]^{1/4}\leq CM(n)^{3/4}.  \label{eq:clt-proofuse-2.3}
\end{align}
By the definition of $A_{3,1}(n)$ in (\ref{eq:def-A2}) and algebra,
we have 
\begin{equation}
\sum|a^{(l_1)}(i;\bi_{-m})a^{(l_2)}(i;\bj_{-m})a^{(l_{3})}(i;\bk_{-m})|\leq C\sum_{(I_n^m)_{\geq1}^{\otimes3}}\left|a(\bi)a(\bj)a(\bk)\right|=Cn^{3m-2}A_{3,1}(n),\label{eq:clt-proofuse-2.5}
\end{equation}
where the summation in the leftmost part of (\ref{eq:clt-proofuse-2.5})
is over $i\in[n]$, $l_1,l_2,l_{3}\in[m]$, and $\bi_{-m},\bj_{-m},\bk_{-m}\in I_{n-1}^{m-1}(-i)$.
By (\ref{eq:clt-proofuse-1.01}), (\ref{eq:clt-proofuse-2.22}), (\ref{eq:clt-proofuse-2.3}),
and (\ref{eq:clt-proofuse-2.5}), we deduce 
\begin{equation}
\sum_{i=1}^nE\left|h_{1,i}(X_i)\right|^{3}\leq CnA_{3,1}(n)M(n)^{3/4}.\label{eq:clt-proofuse-2.6}
\end{equation}
This completes the proof.
\end{proof}


\subsection{Proof of Lemma \ref{lem:decomposition-pf-1}}
\begin{proof}
We prove Lemma \ref{lem:decomposition-pf-1} by showing that for each
$(i_1^*,\ldots,i_m^*)\in I_n^m$, the coefficients of
$a(i_1^*,\ldots,i_m^*)$ on both sides of (\ref{eq:decomposition-proofuse-lemma})
are equal. In the following we fix $(i_1^*,\ldots,i_m^*)\in I_n^m$.

For the left-hand side of (\ref{eq:decomposition-proofuse-lemma}),
we enumerate the combinations in
\[
\{l,i,(i_1,\ldots,i_{m-1}):l\in[m],i\in[n],(i_1,\ldots,i_{m-1})\in I_{n-1}^{m-1}(-i)\}
\]
such that $a^{(l)}(i;i_1,\ldots,i_{m-1})=a(i_1^*,\ldots,i_m^*)$, as follows:
\begin{align}
l=1, & i=i_1^*,(i_1,\ldots,i_{m-1})=(i_1^*,\ldots,i_m^*)\backslash i_1^*;\nonumber  \\
 & \vdots\nonumber  \\
l=j, & i=i_j^*,(i_1,\ldots,i_{m-1})=(i_1^*,\ldots,i_m^*)\backslash i_j^*;\label{eq:decomposition-proofuse-3}\\
 & \vdots\nonumber  \\
l=m, & i=i_m^*,(i_1,\ldots,i_{m-1})=(i_1^*,\ldots,i_m^*)\backslash i_m^*.\nonumber 
\end{align}
When $l=j,i=i_j^*,(i_1,\ldots,i_{m-1})=(i_1^*,\ldots,i_m^*)\backslash i_j^*$,
\[
E\{h^{(l)}(X_i;X_{i_1},\ldots,X_{i_{m-1}})\mid X_i\}-\theta^{(l)}(i;i_1,\ldots,i_{m-1})=E\{h(X_{i_1^*},\ldots,X_{i_m^*})\mid X_{i_j^*}\}-\theta(i_1^*,\ldots,i_m^*).
\]
So the coefficient of $a(i_1^*,\ldots,i_m^*)$ on the left-hand
side of (\ref{eq:decomposition-proofuse-lemma}) is
\[
\sum_{j=1}^m\left[E\{h(X_{i_1^*},\ldots,X_{i_m^*}\mid X_{i_j^*}\}-\theta(i_1^*,\ldots,i_m^*)\right].
\]
This equals the coefficient of $a(i_1^*,\ldots,i_m^*)$ on
the right-hand side of (\ref{eq:decomposition-proofuse-lemma}). This
completes the proof.
\end{proof}


\subsection{Proof of Lemma \ref{lem:u-stat-bootstrap-remainder-small}}

Lemmas \ref{lem:h2-bootstrap-proofuse} and \ref{lem:h2-bootstrap-proofuse-1} that appear in this proof are proven immediately after this proof.

\begin{proof}
Define $\bi:=(i_1,\ldots,i_m)$ and $\Xbi:=(X_{i_1},\ldots,X_{i_m})$.
By (\ref{eq:def-Unstar}) we have 
\begin{align}
E^*\{\sigma_n^{-2}U_n^*(h_2)^2\} & =\sigma_n^{-2}\Big\{\frac{(n-m)!}{n!}\Big\}^2\sum_{(I_n^m)_{\geq0}^{\otimes2}}a(\bi)a(\bj)E^*\{h_{2;\bi}(\Xbi^*)h_{2;\bj}(\Xbj^*)\}\label{eq:h2-bootstrap-proofuse1}\\
~~{\rm and}~~~[E^*\{\sigma_n U_n^*(h_2)\}]^2 & =\sigma_n^{-2}\Big\{\frac{(n-m)!}{n!}\Big\}^2\sum_{(I_n^m)_{\geq0}^{\otimes2}}a(\bi)a(\bj)E^*\{h_{2;\bi}(\Xbi^*)\}E^*\{h_{2;\bj}(\Xbj^*)\}.\label{eq:h2-bootstrap-proofuse2}
\end{align}
Define
\begin{equation}
g(\bi,\bj):=a(\bi)a(\bj)\Big[E^*\{h_{2;\bi}(\Xbi^*)h_{2;\bj}(\Xbj^*)\}-E^*\{h_{2;\bi}(\Xbi^*)\}E^*\{h_{2;\bj}(\Xbj^*)\}\Big].\label{eq:h2-bootstrap-proofuse2.6}
\end{equation}
It follows from (\ref{eq:h2-bootstrap-proofuse1}) and (\ref{eq:h2-bootstrap-proofuse2})
that
\begin{equation}
\var^*\{\sigma_n^{-1}U_n^*(h_2)\}=\sigma_n^{-2}\Big\{\frac{(n-m)!}{n!}\Big\}^2\sum_{(\bi,\bj)\in(I_n^m)_{\geq0}^{\otimes2}}g(\bi,\bj).\label{eq:h2-bootstrap-proofuse3}
\end{equation}
The following proof consists of two steps. In the first step, we establish
\begin{equation}
\var^*\{\sigma_n^{-1}U_n^*(h_2)\}=\sigma_n^{-2}\Big\{\frac{(n-m)!}{n!}\Big\}^2\sum_{(\bi,\bj)\in(I_n^m)_{=1}^{\otimes2}}g(\bi,\bj)+o_{P}(1).\label{eq:h2-bootstrap-proofuse3.5}
\end{equation}
In the second step, we show that
\begin{equation}
\sigma_n^{-2}\Big\{\frac{(n-m)!}{n!}\Big\}^2\sum_{(\bi,\bj)\in(I_n^m)_{=1}^{\otimes2}}g(\bi,\bj)\stackrel{P}{\to}0.\label{eq:h2-bootstrap-proofuse11}
\end{equation}
Lemma \ref{lem:u-stat-bootstrap-remainder-small} then follows from (\ref{eq:h2-bootstrap-proofuse3.5}), (\ref{eq:h2-bootstrap-proofuse11}), and Slutsky's theorem.

\textbf{Step I.} If $(\bi,\bj)\in(I_n^m)_{=0}^{\otimes2}$, we
have $E^*\{h_{2;\bi}(\Xbi^*)h_{2;\bj}(\Xbj^*)\}=E^*\{h_{2;\bi}(\Xbi^*)\}E^*\{h_{2;\bj}(\Xbj^*)\}~\as$.
This implies
\begin{equation}
\sum_{(\bi,\bj)\in(I_n^m)_{=0}^{\otimes2}}g(\bi,\bj)=0~\as.\label{eq:h2-bootstrap-proofuse4}
\end{equation}
For any $(\bi,\bj)\in(I_n^m)_{\geq2}^{\otimes2}$, by the law of iterated expectation, Cauchy-Schwarz inequality, and triangular inequality we have
\begin{align}
E \Big| E^*\{h_{2;\bi}(\Xbi^*)h_{2;\bj}(\Xbj^*)\} \Big| & \leq E \Big\{ |h_{2;\bi}(\Xbi^*)| |h_{2;\bj}(\Xbj^*)|\Big\} \leq \Big[E\{h_{2;\bi}(\Xbi^*)^2\}E\{h_{2;\bj}(\Xbj^*)^2\}\Big]^{\frac{1}{2}}.\label{eq:h2-bootstrap-proofuse5}
\end{align}
Similarly, by Jensen's inequality and triangular inequality we have
\begin{align}
& E \Big| E^*\{h_{2;\bi}(\Xbi^*)\}  E^*\{h_{2;\bj}(\Xbj^*)\} \Big|  \leq E \Big\{ E^* |h_{2;\bi}(\Xbi^*)| E^* |h_{2;\bj}(\Xbj^*)|\Big\} \nonumber  \\
\leq & \Big[E\{E^* |h_{2;\bi}(\Xbi^*)| \}^2 E\{ E^* |h_{2;\bj}(\Xbj^*)|\}^2\Big]^{\frac{1}{2}}
\leq \Big[E\{h_{2;\bi}(\Xbi^*)^2\}E\{h_{2;\bj}(\Xbj^*)^2\}\Big]^{\frac{1}{2}}.\label{eq:h2-bootstrap-proofuse5.5}
\end{align}
Using the law of iterated expectation, we deduce 
\begin{equation}
E\{h_{2;\bi}(\Xbi^*)^2\}=E[E^*\{h_{2;\bi}(\Xbi^*)^2\}]=n^{-m}\sum_{1\leq j_1,\ldots,j_m\leq n}E\{h_{2;\bi}(X_{j_1},\ldots,X_{j_m})^2\}.\label{eq:h2-bootstrap-proofuse6}
\end{equation}
By Lemma \ref{lem:moment-bound}(iii) and (\ref{eq:ustat-bootstrap-condition-vonmises}),
there exists an absolute constant $C>0$ such that for any $n$, for any $\bi\in I_n^m$,
and for any $1\leq j_1,\ldots,j_m\leq n$, 
\begin{equation}
E\{h_{2;\bi}(X_{j_1},\ldots,X_{j_m})^2\}\leq C.\label{eq:h2-bootstrap-proofuse7}
\end{equation}
Combining (\ref{eq:h2-bootstrap-proofuse6}) and (\ref{eq:h2-bootstrap-proofuse7})
yields that $E\{h_{2;\bi}(\Xbi^*)^2\}\leq C$. It then follows from (\ref{eq:h2-bootstrap-proofuse5}) and (\ref{eq:h2-bootstrap-proofuse5.5}) that
\begin{align}
E \Big| E^*\{h_{2;\bi}(\Xbi^*)h_{2;\bj}(\Xbj^*)\} \Big| &\leq C, \label{eq:h2-bootstrap-proofuse8} \\
{\rm and}~~~~~E \Big| E^*\{h_{2;\bi}(\Xbi^*)\}  E^*\{h_{2;\bj}(\Xbj^*)\} \Big| &\leq C. \label{eq:h2-bootstrap-proofuse7.5}
\end{align}
Equations (\ref{eq:h2-bootstrap-proofuse2.6}), (\ref{eq:h2-bootstrap-proofuse8}),  and (\ref{eq:h2-bootstrap-proofuse7.5}) imply that
\begin{equation}
E\Big| \sum_{(I_n^m)_{\geq2}^{\otimes2}}g(\bi,\bj)\Big| \leq 2C\sum_{(I_n^m)_{\geq2}^{\otimes2}}|a(\bi)a(\bj)|=2Cn^{2m-2}A_{2,2}(n).\label{eq:h2-bootstrap-proofuse9}
\end{equation}
By (\ref{eq:clt-degreem-condition-1}), (\ref{eq:clt-degreem-variancetendtoone}),
and (\ref{eq:h2-bootstrap-proofuse9}), we deduce
\[
\sigma_n^{-2}\Big\{\frac{(n-m)!}{n!}\Big\}^2 E\Big| \sum_{(I_n^m)_{\geq2}^{\otimes2}}g(\bi,\bj)\Big| \to0.
\]
It then follows from Markov's inequality that
\begin{equation}
\sigma_n^{-2}\Big\{\frac{(n-m)!}{n!}\Big\}^2\sum_{(I_n^m)_{\geq2}^{\otimes2}}g(\bi,\bj)\stackrel{P}{\to}0.\label{eq:h2-bootstrap-proofuse10}
\end{equation}
Combining (\ref{eq:h2-bootstrap-proofuse3}), (\ref{eq:h2-bootstrap-proofuse4}),
and (\ref{eq:h2-bootstrap-proofuse10}) yields (\ref{eq:h2-bootstrap-proofuse3.5}).
This concludes \textbf{Step I.}

\medskip{}

\textbf{Step II.} Consider a fixed $(\bi,\bj)\in(I_n^m)_{=1}^{\otimes2}$.
Without loss of generality assume $\bi\cap\bj=\{i_p\}=\{j_q\}$
for some $1\leq p,q\leq m$. By the i.i.d.-ness of $X_i^*$'s
given $X_1,\ldots,X_n$, we have
\begin{equation}
E[E^*\{h_{2;\bi}(\Xbi^*)h_{2;\bj}(\Xbj^*)\}]=n^{-(2m-1)}\sum_{\substack{\br,\bs\in[n]^{2m}\\
r_p=s_q
}
}E\{h_{2;\bi}(\Xbr)h_{2;\bj}(\Xbs)\} \label{eq:h2-bootstrap-proofuse12}
\end{equation}
and
\begin{equation}
E[E^*\{h_{2;\bi}(\Xbi^*)\}E^*\{h_{2;\bj}(\Xbj^*)\}]=n^{-2m}\sum_{\br,\bs\in[n]^{2m}}E\{h_{2;\bi}(\Xbr)h_{2;\bj}(\Xbs)\}.\label{eq:h2-bootstrap-proofuse13}
\end{equation}
The number of pairs $(\br,\bs)$ in $\{(\br,\bs)\in[n]^{2m}:r_p=s_q\}$ satisfying any of the following three statements is of order $O(n^{2m-2})$: (1) $\br$ or $\bs$ has duplicate indices (i.e., $\br\notin I_n^m$ or $\bs\notin I_n^m$); (2) $\bi\cap\br\neq\emptyset$; or (3) $\bj\cap\bs\neq\emptyset$. It then follows from (\ref{eq:h2-bootstrap-proofuse8}) that 
\begin{equation}
\sum_{\substack{\br,\bs\in[n]^{2m}\\
r_p=s_q
}
}E\{h_{2;\bi}(\Xbr)h_{2;\bj}(\Xbs)\}=\sum_{\substack{(\br,\bs)\in(I_n^m)_{=1}^{\otimes2}\\
r_p=s_q,\bi\cap\br=\emptyset=\bj\cap\bs
}
}E\{h_{2;\bi}(\Xbr)h_{2;\bj}(\Xbs)\}+O(n^{2m-2}).\label{eq:h2-bootstrap-proofuse14}
\end{equation}
The following lemma gives bound on the right-hand side of (\ref{eq:h2-bootstrap-proofuse14}).
\begin{lemma}
\label{lem:h2-bootstrap-proofuse}For any $(\bi,\bj)\in(I_n^m)_{=1}^{\otimes2}$,
under the assumptions of Theorem \ref{thm:bootstrap-u-stat}, there exists a constant $C$ such that
\begin{equation}
\Big|\sum_{\substack{(\br,\bs)\in(I_n^m)_{=1}^{\otimes2}\\
r_p=s_q,\bi\cap\br=\emptyset=\bj\cap\bs
}
}E\{h_{2;\bi}(\Xbr)h_{2;\bj}(\Xbs)\}\Big|\leq Cn^{2m-1}\{M_1(n)^2+M_2(n)\}.\label{eq:h2-bootstrap-proofuse-lemma}
\end{equation}

\end{lemma}
\bigskip{}
It follows from (\ref{eq:h2-bootstrap-proofuse12}), (\ref{eq:h2-bootstrap-proofuse14})
and Lemma \ref{lem:h2-bootstrap-proofuse} that
\begin{equation}
|E[E^*\{h_{2;\bi}(\Xbi^*)h_{2;\bj}(\Xbj^*)\}]|\leq C\{M_1(n)^2+M_2(n)+n^{-1}\}.\label{eq:h2-bootstrap-proofuse16}
\end{equation}
Using an argument similar to (\ref{eq:h2-bootstrap-proofuse14}),
we have 
\begin{equation}
\sum_{\br,\bs\in\{1,\ldots,n\}^{2m}}E\{h_{2;\bi}(\Xbr)h_{2;\bj}(\Xbs)\}=\sum_{\substack{(\br,\bs)\in(I_n^m)_{=0}^{\otimes2}\\
\bi\cap\br=\emptyset=\bj\cap\bs
}
}E\left\{ h_{2;\bi}(\Xbr)h_{2;\bj}(\Xbs)\right\} +O(n^{2m-1}).\label{eq:h2-bootstrap-proofuse17}
\end{equation}
The following lemma gives bound on the right-hand side of (\ref{eq:h2-bootstrap-proofuse17}). 
\begin{lemma}
\label{lem:h2-bootstrap-proofuse-1}For any $(\bi,\bj)\in(I_n^m)_{=1}^{\otimes2}$,
under the assumptions of Theorem \ref{thm:bootstrap-u-stat}, there
exists a constant $C$ such that
\begin{equation}
\Big|\sum_{\substack{(\br,\bs)\in(I_n^m)_{=0}^{\otimes2}\\
\bi\cap\br=\emptyset=\bj\cap\bs
}
}E\{h_{2;\bi}(\Xbr)h_{2;\bj}(\Xbs)\}\Big|\leq Cn^{2m}M_1(n)^2.\label{eq:h2-bootstrap-proofuse-lemma-1}
\end{equation}

\end{lemma}
\bigskip{}
It follows from (\ref{eq:h2-bootstrap-proofuse13}), (\ref{eq:h2-bootstrap-proofuse17})
and Lemma \ref{lem:h2-bootstrap-proofuse-1} that
\begin{equation}
|E[E^*\{h_{2;\bi}(\Xbi^*)\}E^*\{h_{2;\bj}(\Xbj^*)\}]|\leq CM_1(n)^2.\label{eq:h2-bootstrap-proofuse18}
\end{equation}
Combining (\ref{eq:h2-bootstrap-proofuse2.6}) with (\ref{eq:h2-bootstrap-proofuse16})
and (\ref{eq:h2-bootstrap-proofuse18}) yields that, for any $(\bi,\bj)\in(I_n^m)_{=1}^{\otimes2}$,
\[
|g(\bi,\bj)|\leq C|a(\bi)a(\bj)|\{M_1(n)^2+M_2(n)+n^{-1}\}.
\]
Therefore, by the definition of $A_{2,1}(n)$ in (\ref{eq:def-A2}), we have
\begin{equation}
\Big|\sum_{(I_n^m)_{=1}^{\otimes2}}g(\bi,\bj)\Big|\leq Cn^{2m-1}A_{2,1}(n)\{M_1(n)^2+M_2(n)+n^{-1}\}.\label{eq:h2-bootstrap-proofuse19}
\end{equation}
Equation (\ref{eq:h2-bootstrap-proofuse11}) follows from (\ref{eq:h2-bootstrap-proofuse19})
and (\ref{eq:ustat-bootstrap-condition-3}). This concludes \textbf{Step II}.

The proof is thus finished.
\end{proof}



\begin{proof}[Proof of Lemma \ref{lem:h2-bootstrap-proofuse}]
For a fixed $(\bi,\bj)\in(I_n^m)_{=1}^{\otimes2}$, consider any $(\br,\bs)\in(I_n^m)_{=1}^{\otimes2}$ with $r_p=s_q$ and
$\bi\cap\br=\emptyset=\bj\cap\bs$. By the law of iterated expectation and the independence of $X_i$'s we have
\begin{equation}
E\{h_{2;\bi}(\Xbr)h_{2;\bj}(\Xbs)\}=E[E\{h_{2;\bi}(\Xbr)\mid X_{r_p}\}E\{h_{2;\bj}(\Xbs)\mid X_{s_q}\}].\label{eq:h2-bootstrap-proofuse15}
\end{equation}
For $\bi = (i_1,\ldots,i_m)$ and $l\in[m]$, define
\[
\bi \backslash i_l := (i_1,\ldots,i_{l-1},i_{l+1},\ldots,i_m).
\]
Using the definition of $h_{2;\bi}(\cdot)$ in (\ref{eq:def-h2})
we have
\begin{align}
E\{h_{2;\bi}(\Xbr)\mid X_{r_p}\} & =E\{h(\Xbr)\mid X_{r_p}\}\nonumber  \\
-\sum_{l=1}^mE[E_{\bi\backslash i_l} & \{h^{(l)}(X_{r_l};Y_1,\ldots,Y_{m-1})\mid X_{r_l}\}\mid X_{r_p}]+(m-1)\theta(\bi).\label{eq:h2-bootstrap-lemma-proofuse-1}
\end{align}
By the independence of the $X_i$'s we have
\begin{align}
\sum_{l=1}^mE[E_{\bi\backslash i_l} & \{h^{(l)}(X_{r_l};Y_1,\ldots,Y_{m-1})\mid X_{r_l}\}\mid X_{r_p}].\nonumber \\
 & =\sum_{\substack{l=1\\
l\neq p
}
}^m\theta^{(l)}(r_l;\bi\backslash i_l)+E_{\bi\backslash i_p}\{h^{(l)}(X_{r_p};Y_1,\ldots,Y_{m-1})\mid X_{r_p}\}\label{eq:h2-bootstrap-lemma-proofuse-2}
\end{align}
Using (\ref{eq:h2-bootstrap-lemma-proofuse-1}) and (\ref{eq:h2-bootstrap-lemma-proofuse-2}) we obtain
\begin{align}
E\{h_{2;\bi}(\Xbr)\mid X_{r_p}\} & =E\{h(\Xbr)\mid X_{r_p}\}-\sum_{\substack{l=1\\
l\neq p
}
}^m\theta^{(l)}(r_l;\bi\backslash i_l)\nonumber  \\
- & E_{\bi\backslash i_p}\{h^{(l)}(X_{r_p};Y_1,\ldots,Y_{m-1})\mid X_{r_p}\}+(m-1)\theta(\bi).\label{eq:h2-bootstrap-lemma-proofuse-3}
\end{align}
We introduce some notation: 
\begin{align*}
\bi\backslash i_l\oplus k & :=(i_1,\ldots,i_{l-1},k,i_{l+1},\ldots,i_m), \\
X_{\bi\backslash i_l\oplus k} & :=(X_{i_1},\ldots,X_{i_{l-1}},X_k,X_{i_{l+1}},\ldots,X_{i_m}), \\
\theta(\bi\mid i_l) & := E\{h(\Xbi)\mid X_{i_l}\}.
\end{align*}
Using the new notation, (\ref{eq:h2-bootstrap-lemma-proofuse-3}) becomes
\begin{equation}
E\{h_{2;\bi}(\Xbr)\mid X_{r_p}\}=\theta(\br\mid r_p)-\sum_{\substack{l=1\\
l\neq p
}
}^m\theta(\bi\backslash i_l\oplus r_l)-\theta(\bi\backslash i_p\oplus r_p\mid r_p)+(m-1)\theta(\bi).\label{eq:h2-bootstrap-lemma-proofuse-4}
\end{equation}
 Similarly, we have
\begin{equation}
E\{h_{2;\bj}(\Xbs)\mid X_{s_q}\}=\theta(\bs\mid s_q)-\sum_{\substack{l=1\\
l\neq q
}
}^m\theta(\bj\backslash j_l\oplus s_l)-\theta(\bj\backslash j_q\oplus s_q\mid s_q)+(m-1)\theta(\bj).\label{eq:h2-bootstrap-lemma-proofuse-5}
\end{equation}

By algebra and the law of iterated expectation, we derive from (\ref{eq:h2-bootstrap-lemma-proofuse-4}) and (\ref{eq:h2-bootstrap-lemma-proofuse-5}) that 
\begin{equation}
E[E\{h_{2;\bi}(\Xbr)\mid X_{r_p}\}E\{h_{2;\bj}(\Xbs)\mid X_{s_q}\}]=T_1+T_2+T_{3}+T_{4}+T_{5},\label{eq:h2-bootstrap-lemma-proofuse-6}
\end{equation}
where
\begin{align*}
T_1 & =E\{\theta(\br\mid r_p)\theta(\bs\mid s_q)\}-E\{\theta(\br\mid r_p)\theta(\bj\backslash j_q\oplus s_q\mid s_q)\}\\
 & \ \ \ -E\{\theta(\bi\backslash i_p\oplus r_p\mid r_p)\theta(\bs\mid s_q)\}+E\{\theta(\bi\backslash i_p\oplus r_p\mid r_p)\theta(\bj\backslash j_q\oplus s_q\mid s_q)\},\\
T_2 & =(m-1)\theta(\br)\theta(\bj)-\theta(\br)\sum_{l\neq q}\theta(\bj\backslash j_l\oplus s_l)+(m-1)\theta(\bi)\theta(\bs)-\theta(\bs)\sum_{l\neq p}\theta(\bi\backslash i_l\oplus r_l),\\
T_{3} & =\Big\{\sum_{l=1}^m\theta(\bi\backslash i_l\oplus r_l)-m\theta(\bi)\Big\}\Big\{\sum_{l=1}^m\theta(\bj\backslash j_l\oplus s_l)-m\theta(\bj)\Big\},\\
T_{4} & =\theta(\bi)\sum_{l=1}^m\theta(\bj\backslash j_l\oplus s_l)+\theta(\bj)\sum_{l=1}^m\theta(\bi\backslash i_l\oplus r_l)-2m\theta(\bi)\theta(\bj),\\
T_{5} & =\theta(\bi)\theta(\bj)-\theta(\bi\backslash i_p\oplus r_p)\theta(\bj\backslash j_q\oplus s_q).
\end{align*}
By the definitions of $M_1(n)$ and $M_2(n)$ in (\ref{eq:def-M1})
and (\ref{eq:def-M2}), we have $|T_1|\leq2M_2(n)$, $|T_2|\leq CM_1(n)$,
$|T_{3}|\leq CM_1(n)^2$, $|T_{4}|\leq CM_1(n)$, and $|T_{5}|\leq CM_1(n)$.
Therefore, it follows from (\ref{eq:h2-bootstrap-lemma-proofuse-6})
that
\[
|E[E\{h_{2;\bi}(\Xbr)\mid X_{r_p}\}E\{h_{2;\bj}(\Xbs)\mid X_{s_q}\}]|\leq C\{M_1(n)^2+M_2(n)\}.
\]
This yields (\ref{eq:h2-bootstrap-proofuse-lemma}). The proof is
thus finished.
\end{proof}


\begin{proof}[Proof of Lemma \ref{lem:h2-bootstrap-proofuse-1}]
For a fixed $(\bi,\bj)\in(I_n^m)_{=1}^{\otimes2}$, consider any
$(\br,\bs)\in(I_n^m)_{=0}^{\otimes2}$ such that $\bi\cap\br=\emptyset=\bj\cap\bs$.
By independence of the $X_i$'s we have
\begin{equation}
E\{h_{2;\bi}(\Xbr)h_{2;\bj}(\Xbs)\}=E\{h_{2;\bi}(\Xbr)\}E\{h_{2;\bj}(\Xbs)\}.\label{eq:h2-bootstrap-lemma-proofuse-7}
\end{equation}
By the definition of $h_{2;\bi}(\cdot)$ in (\ref{eq:def-h2}), we have
\begin{align*}
E\{h_{2;\bi}(\Xbr)\} & =E\{h(\Xbr)\}-\sum_{l=1}^mE[E_{\bi\backslash i_l}\{h^{(l)}(X_{r_l};Y_1,\ldots,Y_{m-1})\mid X_{r_l}\}]+(m-1)\theta(\bi)\\
 & =\theta(\br)-\sum_{l=1}^m\theta^{(l)}(r_l;\bi\backslash i_l)+(m-1)\theta(\bi).
\end{align*}
It then follows from the definition of $M_1(n)$ in (\ref{eq:def-M1})
that
\begin{equation}
|E\{h_{2;\bi}(\Xbr)\}|\leq mM_1(n).\label{eq:h2-bootstrap-lemma-proofuse-8}
\end{equation}
Combining (\ref{eq:h2-bootstrap-lemma-proofuse-7}) and (\ref{eq:h2-bootstrap-lemma-proofuse-8}) yields that
\[
|E\{h_{2;\bi}(\Xbr)h_{2;\bj}(\Xbs)\}|\leq m^2M_1(n).
\]
This implies (\ref{eq:h2-bootstrap-proofuse-lemma-1}). The proof is thus finished.
\end{proof}


\subsection{Proof of Lemma \ref{lem:tauap-AN-proofuse}}
\begin{proof}
Define 
\begin{equation}
f_{ij}(x):=P(X_j>x)-\theta(i,j),~~S_i^{(1)}(x):=\sum_{j=1}^{i-1}\frac{n}{i-1}f_{ij}(x),~~{\rm and}~~S_i^{(2)}(x):=\sum_{j=i+1}^n\frac{n}{j-1}f_{ij}(x).\label{eq:def-S1S2-tauap}
\end{equation}
By (\ref{eq:tauap-h1i}) we have $h_{1,i}^{\ap}(X_i)=\{S_i^{(1)}(X_i)-S_i^{(2)}(X_i)\}/(n-1)$
for any $i\in[n]$. In the following we use Lemma \ref{lem:bound-harmonic-sum} repeatedly to bound $\varphi(n):=\sum_{k=1}^nk^{-1}$.

First, we show that (\ref{eq:lemma-bound-h1i-tauap-1}) and (\ref{eq:lemma-bound-h1i-tauap-2}) hold under Condition (i) of Theorem \ref{thm:tauap-taukendall-AN}.
Using $f_{ij}(\cdot)$ notation, Condition (i) becomes
\begin{equation}
P\{\delta_n\leq f_{ij}(X_i)\leq1,\forall j\in[n]\backslash\{i\}\}\geq p_n.\label{eq:condition-i-proofuse-1}
\end{equation}
If $\delta_n\leq f_{ij}(x)\leq1,\forall j\in[n]\backslash\{i\}$,
we have 
\begin{equation}
n\delta_n=\sum_{j=1}^{i-1}\frac{n}{i-1}\delta_n\leq S_i^{(1)}(x)\leq\sum_{j=1}^{i-1}\frac{n}{i-1}=n,\label{eq:bound-tauap-i-1}
\end{equation}
and 
\begin{align}
S_i^{(2)}(x) & \geq\sum_{j=i+1}^n\frac{n}{j-1}\delta_n=n\delta_n\{\varphi(n-1)-\varphi(i-1)\}\geq n\delta_n\log\frac{n}{i},\label{eq:bound-tauap-i-2}\\
S_i^{(2)}(x) & \leq \sum_{j=i+1}^n \frac{n}{j-1}=n\{\varphi(n-1)-\varphi(i-1)\}\leq n\log\frac{n-1}{i-1}.\label{eq:bound-tauap-i-3}
\end{align}
Using (\ref{eq:bound-tauap-i-1}), (\ref{eq:bound-tauap-i-2}), and
(\ref{eq:bound-tauap-i-3}), it follows from (\ref{eq:condition-i-proofuse-1})
that 
\begin{equation}
P\Big\{ n\delta_n\leq S_i^{(1)}(X_i)\leq n,n\delta_n\log\frac{n}{i}\leq S_i^{(2)}(X_i)\leq n\log\frac{n-1}{i-1}\Big\}\geq p_n.\label{eq:condition-i-proofuse-2}
\end{equation}
If $\log\{(n-1)/(i-1)\}\leq\delta_n/2$, the monotonicity property
of probability measure gives
\begin{align}
 & P\{h_{1,i}(X_i)\geq\delta_n/2\}\geq P\Big\{\frac{S_i^{(1)}(X_i)}{n-1}\geq\frac{n}{n-1}\delta_n,\frac{S_i^{(2)}(X_i)}{n-1}\leq\frac{n}{n-1}\log\frac{n-1}{i-1}\Big\}\nonumber  \\
= & P\Big\{ S_i^{(1)}\geq n\delta_n,S_i^{(2)}\leq n\log\frac{n-1}{i-1}\Big\}\geq P\Big\{ n\delta_n\leq S_i^{(1)}\leq n,n\delta_n\log\frac{n}{i}\leq S_i^{(2)}\leq n\log\frac{n-1}{i-1}\Big\}.\label{eq:condition-i-proofuse-3}
\end{align}
Note that 
\begin{equation}
P\{|h_{1,i}(X_i)|\geq\delta_n/2\}\geq P\{h_{1,i}(X_i)\geq\delta_n/2\}.\label{eq:condition-i-proofuse-4}
\end{equation}
Equation (\ref{eq:lemma-bound-h1i-tauap-1}) follows from (\ref{eq:condition-i-proofuse-2}), (\ref{eq:condition-i-proofuse-3}), and (\ref{eq:condition-i-proofuse-4}).
If $\delta_n\log(n/i)\geq2$, the monotonicity property of probability measure gives
\begin{align}
 & P\{h_{1,i}(X_i)\leq-1\}\geq P\Big\{\frac{S_i^{(1)}(X_i)}{n-1}\leq\frac{n}{n-1},\frac{S_i^{(2)}(X_i)}{n-1}\geq\frac{n}{n-1}\delta_n\log\frac{n}{i}\Big\}\nonumber  \\
= & P\Big\{ S_i^{(1)}\leq n,S_i^{(2)}\geq n\delta_n\log\frac{n}{i}\Big\}\geq P\Big\{ n\delta_n\leq S_i^{(1)}\leq n,n\delta_n\log\frac{n}{i}\leq S_i^{(2)}\leq n\log\frac{n-1}{i-1}\Big\}.\label{eq:condition-i-proofuse-5}
\end{align}
Note that 
\begin{equation}
P\{|h_{1,i}(X_i)|\geq1\}\geq P\{h_{1,i}(X_i)\leq-1\}.\label{eq:condition-i-proofuse-6}
\end{equation}
Equation (\ref{eq:lemma-bound-h1i-tauap-2}) follows from (\ref{eq:condition-i-proofuse-2}), (\ref{eq:condition-i-proofuse-5}), and (\ref{eq:condition-i-proofuse-6}).

\medskip{}

Secondly, we show that (\ref{eq:lemma-bound-h1i-tauap-1}) and (\ref{eq:lemma-bound-h1i-tauap-2}) hold under Condition (ii) of Theorem \ref{thm:tauap-taukendall-AN}.
Using $f_{ij}(\cdot)$ notation, Condition (ii) becomes
\begin{equation}
P\{-1\leq f_{ij}(X_i)\leq-\delta_n,\forall j\in[n]\backslash\{i\}\}\geq p_n.\label{eq:condition-ii-proofuse-1}
\end{equation}
By an argument similar to (\ref{eq:bound-tauap-i-1})-(\ref{eq:bound-tauap-i-3}),
if $-1\leq f_{ij}(x)\leq-\delta_n,\forall j\in[n]\backslash\{i\}$
we have
\begin{equation}
-n\leq S_i^{(1)}(x)\leq-n\delta_n,~~{\rm and}~~-n\log\frac{n-1}{i-1}\leq S_i^{(2)}(x)\leq-n\delta_n\log\frac{n}{i}.\label{eq:bound-tauap-ii}
\end{equation}
By (\ref{eq:bound-tauap-ii}), Condition (ii) in Theorem \ref{thm:tauap-taukendall-AN}
implies that 
\begin{equation}
P\Big\{-n\leq S_i^{(1)}(X_i)\leq-n\delta_n,-n\log\frac{n-1}{i-1}\leq S_i^{(2)}(X_i)\leq-n\delta_n\log\frac{n}{i}\Big\}\geq p_n.\label{eq:condition-ii-proofuse}
\end{equation}
If $\log\{(n-1)/(i-1)\}\leq\delta_n/2$, the monotonicity property
of probability measure gives
\begin{align}
 & P\{h_{1,i}(X_i)\leq-\delta_n/2\}\geq P\Big\{\frac{S_i^{(1)}(X_i)}{n-1}\leq-\frac{n}{n-1}\delta_n,\frac{S_i^{(2)}(X_i)}{n-1}\geq-\frac{n}{n-1}\log\frac{n-1}{i-1}\Big\}\nonumber \\
= & P\Big\{ S_i^{(1)}\leq-n\delta_n,S_i^{(2)}\geq-n\log\frac{n-1}{i-1}\Big\}\nonumber \\
\geq & P\Big\{-n\leq S_i^{(1)}\leq-n\delta_n,-n\log\frac{n-1}{i-1}\leq S_i^{(2)}\leq-n\delta_n\log\frac{n}{i}\Big\}.\label{eq:condition-ii-proofuse-3}
\end{align}
Note that 
\begin{equation}
P\{|h_{1,i}(X_i)|\geq\delta_n/2\}\geq P\{h_{1,i}(X_i)\leq-\delta_n/2\}.\label{eq:condition-ii-proofuse-4}
\end{equation}
Equation (\ref{eq:lemma-bound-h1i-tauap-1}) follows from (\ref{eq:condition-ii-proofuse}),
(\ref{eq:condition-ii-proofuse-3}), and (\ref{eq:condition-ii-proofuse-4}).
If $\delta_n\log(n/i)\geq2$, the monotonicity property of probability
measure gives
\begin{align}
 & P\{h_{1,i}(X_i)\geq1\}\geq P\Big\{\frac{S_i^{(1)}(X_i)}{n-1}\geq-\frac{n}{n-1},\frac{S_i^{(2)}(X_i)}{n-1}\leq-\frac{n}{n-1}\delta_n\log\frac{n}{i}\Big\}\nonumber \\
= & P\Big\{ S_i^{(1)}\geq-n,S_i^{(2)}\leq-n\delta_n\log\frac{n}{i}\Big\}\nonumber  \\
\geq & P\Big\{-n\leq S_i^{(1)}\leq-n\delta_n,-n\log\frac{n-1}{i-1}\leq S_i^{(2)}\leq-n\delta_n\log\frac{n}{i}\Big\}.\label{eq:condition-ii-proofuse-5}
\end{align}
Note that 
\begin{equation}
P\{|h_{1,i}(X_i)|\geq1\}\geq P\{h_{1,i}(X_i)\geq1\}.\label{eq:condition-ii-proofuse-6}
\end{equation}
Equation (\ref{eq:lemma-bound-h1i-tauap-2}) follows from (\ref{eq:condition-ii-proofuse}),
(\ref{eq:condition-ii-proofuse-5}), and (\ref{eq:condition-ii-proofuse-6}).

This completes the proof.
\end{proof}


\subsection{Proof of Lemma \ref{lem:taukendall-AN-proofuse}}
\begin{proof}
Define 
\begin{equation}
f_{ij}(x):=P(X_j>x)-\theta(i,j),~~S_i^{(1)}(x):=\sum_{j=1}^{i-1}f_{ij}(x),~~{\rm and}~~S_i^{(2)}(x):=\sum_{j=i+1}^nf_{ij}(x).\label{eq:def-S1S2-tauap-1}
\end{equation}
By (\ref{eq:taukendall-h1i}) we have $h_{1,i}^{\ap}(X_i)=\{S_i^{(1)}(X_i)-S_i^{(2)}(X_i)\}/(n-1)$ for $2\leq i\leq n$.

First, we show that (\ref{eq:lemma-bound-h1i-taukendall-1}) and (\ref{eq:lemma-bound-h1i-taukendall-2}) hold under Condition (i) of Theorem \ref{thm:tauap-taukendall-AN}.
Using $f_{ij}(\cdot)$ notation, Condition (i) becomes
\begin{equation}
P\{\delta_n\leq f_{ij}(X_i)\leq1,\forall j\in[n]\backslash\{i\}\}\geq p_n.\label{eq:condition-i-proofuse-1-1}
\end{equation}
If $\delta_n\leq f_{ij}(x)\leq1,\forall j\in[n]\backslash\{i\}$,
we have 
\begin{equation}
(i-1)\delta_n=\sum_{j=1}^{i-1}\delta_n\leq S_i^{(1)}(x)\leq\sum_{j=1}^{i-1}1=i-1,\label{eq:bound-tauap-i-1-1}
\end{equation}
and
\begin{equation}
(n-i)\delta_n=\sum_{j=i+1}^n\delta_n\leq S_i^{(2)}(x)\leq\sum_{j=i+1}^n1=n-i.\label{eq:bound-tauap-i-2-1}
\end{equation}
Using (\ref{eq:bound-tauap-i-1-1}) and (\ref{eq:bound-tauap-i-2-1}), it follows from (\ref{eq:condition-i-proofuse-1-1}) that 
\begin{equation}
P\Big\{(i-1)\delta_n\leq S_i^{(1)}(X_i)\leq i-1,(n-i)\delta_n\leq S_i^{(2)}(X_i)\leq n-i\Big\}\geq p_n.\label{eq:condition-i-proofuse-2-1}
\end{equation}
If $n-i\leq(i-1)\delta_n/2$, the monotonicity property of probability
measure gives
\begin{align}
 & P\Big\{ h_{1,i}(X_i)\geq\frac{i-1}{n-1}\frac{\delta_n}{2}\Big\}\geq P\Big\{\frac{S_i^{(1)}(X_i)}{n-1}\geq\frac{i-1}{n-1}\delta_n,\frac{S_i^{(2)}(X_i)}{n-1}\leq\frac{n-i}{n-1}\Big\}\nonumber  \\
= & P\Big\{ S_i^{(1)}\geq(i-1)\delta_n,S_i^{(2)}\leq n-i\Big\}\nonumber  \\
\geq & P\Big\{(i-1)\delta_n\leq S_i^{(1)}\leq i-1,(n-i)\delta_n\leq S_i^{(2)}\leq n-i\Big\}.\label{eq:condition-i-proofuse-3-1}
\end{align}
Note that 
\begin{equation}
P\Big\{|h_{1,i}(X_i)|\geq\frac{i-1}{n-1}\frac{\delta_n}{2}\Big\}\geq P\Big\{ h_{1,i}(X_i)\geq\frac{i-1}{n-1}\frac{\delta_n}{2}\Big\}.\label{eq:condition-i-proofuse-4-1}
\end{equation}
Equation (\ref{eq:lemma-bound-h1i-taukendall-1}) follows from (\ref{eq:condition-i-proofuse-2-1}), (\ref{eq:condition-i-proofuse-3-1}), and (\ref{eq:condition-i-proofuse-4-1}).
If $i-1\leq(n-i)\delta_n/2$, the monotonicity property of probability measure gives
\begin{align}
 & P\{h_{1,i}(X_i)\leq-\frac{n-i}{n-1}\frac{\delta_n}{2}\}\geq P\Big\{\frac{S_i^{(1)}(X_i)}{n-1}\leq\frac{i-1}{n-1},\frac{S_i^{(2)}(X_i)}{n-1}\geq\frac{n-i}{n-1}\delta_n\Big\}\nonumber  \\
= & P\Big\{ S_i^{(1)}\leq i-1,S_i^{(2)}\geq(n-i)\delta_n\Big\}\nonumber \\
\geq & P\Big\{(i-1)\delta_n\leq S_i^{(1)}\leq i-1,(n-i)\delta_n\leq S_i^{(2)}\leq n-i\Big\}.\label{eq:condition-i-proofuse-5-1}
\end{align}
Note that 
\begin{equation}
P\Big\{|h_{1,i}(X_i)|\geq\frac{n-i}{n-1}\frac{\delta_n}{2}\Big\}\geq P\Big\{ h_{1,i}(X_i)\leq-\frac{n-i}{n-1}\frac{\delta_n}{2}\Big\}.\label{eq:condition-i-proofuse-6-1}
\end{equation}
Equation (\ref{eq:lemma-bound-h1i-taukendall-2}) follows from (\ref{eq:condition-i-proofuse-2-1}), (\ref{eq:condition-i-proofuse-5-1}), and (\ref{eq:condition-i-proofuse-6-1}).

\medskip{}

Secondly, we show that (\ref{eq:lemma-bound-h1i-taukendall-1}) and (\ref{eq:lemma-bound-h1i-taukendall-2}) hold under Condition (ii) of Theorem \ref{thm:tauap-taukendall-AN}.
Using $f_{ij}(\cdot)$ notation, Condition (ii) becomes
\begin{equation}
P\{-1\leq f_{ij}(X_i)\leq-\delta_n,\forall j\in[n]\backslash\{i\}\}\geq p_n.\label{eq:condition-ii-proofuse-1-1}
\end{equation}
If $-1\leq f_{ij}(x)\leq-\delta_n,\forall j\in[n]\backslash\{i\}$,
we have 
\begin{equation}
-(i-1)=-\sum_{j=1}^{i-1}1\leq S_i^{(1)}(x)\leq-\sum_{j=1}^{i-1}\delta_n=-(i-1)\delta_n,\label{eq:bound-tauap-i-1-1-1}
\end{equation}
and
\begin{equation}
-(n-i)=-\sum_{j=i+1}^n1\leq S_i^{(2)}(x)\leq-\sum_{j=i+1}^n\delta_n=-(n-i)\delta_n.\label{eq:bound-tauap-i-2-1-1}
\end{equation}
Using (\ref{eq:bound-tauap-i-1-1-1}) and (\ref{eq:bound-tauap-i-2-1-1}), it follows from (\ref{eq:condition-i-proofuse-1-1}) that 
\begin{equation}
P\Big\{-(i-1)\leq S_i^{(1)}(X_i)\leq-(i-1)\delta_n,-(n-i)\leq S_i^{(2)}(X_i)\leq-(n-i)\delta_n\Big\}\geq p_n.\label{eq:condition-i-proofuse-2-1-1}
\end{equation}
If $n-i\leq(i-1)\delta_n/2$, the monotonicity property of probability
measure gives
\begin{align}
 & P\Big\{ h_{1,i}(X_i)\leq-\frac{i-1}{n-1}\frac{\delta_n}{2}\Big\}\geq P\Big\{\frac{S_i^{(1)}(X_i)}{n-1}\leq-\frac{i-1}{n-1}\delta_n,\frac{S_i^{(2)}(X_i)}{n-1}\geq-\frac{n-i}{n-1}\Big\}\nonumber \\
= & P\Big\{ S_i^{(1)}\leq-(i-1)\delta_n,S_i^{(2)}\geq-(n-i)\Big\}\nonumber  \\
\geq & P\Big\{-(i-1)\leq S_i^{(1)}\leq-(i-1)\delta_n,-(n-i)\leq S_i^{(2)}\leq-(n-i)\delta_n\Big\}.\label{eq:condition-i-proofuse-3-1-1}
\end{align}
Note that 
\begin{equation}
P\Big\{|h_{1,i}(X_i)|\geq\frac{i-1}{n-1}\frac{\delta_n}{2}\Big\}\geq P\Big\{ h_{1,i}(X_i)\leq-\frac{i-1}{n-1}\frac{\delta_n}{2}\Big\}.\label{eq:condition-i-proofuse-4-1-1}
\end{equation}
Equation (\ref{eq:lemma-bound-h1i-taukendall-1}) follows from (\ref{eq:condition-i-proofuse-2-1-1}), (\ref{eq:condition-i-proofuse-3-1-1}), and (\ref{eq:condition-i-proofuse-4-1-1}).
If $i-1\leq(n-i)\delta_n/2$, the monotonicity property of probability measure gives
\begin{align}
 & P\{h_{1,i}(X_i)\geq\frac{n-i}{n-1}\frac{\delta_n}{2}\}\geq P\Big\{\frac{S_i^{(1)}(X_i)}{n-1}\geq-\frac{i-1}{n-1},\frac{S_i^{(2)}(X_i)}{n-1}\leq-\frac{n-i}{n-1}\delta_n\Big\}\nonumber  \\
= & P\Big\{ S_i^{(1)}\geq-(i-1),S_i^{(2)}\leq-(n-i)\delta_n\Big\}\nonumber  \\
\geq & P\Big\{-(i-1)\leq S_i^{(1)}\leq-(i-1)\delta_n,-(n-i)\leq S_i^{(2)}\leq-(n-i)\delta_n\Big\}.\label{eq:condition-i-proofuse-5-1-1}
\end{align}
Note that 
\begin{equation}
P\Big\{|h_{1,i}(X_i)|\geq\frac{n-i}{n-1}\frac{\delta_n}{2}\Big\}\geq P\Big\{ h_{1,i}(X_i)\geq\frac{n-i}{n-1}\frac{\delta_n}{2}\Big\}.\label{eq:condition-i-proofuse-6-1-1}
\end{equation}
Equation (\ref{eq:lemma-bound-h1i-taukendall-2}) follows from (\ref{eq:condition-i-proofuse-2-1-1}), (\ref{eq:condition-i-proofuse-5-1-1}), and (\ref{eq:condition-i-proofuse-6-1-1}).

This completes the proof.
\end{proof}


\subsection{Proof of Lemma \ref{lem:tail-condition-lemma-1}}
\begin{proof}
As in the statement of Lemma \ref{lem:tail-condition-lemma-1}, we consider a fixed $i\in[n]$.
For any $j\in[n]\backslash\{i\}$, we have $\rho_{ij}^{-1}\leq\rho_n$ and $-r_{ij}\leq R_n$.
This combined with (\ref{eq:tail-condition-lemma-1-zi}) implies that $z_i\geq\rho_{ij}^{-1}t_0-r_{ij}$, or equivalently
\begin{equation}
\rho_{ij}(z_i+r_{ij})\geq t_0.\label{eq:tail-condition-lemma-1-proofuse1}
\end{equation}
Equations (\ref{eq:tail-condition-lemma-1-proofuse1}) and (\ref{eq:tail-condition-1-1}) imply that
\begin{equation}
F_j^{c}\{\rho_{ij}(z_i+r_{ij})\}\leq c_2\{\rho_{ij}(z_i+r_{ij})\}^{-b_2}.\label{eq:tail-condition-lemma-1-proofuse2}
\end{equation}
Define
\[
\delta_n:=\min\Big\{\frac{c_1}{2}R_n^{-b_1},  \frac{c_1}{2}t_0^{-b_1}, \frac{1}{2}\Big\}.
\]
This implies that $\delta_n\in(0,1)$ and
\begin{align}
-\frac{\delta_n}{c_2}+\frac{c_1}{c_2}t_0^{-b_1} & \geq\frac{c_1}{2c_2}t_0^{-b_1},\label{eq:tail-condition-lemma-1-delta-1}\\
~~{\rm and}~~-\frac{\delta_n}{c_2}+\frac{c_1}{c_2}R_n^{-b_1} & \geq\frac{c_1}{2c_2}R_n^{-b_1}.\label{eq:tail-condition-lemma-1-delta-2}
\end{align}
For an arbitrary $j\in[n]\backslash\{i\}$, either $r_{ij}(1+\rho_{ij}^{-2})^{-1/2}\leq t_0$ or $r_{ij}(1+\rho_{ij}^{-2})^{-1/2}>t_0$ holds.
In the following we show $f_{ij}(x)\leq-\delta_n$ for all $j\in[n]\backslash\{i\}$ under these two mutually exclusive and collectively exhaustive cases.

\medskip{}

\textbf{Case 1}: Assume that for a fixed $j$ we have
\begin{equation}
r_{ij}(1+\rho_{ij}^{-2})^{-1/2}\leq t_0.\label{eq:tail-condition-lemma-1-proofuse3}
\end{equation}
By the monotonicity of $F_{ji}^{c}(\cdot)$ we have
\begin{equation}
F_{ji}^{c}\{r_{ij}(1+\rho_{ij}^{-2})^{-1/2}\}\geq F_{ji}^{c}(t_0).\label{eq:tail-condition-lemma-1-proofuse4}
\end{equation}
By (\ref{eq:tail-condition-1-2}) we have
\begin{equation}
F_{ji}^{c}(t_0)\geq c_1t_0^{-b_1}.\label{eq:tail-condition-lemma-1-proofuse5}
\end{equation}
Combining (\ref{eq:tail-condition-lemma-1-proofuse4}) and (\ref{eq:tail-condition-lemma-1-proofuse5}) yields
\begin{equation}
F_{ji}^{c}\{r_{ij}(1+\rho_{ij}^{-2})^{-1/2}\}\geq c_1t_0^{-b_1}.\label{eq:tail-condition-lemma-1-proofuse6}
\end{equation}
Combining (\ref{eq:tail-condition-proofuse-fijFjFji}), (\ref{eq:tail-condition-lemma-1-proofuse2}),
and (\ref{eq:tail-condition-lemma-1-proofuse6}) gives
\begin{equation}
f_{ij}(x)\leq c_2\{\rho_{ij}(z_i+r_{ij})\}^{-b_2}-c_1t_0^{-b_1}.\label{eq:tail-condition-lemma-1-proofuse7}
\end{equation}
Equation (\ref{eq:tail-condition-lemma-1-delta-1}) implies
\begin{equation}
\Big(-\frac{\delta_n}{c_2}+\frac{c_1}{c_2}t_0^{-b_1}\Big)^{-1/b_2}\leq\Big(t_0^{-b_1}\frac{c_1}{2c_2}\Big)^{-1/b_2}.\label{eq:tail-condition-lemma-1-proofuse8}
\end{equation}
Noting that $t_0>0$ and $R_n\geq-r_{ij}$, (\ref{eq:tail-condition-lemma-1-zi}) implies
\begin{equation}
z_i\geq-r_{ij}+\Big(t_0^{-b_1}\frac{c_1}{2c_2}\Big)^{-1/b_2}\rho_n,\label{eq:tail-condition-lemma-1-proofuse9}
\end{equation}
Combining (\ref{eq:tail-condition-lemma-1-proofuse8}) and (\ref{eq:tail-condition-lemma-1-proofuse9})
gives 
\begin{equation}
\rho_{ij}(z_i+r_{ij})\geq \Big(-\frac{\delta_n}{c_2}+\frac{c_1}{c_2}t_0^{-b_1}\Big)^{-1/b_2}.\label{eq:tail-condition-lemma-1-proofuse10}
\end{equation}
Therefore, by (\ref{eq:tail-condition-lemma-1-proofuse7}) and (\ref{eq:tail-condition-lemma-1-proofuse10})
we deduce
\[
f_{ij}(x)\leq-\delta_n+c_1t_0^{-b_1}-c_1t_0^{-b_1}=-\delta_n.
\]

\medskip{}

\textbf{Case 2}: Assume that for a fixed $j$ we have
\begin{equation}
r_{ij}(1+\rho_{ij}^{-2})^{-1/2}>t_0.\label{eq:tail-condition-lemma-1-proofuse11}
\end{equation}
By (\ref{eq:tail-condition-1-2}) we have
\begin{equation}
F_{ji}^{c}\{r_{ij}(1+\rho_{ij}^{-2})^{-1/2}\}\geq c_1\{r_{ij}(1+\rho_{ij}^{-2})^{-1/2}\}^{-b_1}.\label{eq:tail-condition-lemma-1-proofuse12}
\end{equation}
Combining (\ref{eq:tail-condition-proofuse-fijFjFji}), (\ref{eq:tail-condition-lemma-1-proofuse2}),
and (\ref{eq:tail-condition-lemma-1-proofuse12}) gives
\begin{equation}
f_{ij}(x)\leq c_2\{\rho_{ij}(z_i+r_{ij})\}^{-b_2}-c_1\{r_{ij}(1+\rho_{ij}^{-2})^{-1/2}\}^{-b_1}.\label{eq:tail-condition-lemma-1-proofuse13}
\end{equation}
Equation (\ref{eq:tail-condition-lemma-1-delta-2}) implies
\begin{equation}
\Big(-\frac{\delta_n}{c_2}+\frac{c_1}{c_2}R_n^{-b_1}\Big)^{-1/b_2}\leq\Big(\frac{c_1}{2c_2}R_n^{-b_1}\Big)^{-1/b_2}.\label{eq:tail-condition-lemma-1-proofuse14}
\end{equation}
Noting that $t_0>0$ and $R_n\geq-r_{ij}$, (\ref{eq:tail-condition-lemma-1-zi})
implies
\begin{equation}
z_i\geq-r_{ij}+\rho_{ij}^{-1}\Big(\frac{c_1}{2c_2}\Big)^{-1/b_2}R_n^{b_1/b_2}.\label{eq:tail-condition-lemma-1-proofuse15}
\end{equation}
Combining (\ref{eq:tail-condition-lemma-1-proofuse14}) and (\ref{eq:tail-condition-lemma-1-proofuse15})
gives 
\begin{equation}
\rho_{ij}(z_i+r_{ij})\geq \Big(-\frac{\delta_n}{c_2}+\frac{c_1}{c_2}R_n^{-b_1}\Big)^{-1/b_2}.\label{eq:tail-condition-lemma-1-proofuse16}
\end{equation}
Equation (\ref{eq:tail-condition-lemma-1-proofuse16}) implies
\begin{equation}
c_2\{\rho_{ij}(z_i+r_{ij})\}^{-b_2}\leq-\delta_n+c_1R_n^{-b_1}.\label{eq:tail-condition-lemma-1-proofuse17}
\end{equation}
Since $r_{ij}\leq R_n$ and $(1+\rho_{ij}^{-2})^{-1/2}\leq1$, we
have
\begin{equation}
c_1\{r_{ij}(1+\rho_{ij}^{-2})^{-1/2}\}^{-b_1}\geq c_1R_n^{-b_1}.\label{eq:tail-condition-lemma-1-proofuse18}
\end{equation}
Therefore, by (\ref{eq:tail-condition-lemma-1-proofuse13}), (\ref{eq:tail-condition-lemma-1-proofuse17}), and (\ref{eq:tail-condition-lemma-1-proofuse18}) we deduce
\[
f_{ij}(x)\leq-\delta_n+c_1R_n^{-b_1}-c_1R_n^{-b_1}=-\delta_n.
\]

This completes the proof.
\end{proof}


\subsection{Proof of Lemma \ref{lem:tail-condition-lemma-2}}

\begin{proof}
For any $j\in[n]\backslash\{i\}$, we have $\rho_{ij}^{-1}\leq\rho_n$
and $-r_{ij}\leq R_n$. This combined with (\ref{eq:tail-condition-lemma-2-zi})
implies that $z_i\geq\rho_{ij}^{-1}t_0-r_{ij}$, or equivalently
\begin{equation}
\rho_{ij}(z_i+r_{ij})\geq t_0.\label{eq:tail-condition-lemma-2-proofuse1}
\end{equation}
Equations (\ref{eq:tail-condition-lemma-2-proofuse1}) and (\ref{eq:tail-condition-2-1})
imply that
\begin{equation}
F_j^{c}\{\rho_{ij}(z_i+r_{ij})\}\leq c_2\exp[-b_2\{\rho_{ij}(z_i+r_{ij})\}^{\lambda}].\label{eq:tail-condition-lemma-2-proofuse2}
\end{equation}
Define 
\[
\delta_n:=\min\Big\{ \frac{c_1}{2}\exp(-b_1R_n^{\lambda}), \frac{c_1}{2}\exp(-b_1t_0^{\lambda}), \frac{1}{2}\Big\}.
\]
This implies that $\delta_n\in(0,1)$ and that
\begin{align}
-\frac{\delta_n}{c_2}+\frac{c_1}{c_2}\exp(-b_1t_0^{\lambda}) & \geq\frac{c_1}{2c_2}\exp(-b_1t_0^{\lambda})\label{eq:tail-condition-lemma-2-delta-1}\\
~~{\rm and}~~-\frac{\delta_n}{c_2}+\frac{c_1}{c_2}\exp(-b_1R_n^{\lambda}) & \geq\frac{c_1}{2c_2}\exp(-b_1R_n^{\lambda}).\label{eq:tail-condition-lemma-2-delta-2}
\end{align}
In the following we show $f_{ij}(x)\leq-\delta_n$ for all $j\in[n]\backslash\{i\}$ under these two mutually exclusive and collectively exhaustive cases.

\medskip{}

\textbf{Case 1}: Assume that for a fixed $j$ we have
\begin{equation}
r_{ij}(1+\rho_{ij}^{-2})^{-1/2}\leq t_0.\label{eq:tail-condition-lemma-2-proofuse3}
\end{equation}
By the monotonicity of $F_{ji}^{c}(\cdot)$ we have
\begin{equation}
F_{ji}^{c}\{r_{ij}(1+\rho_{ij}^{-2})^{-1/2}\}\geq F_{ji}^{c}(t_0).\label{eq:tail-condition-lemma-2-proofuse4}
\end{equation}
By (\ref{eq:tail-condition-2-2}) we have
\begin{equation}
F_{ji}^{c}(t_0)\geq c_1\exp(-b_1t_0^{\lambda}).\label{eq:tail-condition-lemma-2-proofuse5}
\end{equation}
Combining (\ref{eq:tail-condition-lemma-2-proofuse4}) and (\ref{eq:tail-condition-lemma-2-proofuse5})
yields
\begin{equation}
F_{ji}^{c}\{r_{ij}(1+\rho_{ij}^{-2})^{-1/2}\}\geq c_1\exp(-b_1t_0^{\lambda}).\label{eq:tail-condition-lemma-2-proofuse6}
\end{equation}
Combining (\ref{eq:tail-condition-proofuse-fijFjFji}), (\ref{eq:tail-condition-lemma-2-proofuse2}),
and (\ref{eq:tail-condition-lemma-2-proofuse6}) gives
\begin{equation}
f_{ij}(x)\leq c_2\exp[-b_2\{\rho_{ij}(z_i+r_{ij})\}^{\lambda}]-c_1\exp(-b_1t_0^{\lambda}).\label{eq:tail-condition-lemma-2-proofuse7}
\end{equation}
Equation (\ref{eq:tail-condition-lemma-2-delta-1}) implies
\begin{equation}
-\frac{1}{b_2}\log\{-\frac{\delta_n}{c_2}+\frac{c_1}{c_2}\exp(-b_1t_0^{\lambda})\}\leq-\frac{1}{b_2}\log\frac{c_1}{2c_2}+\frac{b_1}{b_2}t_0^{\lambda}.\label{eq:tail-condition-lemma-2-proofuse8}
\end{equation}
Noting that $t_0>0$ and $R_n\geq-r_{ij}$, (\ref{eq:tail-condition-lemma-2-zi})
implies
\begin{equation}
z_i\geq-r_{ij}+\rho_{ij}^{-1}K_{3}\geq-r_{ij}+\rho_{ij}^{-1}\Big(-\frac{1}{b_2}\log\frac{c_1}{2c_2}+\frac{b_1}{b_2}t_0^{\lambda}\Big)^{1/\lambda}.\label{eq:tail-condition-lemma-2-proofuse9}
\end{equation}
Combining (\ref{eq:tail-condition-lemma-2-proofuse8}) and (\ref{eq:tail-condition-lemma-2-proofuse9})
gives 
\begin{equation}
\rho_{ij}(z_i+r_{ij})\geq \Big[-\frac{1}{b_2}\log\Big\{-\frac{\delta_n}{c_2}+\frac{c_1}{c_2}\exp(-b_1t_0^{\lambda})\Big\}\Big]^{1/\lambda}.\label{eq:tail-condition-lemma-2-proofuse10}
\end{equation}
Therefore, by (\ref{eq:tail-condition-lemma-2-proofuse7}) and (\ref{eq:tail-condition-lemma-2-proofuse10})
we deduce
\[
f_{ij}(x)\leq-\delta_n+c_1\exp(-b_1t_0^{\lambda})-c_1\exp(-b_1t_0^{\lambda})=-\delta_n.
\]

\medskip{}

\textbf{Case 2}: Assume that for a fixed $j$ we have
\begin{equation}
r_{ij}(1+\rho_{ij}^{-2})^{-1/2}>t_0.\label{eq:tail-condition-lemma-2-proofuse11}
\end{equation}
By (\ref{eq:tail-condition-2-2}) we have
\begin{equation}
F_{ji}^{c}\{r_{ij}(1+\rho_{ij}^{-2})^{-1/2}\}\geq c_1\exp\Big[-b_1\Big\{r_{ij}\Big(1+\rho_{ij}^{-2}\Big)^{-1/2}\Big\}^{\lambda}\Big].\label{eq:tail-condition-lemma-2-proofuse12}
\end{equation}
Combining (\ref{eq:tail-condition-proofuse-fijFjFji}), (\ref{eq:tail-condition-lemma-2-proofuse2}),
and (\ref{eq:tail-condition-lemma-2-proofuse12}) gives
\begin{equation}
f_{ij}(x)\leq c_2\exp[-b_2\{\rho_{ij}(z_i+r_{ij})\}^{\lambda}]-c_1\exp[-b_1\{r_{ij}(1+\rho_{ij}^{-2})^{-1/2}\}^{\lambda}].\label{eq:tail-condition-lemma-2-proofuse13}
\end{equation}
Equation (\ref{eq:tail-condition-lemma-2-delta-2}) implies
\begin{equation}
-\frac{1}{b_2}\log\{-\frac{\delta_n}{c_2}+\frac{c_1}{c_2}\exp(-b_1R_n^{\lambda})\}\leq-\frac{1}{b_2}\log\frac{c_1}{2c_2}+\frac{b_1}{b_2}R_n^{\lambda}.\label{eq:tail-condition-lemma-2-proofuse14}
\end{equation}
Equation (\ref{eq:tail-condition-lemma-2-zi}) implies
\begin{equation}
z_i\geq R_n+\rho_n\xi(\lambda^{-1})\Big\{\Big(-\frac{1}{b_2}\log\frac{c_1}{2c_2}\Big)^{1/\lambda}+\Big(\frac{b_1}{b_2}R_n^{\lambda}\Big)^{1/\lambda}\Big\}.\label{eq:tail-condition-lemma-2-proofuse14.5}
\end{equation}
It follows from (\ref{eq:tail-condition-lemma-2-proofuse14.5}) and
Lemma \ref{lem:bound-sum-power} that
\begin{equation}
z_i\geq R_n+\rho_n\Big(-\frac{1}{b_2}\log\frac{c_1}{2c_2}+\frac{b_1}{b_2}R_n^{\lambda}\Big)^{1/\lambda}.\label{eq:tail-condition-lemma-2-proofuse14.6}
\end{equation}
Noting that $t_0>0$ and $R_n\geq-r_{ij}$, (\ref{eq:tail-condition-lemma-2-proofuse14.6})
implies
\begin{equation}
z_i\geq-r_{ij}+\rho_{ij}^{-1}\Big(-\frac{1}{b_2}\log\frac{c_1}{2c_2}+\frac{b_1}{b_2}R_n^{\lambda}\Big)^{1/\lambda}.\label{eq:tail-condition-lemma-2-proofuse15}
\end{equation}
Combining (\ref{eq:tail-condition-lemma-2-proofuse14}) and (\ref{eq:tail-condition-lemma-2-proofuse15})
gives 
\begin{equation}
\rho_{ij}(z_i+r_{ij})\geq\Big[-\frac{1}{b_2}\log\{-\frac{\delta_n}{c_2}+\frac{c_1}{c_2}\exp(-b_1R_n^{\lambda})\}\Big]^{1/\lambda}.\label{eq:tail-condition-lemma-2-proofuse16}
\end{equation}
Equation (\ref{eq:tail-condition-lemma-2-proofuse16}) implies
\begin{equation}
c_2\exp[-b_2\{\rho_{ij}(z_i+r_{ij})\}^{\lambda}]\leq-\delta_n+c_1\exp(-b_1R_n^{\lambda}).\label{eq:tail-condition-lemma-2-proofuse17}
\end{equation}
Since $r_{ij}\leq R_n$ and $(1+\rho_{ij}^{-2})^{-1/2}\leq1$, we
have
\begin{equation}
c_1\exp\Big[-b_1\{r_{ij}(1+\rho_{ij}^{-2})^{-1/2}\}^{\lambda}\Big]\geq c_1\exp(-b_1R_n^{\lambda}).\label{eq:tail-condition-lemma-2-proofuse18}
\end{equation}
Therefore, by (\ref{eq:tail-condition-lemma-2-proofuse13}), (\ref{eq:tail-condition-lemma-2-proofuse17}),
and (\ref{eq:tail-condition-lemma-2-proofuse18}) we deduce
\[
f_{ij}(x)\leq-\delta_n+c_1\exp(-b_1R_n^{\lambda})-c_1\exp(-b_1R_n^{\lambda})=-\delta_n.
\]

This completes the proof.
\end{proof}


\subsection{Proof of Lemma \ref{lem:taukendall-bootstrap-lemma}}
\begin{proof}
Consider an arbitrary vector $(l_1,\ldots,l_n)$, with each $l_i\in[n]$.
Define the sign function $\sgn(x):=\ind(x>0)-\ind(x<0)$. It follows from (\ref{eq:taukendall-h1i}) that 
\begin{align}
\sum_{i=1}^nE\{h_{1,i}^{\ken}(X_{l_i})^2\} & =\frac{1}{(n-1)^2}\sum_{i=1}^nE\Big[\sum_{k=1}^n\sgn(i-k)\{P(X_k>X_{l_i}\mid X_{l_i})-P(X_k>X_i)\}\Big]^2\nonumber  \\
 & =T_1-2T_2+T_{3},\label{eq:taukendall-bootstrap-lemma-proofuse-1}
\end{align}
where
\begin{align}
T_1 & =\frac{1}{(n-1)^2}\sum_{i=1}^n\sum_{k_1,k_2=1}^n\sgn(i-k_1)\sgn(i-k_2)E\left[P(X_{k_1}>X_{l_i}\mid X_{l_i})P(X_{k_2}>X_{l_i}\mid X_{l_i})\right],\label{eq:taukendall-bootstrap-lemma-proofuse-2}\\
T_2 & =\frac{1}{(n-1)^2}\sum_{i=1}^n\sum_{k_1,k_2=1}^n\sgn(i-k_1)\sgn(i-k_2)P(X_{k_1}>X_j)P(X_{k_2}>X_i),\label{eq:taukendall-bootstrap-lemma-proofuse-3}\\
T_{3} & =\frac{1}{(n-1)^2}\sum_{i=1}^n\sum_{k_1,k_2=1}^n\sgn(i-k_1)\sgn(i-k_2)P(X_{k_1}>X_i)P(X_{k_2}>X_i).\label{eq:taukendall-bootstrap-lemma-proofuse-4}
\end{align}
We have
\begin{equation}
\sum_{i=1}^n\sum_{k_1,k_2=1}^n\sgn(i-k_1)\sgn(i-k_2)=\sum_{i=1}^n(2i-n-1)^2=\frac{1}{3}n(n-1)(n+1).\label{eq:taukendall-bootstrap-lemma-proofuse-5}
\end{equation}
It follows from (\ref{eq:taukendall-bootstrap-condition-2}), (\ref{eq:taukendall-bootstrap-lemma-proofuse-2}),
and (\ref{eq:taukendall-bootstrap-lemma-proofuse-5}) that
\begin{equation}
T_1=\frac{n(n-1)(n+1)}{3(n-1)^2}\{\eta^2+O(n^{-1/3})\}=\frac{n(n+1)}{3(n-1)}\eta^2+O(n^{2/3}).\label{eq:taukendall-bootstrap-lemma-proofuse-6}
\end{equation}
It follows from (\ref{eq:taukendall-bootstrap-condition-1}), (\ref{eq:taukendall-bootstrap-lemma-proofuse-3}),
(\ref{eq:taukendall-bootstrap-lemma-proofuse-4}), and (\ref{eq:taukendall-bootstrap-lemma-proofuse-5})
that
\begin{align}
T_2 & =\frac{n(n-1)(n+1)}{3(n-1)^2}\{\theta+O(n^{-1/6})\}^2=\frac{n(n+1)}{3(n-1)}\theta^2+O(n^{5/6}),\label{eq:taukendall-bootstrap-lemma-proofuse-7}\\
T_{3} & =\frac{n(n-1)(n+1)}{3(n-1)^2}\{\theta+O(n^{-1/6})\}^2=\frac{n(n+1)}{3(n-1)}\theta^2+O(n^{5/6}).\label{eq:taukendall-bootstrap-lemma-proofuse-8}
\end{align}
Combining (\ref{eq:taukendall-bootstrap-lemma-proofuse-1}) with (\ref{eq:taukendall-bootstrap-lemma-proofuse-6}),
(\ref{eq:taukendall-bootstrap-lemma-proofuse-7}), and (\ref{eq:taukendall-bootstrap-lemma-proofuse-8})
yields 
\begin{equation}
\sum_{i=1}^nE\left\{ h_{1,i}^{\ken}(X_{l_i})^2\right\} =T_1-2T_2+T_{3}=\frac{n(n+1)}{3(n-1)}(\eta^2-\theta^2)+O(n^{5/6}).\label{eq:taukendall-bootstrap-lemma-proofuse-9}
\end{equation}
In (\ref{eq:taukendall-bootstrap-lemma-proofuse-9}), letting $(l_1,\ldots,l_n)=(j,j,\ldots,j)$ yields (\ref{eq:taukendall-bootstrap-lemma-1}), and letting $(l_1,\ldots,l_n)=(j,j,\ldots,j)$ yields (\ref{eq:taukendall-bootstrap-lemma-2}). 

This completes the proof.
\end{proof}


\subsection{Proof of Lemma \ref{lem:tauap-bootstrap-lemma}}
\begin{proof}
Consider an arbitrary vector $(l_1,\ldots,l_n)$, with each $l_i\in[n]$.
It follows from (\ref{eq:tauap-h1i}) that 
\begin{align}
\sum_{i=1}^nE\{h_{1,i}^{\ap}(X_{l_i})^2\} & =\frac{1}{(n-1)^2}\sum_{i=1}^nE\Big[\sum_{k=1}^n\Big\{\frac{n\ind(k<i)}{i-1}-\frac{n\ind(k>i)}{k-1}\Big\}\{P(X_k>X_{l_i}\mid X_{l_i})-P(X_k>X_i)\}\Big]^2\nonumber  \\
 & =T_1-2T_2+T_{3},\label{eq:tauap-bootstrap-lemma-proofuse-1}
\end{align}
where
\begin{align}
T_1 & =\frac{n^2}{(n-1)^2}\sum_{i=1}^n\sum_{k_1,k_2=1}^n\gamma(i,k_1,k_2)E\left[P(X_{k_1}>X_{l_i}\mid X_{l_i})P(X_{k_2}>X_{l_i}\mid X_{l_i})\right],\label{eq:tauap-bootstrap-lemma-proofuse-2}\\
T_2 & =\frac{n^2}{(n-1)^2}\sum_{i=1}^n\sum_{k_1,k_2=1}^n\gamma(i,k_1,k_2)P(X_{k_1}>X_j)P(X_{k_2}>X_i),\label{eq:tauap-bootstrap-lemma-proofuse-3}\\
T_{3} & =\frac{n^2}{(n-1)^2}\sum_{i=1}^n\sum_{k_1,k_2=1}^n\gamma(i,k_1,k_2)P(X_{k_1}>X_i)P(X_{k_2}>X_i),\label{eq:tauap-bootstrap-lemma-proofuse-4}
\end{align}
and
\[
\gamma(i,k_1,k_2):=\Big\{\frac{\ind(k_1<i)}{i-1}-\frac{\ind(k_1>i)}{k_1-1}\Big\}\Big\{\frac{\ind(k_2<i)}{i-1}-\frac{\ind(k_2>i)}{k_2-1}\Big\}.
\]
By Lemma \ref{lem:sum-weight} and Lemma \ref{lem:bound-harmonic-sum}
we have
\begin{equation}
\sum_{i=1}^n\sum_{k_1,k_2=1}^n\gamma(i,k_1,k_2)=(n-1)+\varphi(n-1)=(n-1)+O(\log n).\label{eq:tauap-bootstrap-lemma-proofuse-5}
\end{equation}
It follows from (\ref{eq:tauap-bootstrap-condition-2}), (\ref{eq:tauap-bootstrap-lemma-proofuse-2}),
and (\ref{eq:tauap-bootstrap-lemma-proofuse-5}) that
\begin{equation}
T_1=\frac{n^2\{(n-1)+O(\log n)\}}{(n-1)^2}\{\eta^2+O(n^{-1/3}(\log n)^2)\}=\frac{n^2}{n-1}\eta^2+O\{n^{2/3}(\log n)^2\}.\label{eq:tauap-bootstrap-lemma-proofuse-6}
\end{equation}
It follows from (\ref{eq:tauap-bootstrap-condition-1}), (\ref{eq:tauap-bootstrap-lemma-proofuse-3}),
(\ref{eq:tauap-bootstrap-lemma-proofuse-4}), and (\ref{eq:tauap-bootstrap-lemma-proofuse-5})
that
\begin{align}
T_2 & =\frac{n^2\{(n-1)+O(\log n)\}}{(n-1)^2}\{\theta+O(n^{-1/6}\log n)\}^2=\frac{n^2}{n-1}\theta^2+O(n^{5/6}\log n),\label{eq:tauap-bootstrap-lemma-proofuse-7}\\
T_{3} & =\frac{n^2\{(n-1)+O(\log n)\}}{(n-1)^2}\{\theta+O(n^{-1/6}\log n)\}^2=\frac{n^2}{n-1}\theta^2+O(n^{5/6}\log n).\label{eq:tauap-bootstrap-lemma-proofuse-8}
\end{align}
Combining (\ref{eq:tauap-bootstrap-lemma-proofuse-1}) with (\ref{eq:tauap-bootstrap-lemma-proofuse-6}),
(\ref{eq:tauap-bootstrap-lemma-proofuse-7}), and (\ref{eq:tauap-bootstrap-lemma-proofuse-8})
yields 
\begin{equation}
\sum_{i=1}^nE\left\{ h_{1,i}^{\ap}(X_{l_i})^2\right\} =T_1-2T_2+T_{3}=\frac{n^2}{n-1}(\eta^2-\theta^2)+O(n^{5/6}\log n).\label{eq:taukendall-bootstrap-lemma-proofuse-9-1}
\end{equation}
In (\ref{eq:taukendall-bootstrap-lemma-proofuse-9-1}), letting $(l_1,\ldots,l_n)=(j,j,\ldots,j)$ yields (\ref{eq:taukendall-bootstrap-lemma-1}), and letting $(l_1,\ldots,l_n)=(j,j,\ldots,j)$ yields (\ref{eq:taukendall-bootstrap-lemma-2}). This completes the proof.
\end{proof}

\section{Auxiliary lemmas and proofs}
\label{sec:aux-lemmas}

In this Section, we state and prove (or give reference to) the auxiliary lemmas that are used in the proofs in earlier sections.


\begin{lemma}
\label{lem:moment-bound} There exists a constant $c_m$ which only
depends on $m$, such that the following results hold.
\begin{enumerate}
\item[(i)] For any $n$ and any $(i_1,\ldots,i_m)\in I_n^m$, 
\[
E\{h_{2;i_1,\ldots,i_m}(X_{i_1},\ldots,X_{i_m})^2\}\leq c_mE\{h(X_{i_1},\ldots,X_{i_m})^2\}.
\]

\item[(ii)] For any $n$, any $i\in[n]$, any $(i_1,\ldots,i_{m-1})\in I_{n-1}^{m-1}(-i)$,
and any $l\in[m]$, 
\[
E[\{f_{i_1,\ldots,i_{m-1}}^{(l)}(X_i)-\theta^{(l)}(i;i_1,\ldots,i_{m-1})\}^{4}]\leq c_mE\{h^{(l)}(X_i;X_{i_1},\ldots,X_{i_{m-1}})^{4}\}.
\]

\item[(iii)] For any $n$, any $(i_1,\ldots,i_m)\in I_n^m$, and any $j_1,\ldots,j_m\in[n]$,
\[
E\{h_{2;i_1,\ldots,i_m}(X_{j_1},\ldots,X_{j_m})^2\}\leq c_m\sup_{1\leq k_1,\ldots,k_m\leq n}E\{h(X_{k_1},\ldots,X_{k_m})^2\}.
\]

\end{enumerate}
\end{lemma}


\begin{lemma}[{\citealp[Section 1.3, Theorem 2]{lee1990u}}]
\label{lem:cov-conditional-expectation} Consider three random variables
$X,Y,Z$. Assume $Y$ is independent of $Z$ conditional on $X$.
Then for two measurable functions $f,g:\mathbb{R}^2\to\mathbb{R}$,
we have 
\begin{align*}
\cov\{f(X,Y),g(X,Z)\} & =\cov\Big[E\{f(X,Y)\mid X\},E\{g(X,Z)\mid X\}\Big].
\end{align*}

\end{lemma}


\begin{lemma}
\label{lem:var-Pconverge} Consider a sequence of random variables
$X_1,X_2,\ldots$, with $E(X_n)=0$ for all $X_n$. If $\var(X_n)\to0$,
then $X_n\stackrel{P}{\to}0$.
\end{lemma}


\begin{lemma}[Lyapunov's central limit theorem]
 \label{lem:lyapunov-clt} Let $X_1,X_2,\ldots$ be a sequence
of independent random variables and let $S_n=n^{-1}\sum_{i=1}^nX_i$.
If there exists $\delta>0$ such that 
\begin{align}
\lim_{n\to\infty}\frac{\sum_{i=1}^nE|X_i-E(X_i)|^{2+\delta}}{\Big\{\sum_{i=1}^nE\left|X_i-E(X_i)\right|^2\Big\}^{\frac{2+\delta}{2}}} & =0,\label{eq:lyapunov-clt-condition}
\end{align}
then 
\begin{align*}
\var(S_n)^{-1/2}\{S_n-E(S_n)\} & \stackrel{d}{\to}N(0,1).
\end{align*}

\end{lemma}


\begin{lemma}[{{\citealp[Theorem 2.6.1]{lehmann1999elements}}}]
 \label{lem:cdf-uniform-convergence} If a sequence of cumulative
distribution functions $H_n$ tends to a continuous cdf $H$, then
$H_n(x)$ converges to $H(x)$ uniformly in $x$.
\end{lemma}


\begin{lemma}[{{\citealp[Theorem 2.2]{mammen2012bootstrap}}}]
 \label{lem:bootstrap-noniid} Consider a sequence $Y_{n,1},\ldots,Y_{n,n}$
of independent random variables with distribution $P_{n,i}$. For
a function $g_n$ define $\hat{T}_n=n^{-1}\sum_{i=1}^ng_n(Y_{n,i})$.
Consider a bootstrap sample $Y_{n,1}^*,\ldots,Y_{n,n}^*$ and
define $\hat{T}_n^*=n^{-1}\sum_{i=1}^ng_n(Y_{n,i}^*)$.
Then for every sequence $t_n$ the following assertions are equivalent: 
\begin{enumerate}
\item[(i)] There exists $\sigma_n$ such that for every $\epsilon>0$ 
\begin{align}
\sup_{1\leq i\leq n}P\Big\{\Big|\frac{g_n(Y_{n,i})-t_n}{n\sigma_n}\Big|\geq\epsilon\Big\} & \to0,\label{eq:bootstrap-noniid-condition1}\\
\sum_{i=1}^n\Big(E\Big[\frac{g_n(Y_{n,i})-t_n}{n\sigma_n}1\Big\{\Big|\frac{g_n(Y_{n,i})-t_n}{n\sigma_n}\Big|\leq\epsilon\Big\}\Big]\Big)^2 & \to0,\label{eq:bootstrap-noniid-condition2}\\
\sup_{t\in\mathbb{R}}|P(\hat{T}_n-t_n\leq t)-\Phi(t)| & \to0.\label{eq:bootstrap-noniid-condition3}
\end{align}

\item[(ii)] Bootstrap works: 
\begin{align*}
\sup_{t\in\mathbb{R}}|P(\hat{T}_n^*-\hat{T}_n\leq t\mid Y_{n,1},\ldots,Y_{n,n})-P(\hat{T}_n-t_n\leq t)| & \stackrel{P}{\to}0.
\end{align*}

\end{enumerate}
\end{lemma}


\begin{lemma}[{{\citealp[Theorem 1.8 C]{serfling2009approximation}}}]
 \label{lem:WLLN} Let $X_1,X_2,\ldots$ be uncorrelated with
means $\mu_1,\mu_2,\ldots$ and variances $\sigma_1^2,\sigma_2^2,\ldots$.
If $\sum_{i=1}^n\sigma_i^2=o(n^{-2})$, $n\to\infty$, then
\[
\frac{1}{n}\sum_{i=1}^nX_i-\frac{1}{n}\sum_{i=1}^n\mu_i\stackrel{P}{\to}0.
\]

\end{lemma}


\begin{lemma}[Bound on the partial sum of harmonic series]
 \label{lem:bound-harmonic-sum} Denote $\varphi(n)=\sum_{k=1}^nk^{-1}$.
Then for any two integers $m,n$ such that $1\leq m\leq n$, 
\begin{align}
\log\frac{n+1}{m+1} & \leq\varphi(n)-\varphi(m)\leq\log\frac{n}{m},\label{eq:bound-harmonic-sum-1}\\
\log(n+1) & \leq\varphi(n)\leq1+\log n.\label{eq:bound-harmonic-sum-2}
\end{align}

\end{lemma}


\begin{lemma}
\label{lem:bound-sum-power} For any two positive real numbers $a,b$
and real number $p>0$, we have 
\begin{align*}
(a+b)^p & \leq\xi(p)(a^p+b^p),
\end{align*}
where 
\begin{align*}
\xi(p) & =\begin{cases}
2^{p-1} & \mbox{if }p\geq1,\\
1 & \mbox{if }0<p<1.
\end{cases}
\end{align*}

\end{lemma}


\begin{lemma}
\label{lem:sum-weight}We have
\begin{equation}
\sum_{i=1}^n\sum_{j=1}^n\sum_{k=1}^n\Big\{\frac{\ind(j<i)}{i-1}-\frac{\ind(j>i)}{j-1}\Big\}\Big\{\frac{\ind(k<i)}{i-1}-\frac{\ind(k>i)}{k-1}\Big\}=(n-1)+\varphi(n-1),\label{eq:sum-weight}
\end{equation}
where we define $0/0:=0$ and $\varphi(n):=\sum_{k=1}^nk^{-1}$.
\end{lemma}

\begin{lemma}
\label{lem:bound-normal-cdf} Define $\Phi^{c}(x)=\frac{1}{\sqrt{2\pi}}\int_{x}^{\infty}\exp(-\frac{t^2}{2})dt$ to be the complement distribution function for the standard Gaussian. We
have the following bounds for $\Phi^{c}(x)$:
\begin{eqnarray*}
\frac{1}{\sqrt{2\pi}}\left(\frac{1}{x}-\frac{1}{x^{3}}\right)\exp(-\frac{x^2}{2}) & \leq\Phi^{c}(x)\leq & \frac{1}{\sqrt{2\pi}}\frac{1}{x}\exp(-\frac{x^2}{2}),\qquad\mbox{if }x>0,\\
1+\frac{1}{\sqrt{2\pi}}\frac{1}{x}\exp(-\frac{x^2}{2}) & \leq\Phi^{c}(x)\leq & 1+\frac{1}{\sqrt{2\pi}}\left(\frac{1}{x}-\frac{1}{x^{3}}\right)\exp(-\frac{x^2}{2}),\qquad\mbox{if }x<0.
\end{eqnarray*}

\end{lemma}


\subsection{Proof of auxiliary lemmas}


\begin{proof}[Proof of Lemma \ref{lem:moment-bound}]
Define $\bi=(i_1,\ldots,i_m)$, $\Xbi=(X_{i_1},\ldots,X_{i_m})$,
and $\bi_{-m}=(i_1,\ldots,i_{m-1})$.
\begin{enumerate}
\item[(i)]  By the definition of $h_{2;\bi}(\cdot)$ in (\ref{eq:def-h2}) we
have 
\begin{equation}
E\{h_{2;}(\Xbi)^2\}\leq2^{m+2}\Big[E\{h(\Xbi)^2\}+\sum_{l=1}^mE\{f_{\bi\backslash i_l}^{(l)}(X_{i_l})^2\}+(m-1)^2\theta^2(\bi)\Big].\label{eq:aux-proofuse-1}
\end{equation}
Jensen's inequality and the law of iterated expectation yield
\begin{equation}
E\{f_{\bi\backslash i_l}^{(l)}(X_{i_l})^2\}=E_{i_l}[E_{\bi\backslash i_l}\{h^{(l)}(X_{i_l};Y_1,\ldots Y_{m-1})\mid X_{i_l}\}^2]\leq E\{h(\Xbi)^2\}\label{eq:aux-proofuse-2}
\end{equation}
and 
\begin{equation}
\theta^2(\bi)\leq E\{h(\Xbi)^2\}.\label{eq:aux-proofuse-3}
\end{equation}
Equations (\ref{eq:aux-proofuse-1}), (\ref{eq:aux-proofuse-2}), and
(\ref{eq:aux-proofuse-3}) imply
\[
E\{h_{2;\bi}(\Xbi)^2\}\leq2^{m+2}\{1+m+(m-1)^2\}E\{h(\Xbi)^2\}.
\]
This proves (i).

\item[(ii)] We have 
\begin{equation}
E[\{f_{\bi_{-m}}^{(l)}(X_i)-\theta^{(l)}(i;\bi_{-m})\}^{4}]\leq2^{4}[E\{f_{\bi_{-m}}^{(l)}(X_i)^{4}\}+\theta^{(l)}(i;\bi_{-m})^{4}]\label{eq:aux-proofuse-4}
\end{equation}
By the definition of $f_{\bi_{-m}}^{(l)}(\cdot)$ in (\ref{eq:def-fl})
and Jensen's inequality we have
\begin{equation}
E\{f_{\bi_{-m}}^{(l)}(X_i)^{4}\}\leq E_i[E_{\bi_{-m}}\{h^{(l)}(X_i;Y_1,\ldots Y_{m-1})^{4}\mid X_i\}]=E\{h^{(l)}(X_i;X_{\bi_{-m}})^{4}\}.\label{eq:aux-proofuse-5}
\end{equation}
Jensen's inequality also implies that
\begin{equation}
\theta^{(l)}(i;\bi_{-m})^{4}=\{Eh^{(l)}(X_i;X_{\bi_{-m}}) \}^{4}\leq E\{h^{(l)}(X_i;X_{\bi_{-m}})^{4}\}.\label{eq:aux-proofuse-6}
\end{equation}
Combining (\ref{eq:aux-proofuse-4}) with (\ref{eq:aux-proofuse-5})
and (\ref{eq:aux-proofuse-6}) yields
\[
E[\{f_{\bi_{-m}}^{(l)}(X_i)-\theta^{(l)}(i;\bi_{-m})\}^{4}]\leq2^{4}E\{h^{(l)}(X_i;X_{\bi_{-m}})^{4}\}.
\]
This proves (ii).

\item[(iii)]  Consider $\bj:=(j_1,\ldots,j_m)$ with each $j_l\in[m]$.
Define $\Xbj:=(X_{j_1},\ldots,X_{j_m})$. By the definition of
$h_{2;\bi}(\cdot)$ in (\ref{eq:def-h2}) we have
\begin{equation}
E\{h_{2;\bi}(\Xbj)^2\}\leq2^{m+2}\Big[E\{h(\Xbj)^2\}+\sum_{l=1}^mE\{f_{\bi\backslash i_l}^{(l)}(X_{j_l})^2\}+(m-1)^2\theta^2(\bi)\Big].\label{eq:aux-proofuse-7}
\end{equation}
By the definition of $f_{\bi_{-m}}^{(l)}(\cdot)$ in (\ref{eq:def-fl}) and Jensen's inequality we have
\begin{equation}
E\{f_{\bi_{-m}}^{(l)}(X_{j_l})^2\}\leq E_{j_l}[E_{\bi_{-m}}\{h^{(l)}(X_{j_l};Y_1,\ldots Y_{m-1})^2\mid X_{j_l}\}]=E\{h^{(l)}(X_{j_l};X_{\bi_{-m}})^2\}.\label{eq:aux-proofuse-8}
\end{equation}
Combining (\ref{eq:aux-proofuse-7}), (\ref{eq:aux-proofuse-8}),
and (\ref{eq:aux-proofuse-3}) yields
\begin{align*}
E\{h_{2;\bi}(\Xbj)^2\} & \leq2^{m+2}\Big[E\{h(\Xbj)^2\}+\sum_{l=1}^mE\{h^{(l)}(X_{j_l};X_{\bi_{-m}})^2\}+(m-1)^2E\{h(\Xbi)^2\}\Big]\\
 & \leq2^{m+2}m^2\sup_{1\leq k_1,\ldots,k_m\leq n}E\{h(X_{k_1},\ldots,X_{k_m})^2\}.
\end{align*}
This proves (iii).
\end{enumerate}
The proof is thus finished.
\end{proof}




\begin{proof}[Proof of Lemma \ref{lem:bound-harmonic-sum}]
We have $\varphi(n)-\varphi(m)=\sum_{k=m+1}^nk^{-1}$. By integral
bound, we have 
\begin{align*}
\log\frac{n+1}{m+1} & =\int_{m+1}^{n+1}\frac{1}{x}dx\leq\sum_{k=m+1}^n\frac{1}{k}\leq\int_m^n\frac{1}{x}dx=\log\frac{n}{m},
\end{align*}
which yields (\ref{eq:bound-harmonic-sum-1}). We also have 
\begin{align*}
\log(n+1) & \leq\int_1^{n+1}\frac{1}{x}dx\leq\sum_{k=1}^n\frac{1}{k}\leq1+\int_1^n\frac{1}{x}dx\leq1+\log n,
\end{align*}
which yields (\ref{eq:bound-harmonic-sum-2}). The proof is thus finished.
\end{proof}


\begin{proof}[Proof of Lemma \ref{lem:sum-weight}]
By algebra we have
\begin{equation}
\sum_{i=1}^n\sum_{j=1}^n\sum_{k=1}^n\Big\{\frac{\ind(j<i)}{i-1}-\frac{\ind(j>i)}{j-1}\Big\}\Big\{\frac{\ind(k<i)}{i-1}-\frac{\ind(k>i)}{k_2-1}\Big\}=T_1-T_2-T_{3}+T_{4},\label{eq:sum-weight-proofuse-1}
\end{equation}
where
\begin{align*}
T_1 & =\sum_{i=1}^n\sum_{j=1}^n\sum_{k=1}^n\frac{\ind(j<i)}{i-1}\cdot\frac{\ind(k<i)}{i-1},~~T_2=\sum_{i=1}^n\sum_{j=1}^n\sum_{k=1}^n\frac{\ind(j<i)}{i-1}\cdot\frac{\ind(k>i)}{k-1},\\
T_{3} & =\sum_{i=1}^n\sum_{j=1}^n\sum_{k=1}^n\frac{\ind(j>i)}{j-1}\cdot\frac{\ind(k<i)}{i-1},~~T_{4}=\sum_{i=1}^n\sum_{j=1}^n\sum_{k=1}^n\frac{\ind(j>i)}{j-1}\cdot\frac{\ind(k>i)}{k-1}.
\end{align*}
For $T_1$ we have
\begin{equation}
T_1=\sum_{i=2}^n\sum_{j=1}^{i-1}\sum_{k=1}^{i-1}\frac{1}{(i-1)^2}=n-1.\label{eq:sum-weight-proofuse-2}
\end{equation}
For $T_2$ we have
\begin{equation}
T_2=\sum_{i=2}^{n-1}\sum_{j=1}^{i-1}\sum_{k=i+1}^n\frac{1}{i-1}\cdot\frac{1}{k-1}=\sum_{k=3}^n\sum_{i=2}^{k-1}\frac{1}{k-1}=\sum_{k=3}^n\Big(1-\frac{1}{k-1}\Big)=(n-1)-\varphi(n-1).\label{eq:sum-weight-proofuse-3}
\end{equation}
By symmetry $T_2=T_{3}$, so
\begin{equation}
T_{3}=(n-1)-\varphi(n-1).\label{eq:sum-weight-proofuse-4}
\end{equation}
For $T_{4}$ we have
\begin{equation}
T_{4}=\sum_{i=1}^{n-1}\sum_{j=i+1}^n\sum_{k=i+1}^n\frac{1}{j-1}\cdot\frac{1}{k-1}=\sum_{j=2}^n\sum_{k=j+1}^n\sum_{i=1}^{j-1}\frac{1}{j-1}\cdot\frac{1}{k-1}+\sum_{j=2}^n\sum_{k=2}^{j}\sum_{i=1}^{k-1}\frac{1}{j-1}\cdot\frac{1}{k-1}.\label{eq:sum-weight-proofuse-5}
\end{equation}
Note that
\begin{align}
\sum_{j=2}^n\sum_{k=j+1}^n\sum_{i=1}^{j-1}\frac{1}{j-1}\cdot\frac{1}{k-1} & =\sum_{j=2}^n\sum_{k=j+1}^n\frac{1}{k-1}=\sum_{k=3}^n\sum_{j=2}^{k-1}\frac{1}{k-1}=(n-1)-\varphi(n-1)\label{eq:sum-weight-proofuse-6}\\
~~{\rm and}~~\sum_{j=2}^n\sum_{k=2}^{j}\sum_{i=1}^{k-1}\frac{1}{j-1}\cdot\frac{1}{k-1} & =\sum_{j=2}^n\sum_{k=2}^{j}\frac{1}{j-1}=\sum_{j=2}^n1=n-1.\label{eq:sum-weight-proofuse-7}
\end{align}
Combining (\ref{eq:sum-weight-proofuse-5}) with (\ref{eq:sum-weight-proofuse-6})
and (\ref{eq:sum-weight-proofuse-7}) yields
\begin{equation}
T_{4}=2(n-1)-\varphi(n-1).\label{eq:sum-weight-proofuse-8}
\end{equation}
Equation (\ref{eq:sum-weight}) follows from (\ref{eq:sum-weight-proofuse-1}),
(\ref{eq:sum-weight-proofuse-2}), (\ref{eq:sum-weight-proofuse-3}),
and (\ref{eq:sum-weight-proofuse-8}). 

This completes the proof.
\end{proof}

\clearpage

\bibliographystyle{apalike}
\bibliography{Ustat}

\end{document}